\documentclass{gtart}

\newcommand{\pathtotrunk}{./}

\newcommand{\forGTART}[1]{#1}
\newcommand{\forQT}[1]{}
\newcommand{\forOUP}[1]{}

\usepackage{ifthen}

\usepackage{amsmath,amssymb,amsfonts}
\usepackage{ifpdf}
\usepackage{enumerate}
\usepackage{ulem}
\usepackage{leftidx}

\ifpdf
\usepackage[pdftex,all,color]{xy}
\else
\usepackage[all,color]{xy}
\fi

\SelectTips{cm}{}

\ifpdf
\newcommand{\pathtodiagrams}{\pathtotrunk diagrams/pdf/}
\else
\newcommand{\pathtodiagrams}{\pathtotrunk diagrams/eps/}
\fi

\newcommand{\baselinetrick}[1]{\hspace{-3pt}\begin{array}{c}%
  \raisebox{-2.5pt}{#1}%
\end{array}\hspace{-3pt}}

\newcommand{\mathfig}[2]{\ensuremath{\hspace{-3pt}\begin{array}{c}%
  \raisebox{-2.5pt}{\includegraphics[width=#1\textwidth]{\pathtodiagrams #2}}%
\end{array}\hspace{-3pt}}}

\newcommand{\arxiv}[1]{Preprint \href{http://arxiv.org/abs/#1}{\tt arXiv:\nolinkurl{#1}}}
\newcommand{\doi}[1]{\href{http://dx.doi.org/#1}{{\tt DOI:#1}}}

\newcommand{\googlebooks}[1]{(preview at \href{http://books.google.com/books?id=#1}{google books})}

\def\RCS$#1: #2 ${\expandafter\def\csname RCS#1\endcsname{#2}}
\RCS$Revision: 909 $ \RCS$Date: 2014-08-07 12:23:12 +1000 (Thu, 07 Aug 2014) $

\theoremstyle{plain}
\newtheorem{prop}{Proposition}[section]

\newtheorem{thm}[prop]{Theorem}
\newtheorem{ex}[prop]{Example}
\newtheorem*{thm*}{Theorem}

\newtheorem{lem}[prop]{Lemma}
\newtheorem{cor}[prop]{Corollary}
\newtheorem*{cor*}{Corollary}

\newtheorem{defn}[prop]{Definition}         
\newtheorem*{defn*}{Definition}             

\newenvironment{rem}{\noindent\textsl{Remark.}}{}  
\numberwithin{equation}{section}

\newcounter{comment}
\newcommand{\noop}[1]{}

\def\clap#1{\hbox to 0pt{\hss#1\hss}}

\newcommand{\code}[1]{{\tt #1}}
\newcommand{\package}[1]{\code{#1`}}
\newcommand{\MMA}{Mathematica{}}

\newenvironment{mma}{}{}

\newcounter{lineNumber}
\newcommand{\inN}{\stepcounter{lineNumber}\arabic{lineNumber}}
\newcommand{\outN}{\arabic{lineNumber}}

\newenvironment{inm}{%
 \vskip 0pt \par\noindent%
  {\footnotesize\color{blue}In[\inN]$:=$}\tt%
  }{}
\newenvironment{outm}{%
 \vskip -6pt \par\noindent\parindent=12pt%
  {\footnotesize\color{blue}Out[\outN]$=$}\tt%
  }{}
\newenvironment{printm}%
    {\begin{quotation}\noindent}%
    {\end{quotation}}%

\newcommand{\Natural}{\mathbb N}
\newcommand{\Integer}{\mathbb Z}

\newcommand{\Complex}{\mathbb C}

\newcommand{\cC}{\mathcal{C}}
\newcommand{\cD}{\mathcal{D}}
\newcommand{\cE}{\mathcal{E}}
\newcommand{\cF}{\mathcal{F}}
\newcommand{\cG}{\mathcal{G}}

\newcommand{\cI}{\mathcal{I}}

\newcommand{\cM}{\mathcal{M}}

\newcommand{\cO}{\mathcal{O}}

\newcommand{\cR}{\mathcal{R}}

\newcommand{\cU}{\mathcal{U}}
\newcommand{\cV}{\mathcal{V}}

\newcommand{\id}{\boldsymbol{1}}
\renewcommand{\imath}{\mathfrak{i}}
\renewcommand{\jmath}{\mathfrak{j}}

\newcommand{\iso}{\cong}

\newcommand{\ceil}[1]{\left\lceil#1\right\rceil}

\newcommand{\restrict}[2]{#1{}_{\mid #2}{}}

\newcommand{\J}[2]{\mathcal{J}_{#1,#2}}
\newcommand{\Jf}[2]{\tilde{\mathcal{J}}_{#1,#2}}
\newcommand{\HOMFLY}{\operatorname{HOMFLYPT}}
\newcommand{\Kauffman}{\operatorname{Kauffman}}
\newcommand{\Dubrovnik}{\operatorname{Dubrovnik}}
\newcommand{\Rep}{\operatorname{Rep}}
\newcommand{\vRep}{\Rep^{\text{vector}}}
\newcommand{\uRep}{\Rep^{\text{uni}}}
\newcommand{\rRep}{\Rep^{\text{root}}}

\newcommand{\SO}[1]{\mathfrak{so}(#1)}
\newcommand{\SL}[1]{\mathfrak{sl}(#1)}
\newcommand{\SP}[1]{\mathfrak{sp}(#1)}

\newcommand{\overcrossing}{\begin{tikzpicture}[baseline]
\node (x) at (0,0){};
    \draw (x.45)-- (.5,.5);
    \draw (x.135) -- (-.5,.5);
    \draw (x.315) -- (.5,-.5);
    \draw (x.45) -- (-.5,-.5);
\end{tikzpicture}}
\newcommand{\undercrossing}{\begin{tikzpicture}[baseline,xscale=-1]
    \draw (-.5,-.5)-- (.5,.5);
    \draw (-.1,.1) -- (-.5,.5);
    \draw (.1,-.1) -- (.5,-.5);
\end{tikzpicture}
}
\newcommand{\identity}{\begin{tikzpicture}[baseline]
    \draw (1.5,.5) .. controls (2,0) .. (1.5,-.5);
    \draw (2.5,.5) .. controls (2,0) .. (2.5,-.5);
\end{tikzpicture}}
\newcommand{\cupcap}{\begin{tikzpicture}[baseline]
    \draw (3.5,.5) .. controls (4,0) .. (4.5,.5);
    \draw (3.5,-.5) .. controls (4,0) .. (4.5,-.5);
\end{tikzpicture}}
\newcommand{\Oovercrossing}{\begin{tikzpicture}[baseline]
\node (x) at (0,0){};
    \draw[->] (x.45)-- (.5,.5);
    \draw[->] (x.135) -- (-.5,.5);
    \draw (x.315) -- (.5,-.5);
    \draw (x.45) -- (-.5,-.5);
\end{tikzpicture}}
\newcommand{\Oundercrossing}{\begin{tikzpicture}[baseline,xscale=-1]
    \draw[->] (-.5,-.5)-- (.5,.5);
    \draw[->] (-.1,.1) -- (-.5,.5);
    \draw (.1,-.1) -- (.5,-.5);
\end{tikzpicture}
}
\newcommand{\Oidentity}{\begin{tikzpicture}[baseline]
    \draw[<-] (1.5,.5) .. controls (2,0) .. (1.5,-.5);
    \draw[<-] (2.5,.5) .. controls (2,0) .. (2.5,-.5);
\end{tikzpicture}}

\newcommand{\positivetwist}{\begin{tikzpicture}[baseline=.5cm, rounded corners = 2mm]
\draw (0,0)--(1.1,1.1) -- (1.1,-0.1)--(.6,.4);
\draw (.4,.6)--(0,1);
\end{tikzpicture}}
\newcommand{\negativetwist}{\begin{tikzpicture}[baseline=-.5cm, rounded corners = 2mm,yscale=-1]
\draw (0,0)--(1.1,1.1) -- (1.1,-0.1)--(.6,.4);
\draw (.4,.6)--(0,1);
\end{tikzpicture}}
\newcommand{\verticalstrand}{\begin{tikzpicture}[baseline = .5 cm,rounded corners = 2mm]
    \draw (0,0)--(.5,.5)--(0,1);
\end{tikzpicture}}
\newcommand{\Opositivetwist}{\begin{tikzpicture}[baseline=.5cm, rounded corners = 2mm]
\draw (0,0)--(1.1,1.1) -- (1.1,-0.1)--(.6,.4);
\draw[->] (.4,.6)--(0,1);
\end{tikzpicture}}
\newcommand{\Onegativetwist}{\begin{tikzpicture}[baseline=-.5cm, rounded corners = 2mm,yscale=-1]
\draw (0,1)--(.4,.6);
\draw[->] (.6,.4)--(1.1,-0.1)--(1.1,1.1)--(0,0);
\end{tikzpicture}}
\newcommand{\Overticalstrand}{\begin{tikzpicture}[baseline = .5 cm,rounded corners = 2mm]
    \draw[->] (0,0)--(.5,.5)--(0,1);
\end{tikzpicture}}

\newcommand{\qiq}[2]{[#1]_{#2}}
\newcommand{\qi}[1]{\qiq{#1}{q}}

\newcommand{\directSum}{\oplus}
\newcommand{\DirectSum}{\bigoplus}
\newcommand{\tensor}{\otimes}

\newcommand{\Inv}[2]{\operatorname{Inv}_{#1}\left(#2\right)}
\newcommand{\Hom}[3]{\operatorname{Hom}_{#1}\left(#2,#3\right)}
\newcommand{\End}[1]{\operatorname{End}\left(#1\right)}
\newcommand{\Reigenvalue}[2]{R_{#1 \subset #2 \tensor #2}}

\newcommand{\pa}{\mathcal{PA}(S)}
\newcommand{\TL}{\mathcal{TL}}
\newcommand{\JW}[1]{f^{(#1)}}
\newcommand{\tr}[1]{\text{tr}(#1)}

\newcommand{\LL}{\mathcal{L}}
\newcommand{\Lhat}{\hat{\mathcal{L}}}
\newcommand{\writhe}{\operatorname{writhe}}

\usepackage[pdftex,plainpages=false,hypertexnames=false,pdfpagelabels]{hyperref}

\usepackage{tikz}
\usetikzlibrary{shapes}
\usetikzlibrary{backgrounds}

\usepackage{microtype}

\begin{document}


\title{\hspace{-0.95cm}Knot polynomial identities and quantum group coincidences}

\newcommand{\Primaryclass}[1]{\primaryclass{#1}}
\newcommand{\Secondaryclass}[1]{\secondaryclass{#1}}
\newcommand{\Keywords}[1]{\keywords{#1}}
\newcommand{\Author}[2]{\author{#1}\email{#2}}

\author{Scott~Morrison}\email{scott@tqft.net}
\author{Emily~Peters}\email{eep@euclid.unh.edu}
\author{Noah~Snyder}\email{nsnyder@math.columbia.edu}

\address{%
\rm URLs:\stdspace \tt \url{http://tqft.net/}\ \
\url{http://euclid.unh.edu/~eep}\\\rm and \tt
\url{http://math.columbia.edu/~nsnyder}}

\Primaryclass{
	18D10 
}
\Secondaryclass{
	57M27 
	17B10 
	81R05 
	57R56 
}
\Keywords{
  Planar Algebras, Quantum Groups, Fusion Categories, Knot Theory, Link Invariants
}

\begin{abstract}
We construct link invariants using the $\mathcal{D}_{2n}$ subfactor planar algebras, and use these to prove new identities relating certain specializations of colored Jones polynomials to specializations of other quantum knot polynomials.  These identities can also be explained by coincidences between small modular categories involving the even parts of the $\mathcal{D}_{2n}$ planar algebras. We discuss the origins of these coincidences, explaining the role of $SO$ level-rank duality, Kirby-Melvin symmetry, and properties of small Dynkin diagrams.  One of these coincidences involves $G_2$ and does not appear to be related to level-rank duality.
\end{abstract}

\maketitle

\setcounter{tocdepth}{2}

\section{Introduction and background}

The goal of this paper is to construct knot and link invariants from the $\mathcal{D}_{2n}$ subfactor planar algebras, and to use these invariants to prove new identities between quantum group knot polynomials. These identities relate certain specializations of colored Jones polynomials to specializations of other knot polynomials.  In particular we prove that for any knot $K$ (but
not for a link!),
\begin{align}
\restrict{\J{\SL{2}}{(2)}(K)}{q=\exp(\frac{2\pi i}{12})} & = 2 \J{\mathcal{D}_4}{P}(K)  \notag \\
& = 2, \tag{Theorem \ref{thm:identities-2}}\\
\restrict{\J{\SL{2}}{(4)}(K)}{q=\exp(\frac{2\pi i}{20})} & = 2 \J{\mathcal{D}_6}{P}(K) \notag \\ & = 2 \restrict{\J{\SL{2}}{(1)}(K)}{q=\exp(- \frac{2\pi i}{10})}, \tag{Theorem \ref{thm:identities-3}}\\
\restrict{\J{\SL{2}}{(6)}(K)}{q=\exp(\frac{2\pi i}{28})} & = 2 \J{\mathcal{D}_8}{P}(K) \notag \\
                & = 2 \HOMFLY(K)(\exp(2\pi i\frac{5}{7}), \exp(- \frac{2\pi i}{14}) - \exp(\frac{2\pi i}{14})), \tag{Theorem \ref{thm:identities-4}} \displaybreak[1] \\
\restrict{\J{\SL{2}}{(8)}(K)}{q=\exp(\frac{2\pi i}{36})} & = 2 \J{\mathcal{D}_{10}}{P}(K) \notag \\
        & = 2  \Kauffman(K) (\exp(2 \pi i \frac{31}{36}), \exp(2 \pi i \frac{25}{36}) + \exp(2 \pi i \frac{11}{36}))\notag  \\
                & = 2 \restrict{\Kauffman(K) (-iq^7,i(q-q^{-1}))}{q=-\exp (\frac{-2 \pi i}{18})}  \tag{Theorem \ref{thm:identities-5}}  \\
\intertext{and}                
\restrict{\J{\SL{2}}{(12)}(K)}{q=\exp(\frac{2\pi i}{52})} & = 2 \J{\mathcal{D}_{14}}{P}(K) \notag \\
& = 2\restrict{\J{G_2}{V_{(10)}}(K)}{q=\exp{2\pi i \frac{23}{26}}}, \tag{Theorem \ref{thm:G2-links}}
\end{align}
where $\J{\SL{2}}{(k)}$ denotes the $k^{\text{th}}$ colored Jones polynomial, $\J{G_2}{V_{(1,0)}}$ denotes the knot invariant associated to the $7$-dimensional representation of $G_2$ and $\J{\mathcal{D}_{2n}}{P}$ is the $\mathcal{D}_{2n}$ link invariant for which we give a skein-theoretic construction.\footnote{Beware, the $\mathcal{D}_{2n}$ planar algebra is  not related to the lie algebra $\mathfrak{so}_{4n}$ with Dynkin diagram $D_{2n}$ but is instead a quantum subgroup of $U_q(\mathfrak{su}_2)$.}    (For our conventions for these polynomials, in particular their normalizations, see Section \ref{conventions}.)

These formulas should appear somewhat mysterious, and much of this paper is concerned with discovering the explanations for them. It turns out that each of these knot invariant identities comes from a coincidence of small modular categories involving the even part of one of the $\mathcal{D}_{2n}$. Just as families of finite groups have coincidences for small values (for example, the isomorphism between the finite groups $\text{Alt}_5$ and $\mathbf{PSL}_2(\mathbb{F}_5)$ or the outer automorphism of $S_6$), modular categories also have small coincidences.  Explicitly, we prove the following coincidences, where $\frac{1}{2} \mathcal{D}_{2n}$ denotes the even part of $\mathcal{D}_{2n}$ (the first of these coincidences is well-known).

\begin{itemize}
\item $\frac{1}{2} \mathcal{D}_{4} \cong \Rep{\mathbb{Z}/3}$, sending $P$ to $\chi_{\exp({\frac{2 \pi i}{3}})}$ (but an unusual braiding on $\Rep{\mathbb{Z}/3}$!).
\item $\frac{1}{2} \mathcal{D}_{6} \cong \uRep{U_{s = \exp({ 2 \pi i \frac{7}{10}})}(\mathfrak{sl}_2 \oplus \mathfrak{sl}_2)}^{modularize}$, sending $P$ to  $V_{1} \boxtimes V_{0}$. (See Theorem \ref{thm:D6-coincidence}, and \S \ref{sec:ribbonfunctors}.)
\item $\frac{1}{2} \mathcal{D}_{8} \cong \uRep{U_{s=\exp({2 \pi i \frac{5}{14} })}(\mathfrak{sl}_4)}^{modularize}$, sending $P$ to $V_{(100)}$. (See Theorem \ref{thm:D8-coincidence}.)
\item $\frac{1}{2} \mathcal{D}_{10}$ has an order $3$ automorphism: $\baselinetrick{\xymatrix@C-8mm@R-3mm{P \ar@{|->}@/^/[rr] & & Q \ar@{|->}@/^/[dl] \\ & f^{(2)} \ar@{|->}@/^/[ul]&}}$. (See Theorem \ref{thm:D10-coincidence}.)
\item $\frac{1}{2} \mathcal{D}_{14} \cong \Rep{U_{\exp({2 \pi i\frac{23}{26}})}(\mathfrak{g}_2)}$ sending $P$ to $V_{(10)}$. (See Theorem \ref{thm:G2}.)
\end{itemize}

To interpret the right hand sides of these equivalences, recall that the definition of the braiding (although not of the quantum group itself) depends on a choice of $s = q^{\frac{1}{L}}$, where $L$ is the index of the root lattice in the weight lattice.  Furthermore, the ribbon structure on the category of representations depends on a choice of a certain square root. In particular, besides the usual pivotal structure there's also another pivotal structure, which is called ``unimodal'' by Turaev \cite{MR1292673} and discussed in \S \ref{sec:quantumgroups} below.  By ``modularize" we mean take the modular quotient described by Brugui\`eres \cite{MR1741269} and M{\"{u}}ger \cite{MR1749250} and recalled in \S \ref{sec:modularization} below.

We first prove the knot polynomial identities directly, and later we give more conceptual explanations of the coincidences using
\begin{itemize}
\item coincidences of small Dynkin diagrams,
\item level-rank duality, and 
\item Kirby-Melvin symmetry.
\end{itemize}
These conceptual explanations do not suffice for the equivalence between the even part of $\mathcal{D}_{14}$ and $\Rep{U_{\exp({2 \pi i\frac{\ell}{26}})}(\mathfrak{g}_2)}$, $\ell=-3$ or $10$, which deserves further exploration.  Nonetheless we can prove this equivalence using direct methods (see Section \ref{sec:recognizeD}), and it answers a conjecture of Rowell's  \cite{MR2414692} concerning the unitarity of  $(G_2)_{\frac{1}{3}}$.

We illustrate each of these coincidences of tensor categories with diagrams of the appropriate quantum group Weyl alcoves; see in particular Figures \ref{fig:SO4}, \ref{fig:SO6}, \ref{fig:SO8} and \ref{fig:SO8-4fold} at the end of the paper. An ambitious reader might jump to those diagrams and try to understand them and their captions, then work back through the paper to pick up needed background or details.


In more detail the outline of the paper is as follows.  In the background section we recall some important facts about planar algebras, tensor categories, quantum groups, knot invariants and their relationships.  We fix our conventions for knot polynomials.  We also briefly recall several key concepts like semisimplification, deequivariantization, and modularization.

In Section \ref{sec:invariants} we use the skein theoretic description of $\mathcal{D}_{2n}$ to show that the Kauffman bracket gives a braiding up to sign for $\mathcal{D}_{2n}$, and in particular gives a braiding on the even part (this was already known; see for example the description of $\mathrm{Rep}^0 A$ in \cite[p. 33]{MR1936496}).  Using this, we define and discuss some new invariants of links which are the $\mathcal{D}_{2n}$ analogues of the colored Jones polynomials.  We also define some refinements of these invariants for multi-component links.

In Section \ref{sec:identities}, we discuss some identities relating the $\mathcal{D}_{2n}$ link invariants at small values of $n$ to other link polynomials.  This allows us to prove the above identities between quantum group invariants of knots.  The main technique is to apply the following schema to an object $X$ in a ribbon category (where $A$ and $B$ always denote simple objects),
\begin{itemize}
\item if $X \otimes X=A$ then the knot invariant coming from $X$ is trivial,
\item if $X \otimes X = 1\oplus A$ then the knot invariant coming from $X$ is a specialization of the Jones polynomial,
\item if $X \otimes X = A\oplus B$ then the knot invariant coming from $X$ is a specialization of the HOMFLYPT polynomial,
\item if $X \otimes X = 1\oplus A \oplus B$ then the knot invariant coming from $X$ is a specialization of the Kauffman polynomial or the Dubrovnik polynomial,
\end{itemize}
Furthermore we give formulas that identify which specialization occurs.  This technique is due to Kauffman, Kazhdan, Tuba, Wenzl, and others \cite{MR2132671, MR1237835, MR958895}, and is well-known to experts. We also use a result of Kuperberg's which gives a similar condition for specializations of the $G_2$ knot polynomial.

In Section \ref{sec:coincidences}, we reprove the results of the previous section using coincidences of Dynkin diagrams, generalized Kirby-Melvin symmetry, and level-rank duality.  In particular, we give a new simple proof of Kirby-Melvin symmetry which applies very generally, and we use a result of Wenzl and Tuba to strengthen Beliakova and Blanchet's statement of $SO$ level-rank duality.

We'd like to thank Stephen Bigelow, Vaughan Jones, Greg Kuperberg, Nicolai Resh\-e\-ti\-khin, Kevin Walker, and Feng Xu for helpful conversations.  During our work on this paper, Scott Morrison was at Microsoft Station Q, Emily Peters was supported in part by NSF grant DMS0401734 and Noah Snyder was supported in part by RTG grant DMS-0354321 and by an NSF postdoctoral grant.


\subsection{Background and Conventions}

The subject of quantum groups and quantum knot invariants suffers from a plethora of inconsistent conventions.  In this section we quickly recall important notions, specify our conventions, and give citations.  The reader is encouraged to skip this section and refer back to it when necessary.  In particular, most of Sections \ref{sec:invariants} and \ref{sec:identities} involve only diagram categories and do not require understanding quantum group constructions or their notation.  

\subsubsection{General conventions}
\begin{defn}
The $n^{th}$ quantum number $\qi{n}$ is defined as $$\frac{q^n-q^{-n}}{q-q^{-1}} = q^{n-1} + q^{n-3} + \cdots + q^{-n+1}.$$
 \end{defn}

Following \cite{MR2286123} the symbol $s$ will always denote a certain root of $q$ which will be specified as appropriate.

\subsubsection{Ribbon categories, diagrams, and knot invariants}

A ribbon category is a braided pivotal monoidal category satisfying a compatibility relation between the pivotal structure and the braiding.  See \cite{0908.3347} for details (warning: that reference uses the word tortile in the place of ribbon). We use the optimistic convention where diagrams are read upward.

The key property of ribbon categories is that if $\cC$ is a ribbon category there is a functor $\cF$ from the category of ribbons labelled by objects of $\cC$ with coupons labelled by morphisms in $\cC$ to the category $\cC$ (see \cite{MR1036112, MR1292673, 0908.3347}).  In particular, if $V$ is an object in $\cC$ and $L_V$ denotes a framed oriented link $L$ labelled by $V$, then 
$$\tilde{\mathcal{J}}_{\cC,V}: L \mapsto \cF(L_V) \in \End{\id}$$ 
is an invariant of oriented framed links (due to Reshetikhin-Turaev \cite{MR1036112}). Whenever $V$ is a simple object, the invariant depends on the framing through a `twist factor'. That is, two links $L$ and $L'$ which are the same except that $w(L)=w(L')+1$, where $w$ denotes the writhe, have invariants satisfying $\tilde{\mathcal{J}}_{\cC,V}(L) = \theta_V \tilde{\mathcal{J}}_{\cC,V}(L')$ for some $\theta_V$ in the ground field (not depending on $L$). Thus $\tilde{\mathcal{J}}_{\cC,V}$ can be modified to give an invariant which does not depend on framing.  

\begin{thm}
Let $\J{\cC}{V}(L)=\theta_V^{-w(L)} \tilde{\mathcal{J}}_{\cC,V}(L)$.  
Then $\J{\cC}{V}(L)$ is an invariant of links.
\end{thm}

Given any pivotal tensor category $\cC$ (in particular any ribbon category) and a chosen object $X \in \cC$, one can consider the full subcategory whose objects are tensor products of $X$ and $X^*$.  This subcategory is more convenient from the diagrammatic perspective because one can drop the labeling of strands by objects and instead assume that all strands are labelled by $X$ (here $X^*$ appears as the downward oriented strand).  Thus this full subcategory becomes a spider \cite{MR1403861}, which is an oriented version of Jones's planar algebras \cite{math.QA/9909027}.  If $X$ is symmetrically self-dual then this full subcategory is an unoriented unshaded planar algebra in the sense of \cite{MR2559686}.

Often one only describes the full subcategory (via diagrams) but wishes to recover the whole category.  If the original category was semisimple and $X$ is a tensor generator, then this can be acheived via the idempotent completion.  This is explained in detail in \cite{MR2559686, MR1217386, MR1403861}.  The simple objects in the idempotent completion are the minimal projections in the full subcategory.


\subsubsection{Conventions for knot polynomials and their diagram categories} \label{conventions}

In this subsection we give our conventions for the following knot polynomials: the Jones polynomial, the colored Jones polynomials, the HOMFLYPT polynomial, the Kauffman polynomial, and the Dubrovnik polynomial.  Each of these comes in a framed version as well as an unframed version.  The framed versions of these polynomials (other than the colored Jones polynomial) are given by simple skein relations.  These skein relations can be thought of as defining a ribbon category whose objects are collections of points (possibly with orientations) and whose morphisms are tangles modulo the skein relations and modulo all negligible morphisms (see \S \ref{subsec:semisimplification}).

We will often use the same name to refer to the knot polynomial and the category.  This is very convenient for keeping track of conventions.  The HOMFLYPT skein category and the Dubrovnik skein category are more commonly known as the Hecke category and the BMW category.

Contrary to historical practice, we normalize the polynomials so they are multiplicative for disjoint union. In particular, the invariant of the empty link is $1$, while the invariant of the unknot is typically nontrivial.

\paragraph{The Temperley-Lieb category}

We first fix our conventions for the Temperley-Lieb ribbon category $\TL(s)$.  Let $s$ be a complex number with $q=s^2$. 

The objects in Temperley-Lieb are natural numbers (thought of as disjoint unions of points).  The morphism space $\Hom{}{a}{b}$ consists of linear combinations of planar tangles with $a$ boundary points on the bottom and $b$ boundary points on the top, modulo the relation that each closed circle can be replaced by a multiplicative factor of $\qi{2} = q + q^{-1}$.  The endomorphism space of the object consisting of $k$ points will be called $\TL_{2k}$.

The braiding (which depends on the choice of $s = q^{\frac{1}{2}}$) is given by
\begin{equation*}
\begin{tikzpicture}[baseline]
\node (x) at (0,0){};
    \draw (x.45)-- (.5,.5);
    \draw (x.135) -- (-.5,.5);
    \draw (x.315) -- (.5,-.5);
    \draw (x.45) -- (-.5,-.5);
\end{tikzpicture}
= i s
\begin{tikzpicture}[baseline]
    \draw (1.5,.5) .. controls (2,0) .. (1.5,-.5);
    \draw (2.5,.5) .. controls (2,0) .. (2.5,-.5);
\end{tikzpicture}
-i s^{-1}
\begin{tikzpicture}[baseline]
    \draw (3.5,.5) .. controls (4,0) .. (4.5,.5);
    \draw (3.5,-.5) .. controls (4,0) .. (4.5,-.5);
\end{tikzpicture}.
\end{equation*}

We also use the following important diagrams.

\begin{itemize}
\item The Jones projections in $\TL_{2n}$:
$$e_i=\qi{2}^{-1} \begin{tikzpicture}[baseline]
    \draw (.2,-.5)--(.2,.5);
    \node at (.6,0) {$\cdots$};
    \draw (1,-.5)--(1,.5);
    \draw (1.2,-.5) arc (180:0:1mm);
    \draw (1.2,.5) arc (-180:0:1mm);
    \draw (1.6,-.5) -- (1.6,.5);
    \node at (2,0) {$\cdots$};
    \draw (2.3,-.5) -- (2.3,.5);
\end{tikzpicture}, \, i \in \{ 1, \ldots, n-1\};$$
\item
The Jones-Wenzl projection  $\JW{n}$ in $\TL_{2n}$ \cite{MR873400} which is the unique projection
with the property
$$\JW{n} e_i = e_i \JW{n} =0, \, \text{for each $i \in \{1, \ldots, n-1 \}$}.$$
\end{itemize}

\paragraph{The Jones polynomial} 

The framed Jones polynomial $\widetilde{J}$ (or Kauffman bracket) is the invariant coming from the ribbon category $\TL(s)$.  In particular, it is defined for unoriented framed links by 
\begin{align}
\bigcirc & = q + q^{-1} \notag \\
\intertext{and}
\overcrossing & =is \identity - i s^{-1} \cupcap. \label{eq:jones-skein}
\end{align}

This implies that
\begin{equation*}
\positivetwist  = i s^3  \verticalstrand \qquad \text{and} \qquad \negativetwist  =  - i s^{-3} \verticalstrand
\end{equation*}
so the twist factor is $i s^{3}$. 

The framing-independent Jones polynomial is defined by $J(L)=(-is^{-3})^{\writhe(L)} \widetilde{J}(L)$.  It satisfies the following version of the Jones skein relation
\begin{align}
q^2 \Oovercrossing - q^{-2} \Oundercrossing & =-i s \overcrossing - i s^{-1} \undercrossing \notag \\
    & = q \identity - q^{-1} \identity  \notag \\
    & = (q - q^{-1}) \Oidentity. \label{eq:oriented-jones-skein} \notag
\end{align}

\paragraph{The colored Jones polynomial}

The framed colored Jones polynomial $\Jf{\SL{2}}{(k)}(K)$ is the invariant coming from the simple projection $\JW{k}$ in $\TL(s)$.  The twist factor is $i^{k^2} s^{k^2+2k}$.

\paragraph{The HOMFLYPT polynomial}

The framed HOMFLYPT polynomial is given by the following skein relation.

\begin{equation}
w \Oovercrossing - w^{-1} \Oundercrossing = z \Oidentity
\end{equation}
\begin{align}
\Opositivetwist & = w^{-1} a \Overticalstrand & \Onegativetwist & = a w^{-1} \Overticalstrand
\end{align}
The twist factor is just $w^{-1}a$.

Thus, the framing-independent HOMFLYPT polynomial is given by the following skein relation.

\begin{equation}
a \Oovercrossing - a^{-1} \Oundercrossing = z \Oidentity
\end{equation}

\paragraph{The Kauffman polynomial} The Kauffman polynomial comes in two closely related versions, known as the Kauffman and Dubrovnik normalizations. Both are invariants of unoriented framed links.
The framed Kauffman polynomial $\widetilde{\Kauffman}$ is defined by
\begin{equation} \label{eq:kauffman}
\overcrossing + \undercrossing = z\left(\identity + \cupcap\right)
\end{equation}
\begin{align}
\positivetwist & = a \verticalstrand & \negativetwist & = a^{-1} \verticalstrand.
\end{align}

Here the value of the unknot is $\frac{a+a^{-1}}{z} - 1$.

The framed Dubrovnik polynomial $\widetilde{\Dubrovnik}$ is defined by
\begin{equation} \label{eq:dubrovnik}
\overcrossing - \undercrossing = z\left(\identity - \cupcap\right)
\end{equation}
\begin{align}
\positivetwist & = a \verticalstrand & \negativetwist & = a^{-1} \verticalstrand.
\end{align}
Here the value of the unknot is $\frac{a-a^{-1}}{z} + 1$.

In both cases, the twist factor is $a$.  The unframed Kauffman and Dubrovnik polynomials do not satisfy any conveniently stated skein relations.  The Kauffman and Dubrovnik polynomials are closely related to each other by $$\widetilde{\Dubrovnik}(L)(a, z) =  i^{-w(L)}  (-1)^{\#L} \widetilde{\Kauffman}(L)(i a,-i z),$$ where $\#L$ is the number of components of the link and $w(L)$ is the writhe of any choice of orientation for $L$ (which turns out not to depend, modulo $4$, on the choice of orientation).  This is due to Lickorish \cite{MR966143}  \cite[p. 466]{MR958895}.

\paragraph{Kuperberg's $G_2$ Spider}
We recall Kuperberg's skein theoretic description of the quantum $G_2$ knot invariant \cite{MR1403861, MR1265145} (warning, there is a sign error in the former source\footnote{In keeping with this tradition, the published version of this paper contains an error in the formula for the $G_2$ braiding; the first two coefficients were mistakenly interchanged.}).  Kuperberg's $q$ is our $q^2$ (which agrees with the usual quantum group conventions). The quantum $G_2$ invariant is defined by
$$
\overcrossing = \frac{1}{1+q^{2}} \mathfig{0.04}{G2/I} + \frac{1}{1+q^{-2}} \mathfig{0.08}{G2/H} + \frac{1}{q^{2}+q^{4}} \cupcap + \frac{1}{q^{-2}+q^{-4}} \identity
$$

where the trivalent vertex satisfies the following relations

\begin{align*}
\mathfig{0.04}{G2/loop} &= q^{10} + q^{8} + q^2 + 1 + q^{-2} + q^{-8} + q^{-10} \\
\mathfig{0.04}{G2/lollipop} &= 0 \\
\mathfig{0.04}{G2/bigon} &= -\left(q^6 + q^4 + q^2 + q^{-2} + q^{-4} + q^{-6}\right) \mathfig{0.0025}{G2/strand}\\
\mathfig{0.06}{G2/triangle} &=  \left(q^4 + 1 +q^{-4}\right)\mathfig{0.06}{G2/vertex}\\
\mathfig{0.06}{G2/square} &=  -\left(q^2+q^{-2}\right) \left(\mathfig{0.04}{G2/I} + \mathfig{0.08}{G2/H} \right)+ \left(q^2+1+q^{-2}\right) \left(\cupcap + \identity \right) \\
\mathfig{0.06}{G2/pentagon} &=  \mathfig{0.06}{G2/tree1} + \mathfig{0.06}{G2/tree2} + \mathfig{0.06}{G2/tree3} + \mathfig{0.06}{G2/tree4} + \mathfig{0.06}{G2/tree5} \\ 
& \qquad - \mathfig{0.06}{G2/forest1}-\mathfig{0.06}{G2/forest2}-\mathfig{0.06}{G2/forest3}-\mathfig{0.06}{G2/forest4}-\mathfig{0.06}{G2/forest5}.
\end{align*}

\subsubsection{Quantum groups}
\label{sec:quantumgroups}

Quantum groups are key sources of ribbon categories.  If $\mathfrak{g}$ is a complex semisimple Lie algebra, let $U_s(\mathfrak{g})$ denote the Drinfel'd-Jimbo quantum group, and let $\Rep{U_s(\mathfrak{g})}$ denote its category of representations.  This category is a ribbon category and hence given a quantum group and any representation the Reshetikhin-Turaev procedure gives a knot invariant.

We follow the conventions from \cite{MR2286123}.  See \cite[p. 2]{MR2286123} for a comprehensive summary of how his conventions line up with those in other sources. (In particular, our $q$ is the same as both Sawin's $q$ and Lusztig's $v$.) We make one significant change: we only require that the underlying Lie algebra $\mathfrak{g}$ be semi-simple rather than simple.  This does not cause any complications because $U_s(\mathfrak{g_1}\oplus \mathfrak{g_2}) \cong U_s(\mathfrak{g_1}) \boxtimes U_s(\mathfrak{g_2})$.

In particular, following  \cite{MR2286123}, we have variables $s$ and $q$ and the relation $s^L = q$ where $L$ is the smallest integer such that $L$ times any inner product of weights is an integer. The values of $L$ for each simple Lie algebra appear in Figure \ref{fig:quantum-group-data}. The quantum group itself and its representation theory only depend on $q$, while the braiding and the ribbon category depend on the additional choice of $s$.

For the quantum groups $U_s(\SO{n})$ we denote by $\vRep(U_q(\SO{n}))$ the collection of representations whose highest weight corresponds to a vector representation of $\SO{n}$ (that is, a representation of $\SO{n}$ which lifts to the non-simply connected Lie group $SO(n)$).  Note that the braiding on the vector representations does not depend on $s$ so we use $q$ as our subscript here instead.

Often the pivotal structure on a tensor category is not unique, and indeed for the representation theory of a quantum group the pivotal structures are a torsor over the group of maps from the weight lattice modulo the root lattice to $\pm 1$. In general there is no `standard' pivotal structure, but for the representation theory of a quantum group there is both the usual one defined by the Hopf algebra structure of the Drinfel'd-Jimbo quantum group, and Turaev's `unimodal' pivotal structure, $\uRep{U_s(\mathfrak{g})}$. Changing the pivotal structure by $\chi$, a map from the weight lattice modulo the root lattice to $\pm 1$, has two major effects: it changes both the dimension of an object and its twist factor by multiplying by  $\chi(V)$. 
The unimodal pivotal structure is characterized by the condition that every self-dual object is symmetrically self-dual.
One important particular case is that $\uRep{U_s(\SL{2})} \cong \TL(-is)$ \cite{0810.0084, 1002.0555}.

The twist factor for an irreducible representation $V$ is determined by the action of the ribbon element, giving (for the standard pivotal structure) $q^{\langle \lambda,\lambda+2\rho\rangle}$ where $\lambda$ is the highest weight of $V$.   Note that since $\langle \lambda,\lambda+2\rho\rangle \in \frac {1}{L} \mathbb{Z}$ (where $L$ is the exponent of the weight lattice mod the root lattice), the twist factor in general depends on a choice of $s=q^{\frac {1}{L}}$.

For $V=V_{(k)}$, the representation of $U_s(\SL{2})$ with highest weight $k$, the twist factor is $s^{k^2+2k}$ (notice this is the same as the $k$-colored Jones polynomial for $k$ even; for $k$ odd the twist factors differ by a sign as predicted by $\uRep{U_s(\SL{2})} \cong \TL(-is)$).  For $V=V^\natural$, the standard representation of $U_s(\SL{n})$, the twist factor is $s^{n-1}$. For the standard representations of $\SO{2n+1}$, $\SP{2n}$ and $\SO{2n}$ the twist factors are $q^{4n}$, $q^{2n+1}$ and $q^{2n-1}$ respectively.
 The twist factor for the representation $V_{ke_1}$ of $\SO{2n+1}$ is $q^{2k^2+(4n-2)k}$. 
Note the the representation $V_{ke_1}$ of $\SO{3}$ is the representation $V_{(2k)}$ of $\SL{2}$ and in this case the twist factor agrees with the first one given in this paragraph.
The twist factor for the representation $V_{ke_1}$ of $\SO{2n}$ is $q^{k^2+2(n-1)k}$. Later we will need the twist factors for the representations $V_{3e_1}$, $V_{e_{n-1}}$ and $V_{3e_{n-1}}$  of $\SO{2n}$. These are $q^{6n+3}$, $q^{\frac{1}{4}n(2n-1)}$ and $q^{\frac{3}{4}n(2n+1)}$. The twist factor for the $7$-dimensional representation of $G_2$ is $q^{12}$.

The invariants of the unknot are just the quantum dimensions. For the standard representations of $\SL{n}$, $\SO{2n+1}$, $\SP{2n}$ and $\SO{2n}$ these are $\qiq{n}{q}, \qiq{2n}{q^2} + 1,\qiq{2n+1}{q} - 1$ and $\qiq{2n-1}{q} + 1$ respectively.

The invariants of the standard representations are specializations of the HOMFLYPT or Dubrovnik polynomials. Written in terms of the framing-independent invariants, we have
\begin{align}
\label{eq:homflypt-A}%
\text{HOMFLYPT}(L)(q^n,q-q^{-1}) & = \J{\SL{n}}{V^\natural}(L)(q), \\
\label{eq:dubrovnik-B}%
\text{Dubrovnik}(L)(q^{4n}, q^2 - q^{-2}) & = \J{\SO{2n+1}}{V^\natural}(L)(q), \\
\label{eq:dubrovnik-C}%
\text{Dubrovnik}(L)(-q^{2n+1}, q-q^{-1}) & =  (-1)^{\sharp L} \J{\SP{2n}}{V^\natural}(L)(q), \\
\intertext{and}
\label{eq:dubrovnik-D}%
\text{Dubrovnik}(L)(q^{2n-1}, q - q^{-1}) & = \J{\SO{2n}}{V^\natural}(L)(q).
\end{align}

These identities are `well-known', but it's surprisingly hard to find precise statements in the literature, and we include these mostly for reference. The identities follow immediately from Theorem \ref{thm:eigenvalues} below, and the fact that the eigenvalues of the braiding on the natural representations of $\SL{n}, \SO{2n+1}, \SP{2n}$ and $\SO{2n}$ are 
$(-s^{-n-1}, s^{n-1}), (q^{-4n}, q^2, -q^{-2} ), (-q^{-2n-1}, q, -q^{-1})$ and $(q^{-2n+1}, q, -q^{-1})$ respectively. The sign in Equation \eqref{eq:dubrovnik-C} appears because Theorem \ref{thm:eigenvalues} does not apply immediately to the natural representation of $\SP{2n}$, which is antisymmetrically self-dual.  Changing to the unimodal pivotal structure fixes this, introduces the sign in the knot invariant, and explains the discrepancy between the value of $a$ in the specialization of the Dubrovnik polynomial and the twist factor for the natural representation of $\SP{2n}$.

We'll show using techniques inspired by \cite{MR1854694, MR1710999, MR2132671} that several of these identities between knot polynomials come from functors between the corresponding categories.

\subsubsection{Comparison with the \code{KnotTheory`} package}
The HOMFLYPT and Kauffman polynomials defined here agree with those available in the \code{Mathematica} package \code{KnotTheory`} (available at the Knot Atlas \cite{katlas}), except that in the package the invariants are normalized so that their value on the unknot is $1$. The Jones polynomial in the package uses `bad' conventions from the point of view of quantum groups. You'll need to substitute $q \mapsto q^{-2}$, and then multiply by $q+q^{-1}$ to get from the invariant implemented in \code{KnotTheory`} to the one described here. The $G_2$ spider invariant described here agrees with that calculated using the \code{QuantumKnotInvariant} function in the package. The function \code{QuantumKnotInvariant} in the package calculates the framing-independent invariants from quantum groups described here.

\subsubsection{Semisimplification} \label{subsec:semisimplification}

Suppose that $\cC$ is a spherical tensor category which is $\Complex$-linear and which is idempotent complete (every projection has a kernel and an image).  Let $N$ be the collection of negligible morphisms ($f$ is negligible if $\tr{f g} = 0$ for all $g$).  Call a collection of morphisms $I \in \cC$ an ideal if $I$ is closed under composition and tensor product with arbitrary morphisms in $\cC$. We recall the following facts:

\begin{itemize}
\item $N$ is an ideal.
\item Any ideal in $\cC$ is contained in $N$.
\item If $\cC$ semisimple then $N = 0$.
\item If $\cC$ is abelian, then $\cC/N = \cC^{ss}$ is semisimple.
\item If $\cD$ is psuedo-unitary (pivotal, and all quantum dimensions are positive, up to a fixed Galois conjugacy) and $\cF: \cC \rightarrow \cD$ is a functor of pivotal categories, then $\cF$ is trivial on $N$.
\end{itemize}

There are some technical issues which, while not immediately relevant to this paper, are important to keep in mind when dealing with semisimplifications. First, $\cC/N$ may not always be semisimple.  Furthermore, if $\cD$ is semisimple but not psuedo-unitary there may be a functor $\cF: \cC \rightarrow \cD$ which does not factor through $\cC/N$.

\begin{ex}
If $q$ is a root of unity, and $a$ is not an integer power of $q$, then the quotient of the Dubrovnik category at $(a, z=q-q^{-1})$ by negligibles is not semisimple \cite[Cor.~7.9]{MR2132671}.
\end{ex}

\begin{ex}
This example is adapted from \cite[Remark 8.26]{MR2183279}.  Let $$\cE = \vRep U_{q=\exp(\frac{2 \pi i}{10})}(\SO{3})$$ be the Yang-Lee category.  This fusion category has two objects, $1$ and $X$, satisfying $X\otimes X \cong X \oplus 1$. The object $X$ has dimension the golden ratio.  Let $\cE'$ be a Galois conjugate of $\cE$ where $X$ has dimension the conjugate of the golden ratio.  Let $\cD = \cE \boxtimes \cE'$; this is a non-pseudo-unitary semisimple category.  Note that $X \boxtimes X$ is a symmetrically self-dual object with dimension $-1$.  Hence there is a functor from $\cC= TL_{d=-1} \rightarrow \cD$ sending the single strand to $X\boxtimes X$ (see \S \ref{sec:ribbonfunctors} for more details).  The second Jones-Wenzl idempotent is negligible in $TL_{d=-1}$ but it is not killed by this functor.
\end{ex}


For further details see \cite{MR1686423} \cite{MR1217386} and \cite[Proposition 5.7]{MR2348906}.

\subsubsection{Quantum groups at roots of unity}

When $s$ is a root of unity, by $\Rep{U_s(\mathfrak{g})}$ we mean the semisimplified category of tilting modules of the Lusztig integral form.  We only ever consider cases where $q$ is a primitive $\ell$th root of unity with $\ell$ large enough in the sense of \cite[Theorem 2]{MR2286123}.  The key facts about this category are described in full generality in \cite{MR2286123} (based on earlier work by Andersen, Lusztig, and others):

\begin{itemize}
\item The isomorphism classes of simple objects correspond to weights in the fundamental alcove.  (Be careful, as when the Lie algebra is not simply laced the shape of the fundamental alcove depends on the factorization of the order of the root of unity \cite[Lemma 1]{MR2286123}.
\item The dimensions and twist factors for these simple objects are given by specializing the formulas for dimensions and twist factors from generic $q$.
\item The tensor product rule is given by the quantum Racah rule \cite[\S 5]{MR2286123}. 
\end{itemize}

\subsubsection{Modularization}
\label{sec:modularization}

We review the theory of modularization or deequivariantization developed by M\"uger \cite{MR1749250} and Brugui\`eres \cite{MR1741269}.  Suppose that $\cC$ is a ribbon category and that $\cG$ is a collection of invertible ($X \tensor X^* \iso \mathbf{1}$, or equivalently $\dim X = \pm 1$) simple objects in $\cC$ which satisfy four conditions:
\begin{itemize}
\item $\cG$ is closed under tensor product.
\item Every $V \in \cG$ is transparent (that is, the positive and negative braidings with any object $W \in \cC$ are equal).
\item $\dim V = 1$ for every $V \in \cG$. 
\item The twist factor $\theta_V$ is $1$ for every $V \in \cG$. 
\end{itemize}

Let $\cC // \cG$ denote the M\"uger-Brugui\`eres deequivariantization.  There is a faithful essentially surjective functor $\cC \rightarrow \cC // \cG$.  This functor is not full because in the deequivariantization there are more maps:  in $\cC // \cG$ every object in the image of $\cG$ becomes isomorphic to the trivial object. 

We'll often write $\cC // X$ to denote the deequivariantization by the collection of tensor powers of some object $X$.

A ribbon functor between premodular (that is, ribbon and fusion) categories $\cF: \cC \rightarrow \cC'$ is called a modularization if it is dominant (every simple object in $\cC'$ appears as a summand of an object in the image of $\cF$) and if $\cC'$ is modular.

\begin{thm}
Suppose that $\cC$ is a premodular category whose global dimension is nonzero.  Any modularization is naturally isomorphic to $\cF: \cC \rightarrow \cC//\cG$ where $\cG$ is the set of all transparent objects.  A modularization exists if and only if every transparent object has dimension $1$ and twist factor $1$.
\end{thm}

In Section \ref{sec:KM}, we compute modularizations using the following lemma.

\begin{lem}
\label{lem:free-quotient}
Suppose $\cM$ is a modular $\tensor$-category, which is a full subcategory of a tensor category $\cC$. Denote by $\cI$ the subcategory of invertible objects in $\cC$, and assume they are all transparent, and that the group of objects $\cI$ acts (by tensor product) freely on $\cC$. 
Then the orbits of $\cI$ each contain exactly one object from $\cM$, and the modularization $\cC // \cI$ is equivalent to $\cM$.
\end{lem}

For further detail, see \cite[\S 1.3-1.4]{MR1854694}.  Related notions appear in the physics literature, for example \cite{0808.0627}.

\section{Link invariants from $\mathcal{D}_{2n}$}
\label{sec:invariants}%

\subsection{The $ \mathcal{D}_{2n}$ planar algebras}
\label{sec:d2n}

The $\mathcal{D}_{2n}$ planar algebras were first discovered during the classification of subfactors below index $4$, where there is an ADE classification of the principal graphs. This classification has some pecularities: there are the $A_n$ subfactors, then $D_{2n}$ subfactors (but no $D_\text{odd}$ subfactors), and finally the $E_6$ and $E_8$ subfactors (but no $E_7$ subfactor). This classification is described in  \cite{MR999799, MR996454,
MR1145672, MR1936496}. In the $ADE$ subfactor planar algebras the shading is irrelevant, corresponding to the fact that these subfactors come from underlying tensor categories. These tensor categories have been described directly, via commutative algebra objects in $\Rep{U_{s=\exp({\frac{2\pi i}{16n-8}})}(\SL{2})}$, in \cite{MR1936496}. (Here, the $D_\text{odd}$ and $E_7$ graphs appear as the fusion graphs of module categories for noncommutative algebra objects.) 

The $D_{2n}$ subfactors were first constructed in \cite{MR1308617} using an automorphism of the subfactor $A_{4n-3}$. This construction is essentially equivalent to the deequariantization procedure described above. The $D_{2n}$ tensor categories are the simplest example of deequivariantization. In this paper, we use a skein theoretic description of the $\mathcal{D}_{2n}$ planar algebra, from our paper \cite{MR2559686}. 

We recall the definition from  \cite{MR2559686}.

\begin{defn}\label{def:pa}
Fix $q=\exp(\frac{2 \pi i}{8n-4})$.
Let $\pa$ be the planar algebra generated by a single ``box" $S$ with $4n-4$ strands, modulo the following relations.
\begin{enumerate}

\item\label{delta} A closed circle is equal to $[2]_q = (q+q^{-1}) = 2 \cos(\frac{2 \pi}{8n-4})$ times the empty diagram.

\item\label{rotateS} Rotation relation:
$
\begin{tikzpicture}[baseline]
    \clip (-.6,-1) rectangle (.6,1);

    \draw (-.32,0) -- ++(90:3mm) arc (0:180:1mm) -- ++(-90:3.5mm) arc (-180:0:5.3mm) -- ++(90:15mm);
    \draw (.32,0) -- ++(90:15mm);
    \draw (-.15,0) -- ++(90:15mm);
    \draw (.07,.7) node {...};
    
    \node (S) at (0,0) [circle, draw, fill=white] {$S$};
    \filldraw (S.180) circle (.5mm);
\end{tikzpicture}
= i \cdot
\begin{tikzpicture}[baseline]
    \clip (-.6,-1) rectangle (.6,1);

    \draw(-.32,0) -- ++(90:15mm);
    \draw (.32,0) -- ++(90:15mm);
    \draw (-.07,.7) node {...};
    \draw (.15,0) -- ++(90:15mm);
    
    \node (S) at (0,0) [circle, draw, fill=white] {$S$};
    \filldraw (S.180) circle (.5mm);
\end{tikzpicture}
$

\item\label{capS} Capping relation:
$
\begin{tikzpicture}[baseline]
 \node (S) [circle,draw] {$S$};
 \filldraw (S.157) circle (.5mm);
 \draw (S.135) .. controls +(135:3mm) and +(90:3mm) .. (S.90);
 \draw (S.45) -- +(45:3mm);
 \draw (S.0) -- +(0:3mm);
 \draw (S.-45) node [below] {$\cdot$};
  \draw (S.-90) node [below] {$\cdot$};
 \draw (S.-135) node [below] {$\cdot$};
 \draw (S.180) -- +(180:3mm);
\end{tikzpicture}
= 0
$

\item\label{twoS} Two-$S$ relation:
$
\begin{tikzpicture}[baseline]
 \node at (0,.7) (S) [circle,draw] {$S$};
     \filldraw (S.180) circle (.5mm);

 \node at (0,-.7) (S') [circle,draw] {$S$};
     \filldraw (S'.180) circle (.5mm);

 \draw (S.45) -- +(45:4mm);
 \draw (S.90) node [above] {$\ldots$};
 \draw (S.135) -- +(135:4mm);
 \draw (S'.225) -- +(-135:4mm);
 \draw (S'.270) node [below] {$\ldots$};
 \draw (S'.-45) -- +(-45:4mm);
\end{tikzpicture}
=\qi{2n-1} \cdot
\begin{tikzpicture}[baseline]
\node (0,0) (JW) [rectangle,draw] {$\JW{4n-4}$};
\draw (JW.180);

\draw (JW.35) -- +(90:11mm);
\draw (JW.90) node [above] {$\cdots$};
\draw (JW.145) -- +(90:11mm);

\draw (JW.-35) -- +(-90:11mm);
\draw (JW.-90) node [below] {$\cdots$};
\draw (JW.-145) -- +(-90:11mm);
\end{tikzpicture}
$

\end{enumerate}
\end{defn}

In \cite{MR2559686}, our main theorem included the following
statements:
\begin{thm}
\label{thm:main}%
$\pa$ is the $ \mathcal{D}_{2n}$ subfactor planar algebra:
\begin{enumerate}
\item the space of closed diagrams is
$1$-dimensional (in particular, the relations are consistent),
\item it is spherical, 
\item It is unitary, and hence pseudo-unitary and semisimple.
\item the principal graph of $\pa$ is the Dynkin diagram $D_{2n}$.
\end{enumerate}
\end{thm}

In order to prove this theorem, we made liberal use of the following ``half-braided''
relation:

\begin{thm}\label{thm:passacrossS}
You can isotope a strand above an $S$-box, but isotoping a strand below
an $S$-box introduces a factor of $-1$.
\begin{enumerate}
\item\label{pullover}
$\begin{tikzpicture}[baseline]
    \clip (-1.1,-1.4) rectangle (1.1,1.4);

    \node (S) at (0,0) [circle, draw] {$S$};
    \filldraw (S.180) circle (.5mm);

    \draw[rounded corners=2mm] (-1,2) -- (-1,-1) -- (1,-1) -- (1,2);

    \draw (S.135) -- ++(90:15mm);
    \draw (S.90) node [above] {...};
    \draw (S.45) -- ++(90:15mm);
\end{tikzpicture}
=
\begin{tikzpicture}[baseline]
    \clip (-1.1,-1.4) rectangle (1.1,1.4);

    \node (S) at (0,0) [circle, draw] {$ S$};
    \filldraw (S.180) circle (.5mm);
    \node (x1) at (S.135 |- -1,1){};
    \node (x2) at  (S.45 |- -1,1){} ;

    \draw[rounded corners=2mm] (-1,1.5) --  (-1,1) -- (1,1) -- (1,1.5);

    \draw (S.135) -- (x1.-90);
    \draw (x1.90) -- ++(90:5mm);
    \draw (S.90) node [above] {...};
    \draw (x2.90) -- ++(90:5mm);
    \draw (S.45) -- (x2.-90);
\end{tikzpicture}$

\item\label{pullunder}
$\begin{tikzpicture}[baseline]
    \clip (-1.1,-1.4) rectangle (1.1,1.4);

    \node (S) at (0,0) [circle, draw] {$S$};
     \filldraw (S.180) circle (.5mm);

    \draw[rounded corners=2mm] (-1,2) -- (-1,-1) -- (1,-1) -- (1,2);

    \draw (S.135) -- ++(90:15mm);
    \draw (S.90) node [above] {...};
    \draw (S.45) -- ++(90:15mm);
\end{tikzpicture}
=-
\begin{tikzpicture}[baseline]
    \clip (-1.1,-1.4) rectangle (1.1,1.4);

    \node (S) at (0,0) [circle, draw] {$S$};
    \filldraw (S.180) circle (.5mm);
    \node (x1) at (S.135 |- -1,1){};
    \node (x2) at  (S.45 |- -1,1){} ;

    \draw[rounded corners=2mm](-1,1.5) -- (-1,1) -- (x1.180);
    \draw (x1.0) -- (x2.180);
    \draw[rounded corners=2mm] (x2.0) -- (1,1)--(1,1.5);

    \draw (S.135) -- (x1.90);
    \draw (x1.90) -- ++(90:5mm);
    \draw (S.90) node [above] {...};
    \draw (x2.90) -- ++(90:5mm);
    \draw (S.45) -- (x2.90);
\end{tikzpicture}$

\end{enumerate}
\end{thm}

In \cite{MR2559686} we gave a skein theoretic description of each isomorphism class of simple projections in $ \mathcal{D}_{2n}$.  These are $\JW{i}$ for $0 \leq i \leq 2n-3$, the projection $P = \frac{1}{2}( \JW{2n-2} + S)$, and the projection $Q = \frac{1}{2} (\JW{2n-2} - S)$.  We also gave a complete description of the tensor product rules for these projections (most of which appear in \cite{MR1145672,MR1936496}).

By the {\em even part} of $ \mathcal{D}_{2n}$, which we'll denote $\frac{1}{2}  \mathcal{D}_{2n}$, we mean the full subcategory whose objects consist of collections of an even number of points.  The simple projections in the even part of $ \mathcal{D}_{2n}$ are $\JW{0}, \JW{2}, \ldots, \JW{2n-4}, P, Q$.

\begin{prop}
$\frac{1}{2} \mathcal{D}_{2n} \cong \vRep{U_{q=\exp(\frac{2 \pi i}{8n-4})}(\SO{3})}^{modularize}$.  
\end{prop}

\begin{proof}
To see this we observe that $\vRep{U_{q=\exp(\frac{2 \pi i}{8n-4})}(\SO{3})}$ is the even part of    $\Rep{U_{q=\exp(\frac{2 \pi i}{8n-4})}(\SL{2})}$, and the even parts of  $\Rep{U_{q=\exp(\frac{2 \pi i}{8n-4})}(\SL{2})}$ and $\TL$ are the same at that value of $q$ (the change in pivotal structure does not affect the even part).  Hence there is a functor $$\vRep{U_{q=\exp(\frac{2 \pi i}{8n-4})}(\SO{3})} \rightarrow \frac{1}{2} \mathcal{D}_{2n}$$  The description of simple objects in $ \mathcal{D}_{2n}$ shows that this functor is dominant (as $P+Q=\JW{2n-2}$), and a simple calculation shows that $\frac{1}{2} \mathcal{D}_{2n}$ has no transparent objects and so is modular. Hence, the claim follows by the uniqueness of modularization.
\end{proof}

\subsection{Invariants from $ \mathcal{D}_{2n}$}
Although $ \mathcal{D}_{2n}$ is not a ribbon category, its even part is ribbon.  This is in \cite[p. 33]{MR1936496}; we prefer to give a skein theoretic explanation.  Define the braiding using the Kauffman bracket formula.  This braiding clearly satisfies Reidemeister moves $2$ and $3$, as well as the additional ribbon axiom: all of these equalities of diagrams only involve diagrams in Temperley-Lieb, which is a ribbon category.  The only additional thing to check is naturality, which means that any diagram can pass over or under a crossing without changing.  This follows immediately from the ``half-braiding" relation, because all crossings involve an even number of strands.

Since the even part of $ \mathcal{D}_{2n}$ is a ribbon category, any simple object in it defines a link invariant.  For the simple objects $\JW{2k}$ this invariant is just a colored Jones polynomial.  So we concentrate on invariants involving $P$ and $Q$.  Given an oriented framed link $L$, to get the framed $P$-invariant, we first $2n-2$ cable it and place a $P$ (going in the direction of the orientation) on each component.  See Figure \ref{PInvHopf} for an example.


Then we evaluate this new picture in the $ \mathcal{D}_{2n}$ planar algebra (using the Kauffman resolution of crossings).  

In the usual way, we can make it into an invariant of unframed links, which we will call $\J{ \mathcal{D}_{2n}}{P}(L)$.  Since $P = P \JW{2n-2}$, the twist factor is the same as that for $\JW{2n-2}$, namely $q^{2n(n-1)}$.

\begin{thm}
\label{thm:half}
For a knot $K$ (but not for a link!), $$\J{ \mathcal{D}_{2n}}{P}(K)=\frac{1}{2}\J{\SL{2}}{(2n-2)}(K)=\J{ \mathcal{D}_{2n}}{Q}(K)$$
\end{thm}

\begin{proof}
To compute $\Jf{ \mathcal{D}_{2n}}{P}(K)$ we $(2n-2)$-cable $K$, and insert one $P=\frac{1}{2}(\JW{2n-2}+S)$ somewhere.  When we split this into the sum of two diagrams, the diagram with the $S$ in it is zero, since in every resolution the $S$ connects back up with itself.  Meanwhile, the diagram with the $\JW{2n-2}$ in it is the colored Jones polynomial.  Thus $\Jf{ \mathcal{D}_{2n}}{P}(K)=\frac{1}{2}\Jf{\SL{2}}{(2n-2)}(K)$.
Exactly the same argument holds for $Q$.  Furthermore, the twist factors for $P$, $Q$ and $\JW{2n-2}$ are all equal as computed above.
\end{proof}

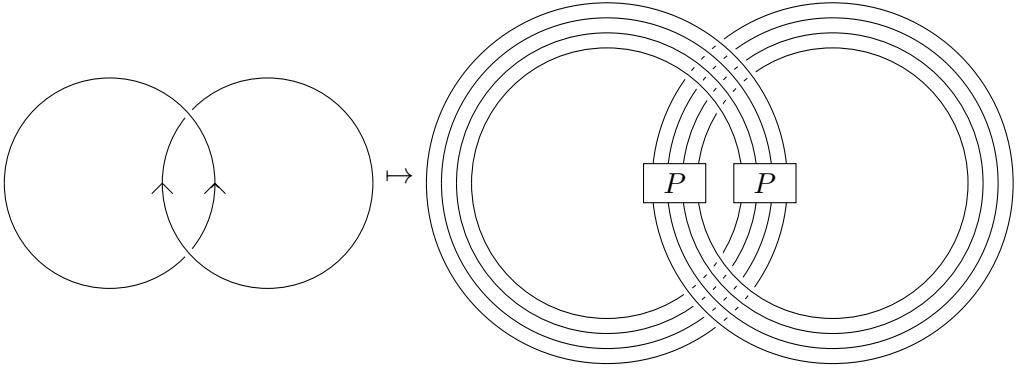
\begin{figure}
\begin{tikzpicture}[scale=.7, baseline=0]
\begin{pgfonlayer}{background}	
	\draw (5,0) arc (0:180:2cm);
	
	\draw[line width=4pt,white,opaque] (2,0) arc (0:90:2cm);
	
	\draw (2,0) arc (0:180:2cm);

	\draw[<-] (2,0) arc (0:-180:2cm);
	
	\draw[line width=4pt,white,opaque] (1,0) arc (-180:-90:2cm);
			
	\draw(5,0) arc (0:-180:2cm);

	\draw (2,0) -- (2.2,-.2);
	\draw (2,0) -- (1.8,-.2);
	\draw (1,0) -- (1.2,-.2);
	\draw (1,0) -- (0.8,-.2);

\end{pgfonlayer}
\end{tikzpicture}
$\mapsto$
\begin{tikzpicture}[scale=1, baseline=0]
	\node at (.9,0) [rectangle,draw, fill=white] {$\; P \;$};
	\node at (2.1,0)  [rectangle,draw, fill=white]  {$\; P \;$};

\begin{pgfonlayer}{background}	
	\draw (5,0) arc (0:180:2cm);
	\draw (4.8,0) arc (0:180:1.8cm);
	\draw (5.2,0) arc (0:180:2.2cm);
	\draw (5.4,0) arc (0:180:2.4cm);	
	
	\draw[line width=4pt,white,opaque] (2,0) arc (0:90:2cm);
	\draw[line width=4pt,white,opaque] (1.8,0) arc (0:90:1.8cm);
	\draw[line width=4pt,white,opaque] (2.2,0) arc (0:90:2.2cm);
	\draw[line width=4pt,white,opaque] (2.4,0) arc (0:90:2.4cm);
	
	\draw (2,0) arc (0:180:2cm);
	\draw (1.8,0) arc (0:180:1.8cm);
	\draw (2.2,0) arc (0:180:2.2cm);
	\draw (2.4,0) arc (0:180:2.4cm);		
	
	\draw (2,0) arc (0:-180:2cm);
	\draw (1.8,0) arc (0:-180:1.8cm);
	\draw (2.2,0) arc (0:-180:2.2cm);
	\draw (2.4,0) arc (0:-180:2.4cm);
	
	\draw[line width=4pt,white,opaque] (1,0) arc (-180:-90:2cm);
	\draw[line width=4pt,white,opaque] (1.2,0) arc (-180:-90:1.8cm);
	\draw[line width=4pt,white,opaque] (0.8,0) arc (-180:-90:2.2cm);
	\draw[line width=4pt,white,opaque] (0.6,0) arc (-180:-90:2.4cm);
			
	\draw (5,0) arc (0:-180:2cm);
	\draw (4.8,0) arc (0:-180:1.8cm);
	\draw (5.2,0) arc (0:-180:2.2cm);	
	\draw (5.4,0) arc (0:-180:2.4cm);	

\end{pgfonlayer}

\end{tikzpicture}
\caption{Computing the framed $ \mathcal{D}_{10}$ invariant of the Hopf link.}\label{PInvHopf}
\end{figure}

\subsection{A refined invariant}

Although this section isn't necessary for the rest of this paper, it may nevertheless be of interest.
We can slightly modify this construction to produce a more refined invariant for links.  Instead of labeling every component with $P$ or every component with $Q$ we could instead label some components with $P$ and others with $Q$.  This would not be an invariant of links, but if you fix which number of links to label with $P$ and sum over all choices of components this is a link invariant.  Notice that since the twist factors for $P$ and $Q$ are the same, this definition makes sense either for framed or unframed versions of the invariant.

\begin{defn}
For $a$ a positive integer, let $\J{ \mathcal{D}_{2n}}{P/Q}^a(L)$ be the sum over all ways of labelling $a$ components of $L$ with $P$ and the remaining components with $Q$.
\end{defn}

Since $P = \frac{1}{2}(\JW{2n-2}+S)$ and $Q=\frac{1}{2}(\JW{2n-2}-S)$, these invariants can be written in terms of simpler-to-compute invariants.

\begin{defn}
Let $\mathcal{J}^k_{ \mathcal{D}_{2n},S/f}(L)$ be $2^{-\ell}$ times the sum of all the ways to put an $S$ on $k$ of the link components and an $\JW{n}$ on the rest of the components.
We call this the {\em $k$-refined $( \mathcal{D}_{2n},P)$-invariant} of an $\ell$-component link $L$.
\end{defn}

This is a refinement of the $( \mathcal{D}_{2n},P)$ link invariant in that
$$\sum_{k=0}^{k=\ell} \mathcal{J}^k_{ \mathcal{D}_{2n},S/f}(L) = \J{ \mathcal{D}_{2n}}{P}$$

More precisely we have the following lemma.

\begin{lem}
\begin{align*}
\J{ \mathcal{D}_{2n}}{P/Q}^a(L) &= \sum_{i=0}^a \sum_{j=0}^{\ell-a} (-1)^{\ell-a-j} \begin{pmatrix}i+j\\ i\end{pmatrix} \begin{pmatrix}\ell-(i+j)\\ a-i\end{pmatrix} \J{ \mathcal{D}_{2n}}{S/f}^{i+j}(L) \\&= \sum_{k=0}^\ell (-1)^{\ell-a-k} \J{ \mathcal{D}_{2n}}{S/f}^{k}(L) \sum_{i=0}^{\text{min}(k,a)} (-1)^i \begin{pmatrix}k \\ i\end{pmatrix} \begin{pmatrix}\ell-k\\a-i\end{pmatrix}.
\end{align*}
\end{lem}

These refined invariants can detect more information than the ordinary invariant.  For example, although we will show in the next section that the $ \mathcal{D}_4$ invariant is trivial, it is not difficult to see using the methods of the next section that its refined invariants detect linking number mod $3$.

\section{Knot polynomial identities}
\label{sec:identities}%

The theorems of this section describe how to identify an invariant coming from an object in a ribbon category as a specialization of the Jones, HOMFLYPT or Kauffman polynomials. These theorems are  well-known to the experts, and versions of them can be found in \cite{MR958895, MR1237835,MR2132671}.  Since we need explicit formulas for which specializations appear we collect the proofs here.

We then identify cases in which these theorems apply, namely $\mathcal{D}_{2n}$ for $n=2,3,4$ and $5$, and explain exactly which specializations occur.

There is a similar procedure, due to Kuperberg, for recognizing knot invariants which are specializations of the $G_2$ knot polynomial.  We apply this technique to $\mathcal{D}_{14}$.  

The identities in this section do not follow from the knot polynomial identities in \cite{MR1134131} \cite[\S 6, Table 2]{MR980759}. (But most of those identities follow from the technique outlined in this section.)

\subsection{Recognizing a specialization of Jones, HOMFLYPT, or Kauffman}
Identifying a knot invariant as a specialization of a classical knot polynomial happens in two steps. Let's say you're looking at the knot invariant coming from an object $V$ in a ribbon category. First, you look at the direct sum decomposition of $V \tensor V$, and hope that you don't see too many summands. Theorem \ref{thm:summands} below describes how to interpret this decomposition, hopefully guaranteeing that the invariant is either trivial, or a specialization of Jones, HOMFLYPT, or Kauffman. If this proves successful, you next look at the eigenvalues of the braiding on the summands of $V \tensor V$. Theorem \ref{thm:eigenvalues} then tells you exactly which specialization you have.

\begin{thm}
\label{thm:summands}
Suppose that $V$ is a simple object in a ribbon category $\cC$ and that if $V$ is self-dual then it is symmetrically self-dual.
\begin{enumerate}
\item If $V \otimes V$ is simple, then $\dim{V}=\pm1$ and the link invariant $\J{\cC}{V} =(\dim V)^\#$ where $\#$ is the number of components of the link.
\item If $V \otimes V = \id \oplus L$ for some simple object $L$, then the link invariant $\J{\cC}{V}$ is a specialization of the Jones polynomial.
\item If $V \otimes V = L \oplus M$ for some simple objects $L$ and $M$, then the link invariant $\J{\cC}{V}$ is a specialization of HOMFLYPT.
\item If $V \otimes V = \id \oplus L \oplus M$ for some simple objects $L$ and $M$, then the link invariant $\J{\cC}{V}$ is a specialization of either the Kauffman polynomial or the Dubrovnik polynomial.
\end{enumerate}
\end{thm}

\begin{proof}
    \begin{enumerate}
    \item {\bf Trivial case} 
    
    Since the category is spherical and braided, $\End{V \otimes V} \cong \End{V \otimes V^*}$.  Hence if $V \otimes V$ is simple we must have $V \otimes V^* = \id$, so $\dim V = \pm 1$.
    Also by simplicity, $\End{V \otimes V}$ is one dimensional, and
    so, up to constants,  a crossing is equal to the identity map. Suppose now that
    $=\alpha $.  Then, Reidemeister two tells us that
    $=\alpha^{-1} $.  Capping this off shows that the twist factor is $\alpha \dim V$.  Thus the framing corrected skein relation is
    $$ = \dim V  = .$$
    
    The equality of the two crossings lets us unlink any link, showing that the framing corrected invariant is $(\dim V)^\#$, where $\#$ is the number of components. 

\item {\bf Jones polynomial case}

Since $\End{V\otimes V}$ is $2$-dimensional there must be a linear dependence between the crossing and the two basis diagrams of Temperley-Lieb.  (If these two Temperley-Lieb diagrams were linearly dependent, then $V\otimes V \cong \id$, contradicting the assumption).  Hence we must have a relation of the form
$$=A  + B .$$
Following Kauffman, rotate this equation, glue them together and apply Reidemeister $2$ to see that $B=A^{-1}$ and $A^2+A^{-2} = \dim V$.  Hence this invariant is given by the Kauffman bracket.

\item {\bf HOMFLYPT case}

Since $\End{V\otimes V}$ is $2$-dimensional there must be a linear dependence between the two crossings and the identity (we can't use the cup-cap diagram here because $V$ is not self dual).  Hence, we have that
$$\alpha  + \beta  = \gamma $$ for some $\alpha$, $\beta$, and $\gamma$.  If $\alpha$ or $\beta$ were zero, $\End{V\otimes V}$ would be $1$-dimensional, so we must have that $\alpha$ and $\beta$ are nonzero.  Hence we can rescale the relation so that $\alpha = w$, $\beta = -w^{-1}$, and $\gamma = z$.  Since the twist is some multiple of the single strand we can define $a$ such that the twist factor is $w^{-1}a$.  Thus we've recovered the framed HOMFLYPT skein relations.

\item {\bf Kauffman case}

Since $V \tensor V$ has three simple summands, its endomorphism space is $3$ dimensional. Moreover, since one of the summands is the trivial representation, one such endomorphism is the `cup-cap' diagram $\cupcap$. There must be some linear relation of the form
\begin{equation*}
p \overcrossing + q \undercrossing + r \identity + s \cupcap = 0.
\end{equation*}
The space of such relations is invariant under a $\pi/2$ rotation, and fixed under a $\pi$ rotation, so there must be a linear relation which is either a $(+1)$- or $(-1)$-eigenvector of the $\pi/2$ rotation. That is, there must be a relation of the form
\begin{equation*}
A\left( \overcrossing \pm \undercrossing\right)  = B \left(\identity \pm \cupcap\right)
\end{equation*}
If $A$ were zero, this would be a linear relation between $\identity$ and $\cupcap$, which would imply that $V \tensor V \cong \id$. Thus we can divide by $A$, and obtain either the Kauffman polynomial (Equation \eqref{eq:kauffman}) or Dubrovnik polynomial (Equation \eqref{eq:dubrovnik}) skein relation with $z = B/A$.

    \end{enumerate}
\end{proof}

This argument for the Dubrovnik polynomial is similar in spirit to Kauffman's original description in  \cite{MR958895}, and the argument for HOMFLYPT polynomial is similar to \cite[\S 4]{MR0908150}.  Similar results were also obtained in \cite{MR2132671, MR1237835}.


We'll now need some notation for eigenvalues. Suppose $N$ appears once as a summand of $V \tensor V$, and consider the braiding as an endomorphism acting by composition on $\End{V \tensor V}$. Then the idempotent projecting onto $N \subset V \tensor V$ is an eigenvector for the braiding, and we'll write $\Reigenvalue{N}{V}$ for the corresponding eigenvalue.  The following is well-known (for example the HOMFLYPT case is essentially \cite[\S 4]{MR0908150}).

\begin{thm}
\label{thm:eigenvalues}
If one of conditions (2)-(4) of Theorem \ref{thm:summands} holds, then we can find which specialization occurs by computing eigenvalues.
    \begin{enumerate}
   \item[(2)] If $\Reigenvalue{L}{V} = \lambda$, then $\Reigenvalue{\id}{V} = -\lambda^{-3}$ and
        \begin{equation*}
        \J{\cC}{V} = \J{\SL{2}}{(1)}(a)
        \end{equation*}
    with $a=-\lambda^{2}$.
\item[(3)] If $\Reigenvalue{L}{V} = \lambda$, $\Reigenvalue{M}{V} = \mu$, and $\theta$ is the twist factor, then
        \begin{equation*}
        \J{\cC}{V} = \HOMFLY(a, z)
        \end{equation*}
    with $a=\frac{\theta}{\sqrt{-\lambda\mu}}$ and $z=\frac{\lambda+\mu}{\sqrt{-\lambda\mu}}$.
\item[(4)] If $\Reigenvalue{L}{V} = \lambda$ and $\Reigenvalue{M}{V} = \mu$, then $\lambda \mu = \pm 1$.
\begin{enumerate}
  \item
   If $\lambda \mu = -1$ then
        \begin{equation*}
        \J{\cC}{V} = \Dubrovnik(a, z)
        \end{equation*}
    with $a=\Reigenvalue{1}{V}^{-1}$ and $z=  \lambda + \mu$.
\item
 If $\lambda \mu = 1$ then
        \begin{equation*}
        \J{\cC}{V} = \Kauffman(a, z)
        \end{equation*}
    with $a=\Reigenvalue{\id}{V}^{-1}$ and $z=  \lambda + \mu$.
    \end{enumerate}
    \end{enumerate}
\end{thm}

\begin{proof}
These proofs all follow the same outline.  We consider the operator $X$ which acts on tangles with four boundary points by multiplication with a positive crossing.  We find the eigenvalues of $X$ in terms of the parameters ($a$ and/or $z$) and then solve for the parameters in terms of the eigenvalues.

    \begin{enumerate}
\item[(2)]
        The Jones skein relation for unoriented framed links is
       $$\overcrossing=i a^{\frac{1}{2}} \identity - i a^{-\frac{1}{2}} \cupcap$$
        if closed circles count for $[2]_a=(a+a^{-1})$.  
        
        The eigenvectors for $X$, multiplication by the positive crossing, are
        $$\begin{tikzpicture}[baseline]
    \draw(0,0) -- (.5,.5);
    \draw (0,0) -- (-.5,.5);
    \draw (0,0) -- (.5,-.5);
    \draw (0,0) -- (-.5,-.5);
    \node at (0,0) [rectangle,draw, fill=white] {$f^{(2)}$};
\end{tikzpicture}
\text{ \quad and \quad }
\begin{tikzpicture}[baseline]
    \draw (3.5,.5) .. controls (4,0) .. (4.5,.5);
    \draw (3.5,-.5) .. controls (4,0) .. (4.5,-.5);
\end{tikzpicture}
$$
        which have eigenvalues $i a^{\frac{1}{2}}$ and $ - i a^{\frac{-3}{2}}$.

        The cup-cap picture must correspond to the summand $\id$, and so we see that if
        $\Reigenvalue{L}{V}=\lambda$,
        then
        $a=-\lambda^{2} $ and $\Reigenvalue{\id}{V} = -\lambda^{-3}$.

    \item[(3)]
        The HOMFLYPT skein relation is for oriented framed links:
\begin{equation}
w \Oovercrossing - w^{-1} \Oundercrossing = z \Oidentity,
\end{equation}
        and the characteristic equation for the operator $X$ which multiplies by the positive crossing is
        $$w x - w^{-1} x^{-1} = z \quad \Longleftrightarrow \quad  x^2- \frac{z}{w} x - \frac{1}{w^2}=0.$$

        So if $\lambda$ and $\mu$ are the eigenvalues of $X$, we have $\lambda \mu = -w^{-2}$ and
        $\lambda + \mu = \frac z w$, so that
        $$w=\frac{1}{\sqrt{-\lambda \mu}} \text{\quad and \quad} z
        = \frac{\lambda + \mu}{\sqrt{-\lambda \mu}}$$
        
        To recover $a$ we note that the twist factor is $a w^{-1}$, hence $a = w \theta$.

    \item[(4)]
       For the Dubrovnik or Kauffman skein relation we have
       $$ \overcrossing \pm \undercrossing = z\left(\identity\pm \cupcap \right).$$  Multiplying by the crossing we see that, 
       $$\begin{tikzpicture}[baseline]
\node (x) at (0,.5){};
    \draw (x.45)-- (.5,1);
    \draw (x.135) -- (-.5,1);
    \draw (x.315) -- (.5,0);
    \draw (x.45) -- (-.5,0);
\node (y) at (0,-.5){};
    \draw (y.45)-- (.5,0);
    \draw (y.135) -- (-.5,0);
    \draw (y.315) -- (.5,-1);
    \draw (y.45) -- (-.5,-1);
\end{tikzpicture} \pm \identity = z \left(\overcrossing \pm a^{-1} \cupcap\right).$$  Subtracting $a^{-1}$ times the first equation from the second and rearranging slightly we see that the characteristic equation for the crossing operator is $(x-a^{-1})(x^2-z x\pm 1)=0$.  Hence the eigenvalues are $a^{-1}$, $\lambda$, and $\mu$, where $\lambda+\mu = z$ and $\lambda \mu = \pm 1$.  Since $a$ is the twist factor it is the inverse of the eigenvalue corresponding to $\id$ (compare with case (2)).
           
    \end{enumerate}
\end{proof}

\subsection{Knot polynomial identities for $\mathcal{D}_{4}$, $\mathcal{D}_{6}$, $\mathcal{D}_{8}$ and $\mathcal{D}_{10}$}

We state four theorems, give two lemmas, and then give rather pedestrian proofs of the theorems.  Snazzier proofs appear in Section \ref{sec:coincidences}, as special cases of Theorem \ref{thm:identities-geq}.  In each of these theorems, we relate two quantum knot invariants via an intermediate knot invariant coming from $\mathcal{D}_{2n}$.  You can think of these results as purely about quantum knot invariants, although the proofs certainly use $\mathcal{D}_{2n}$.

\begin{thm}[Identities for $n=2$]
\label{thm:identities-2}
\begin{align*}
\restrict{\J{\SL{2}}{(2)}(K)}{q=\exp(\frac{2\pi i}{12})} & = 2 \J{\mathcal{D}_4}{P}(K) \\
& = 2
\end{align*}
\end{thm}


\begin{thm}[Identities for $n=3$]
\label{thm:identities-3}
\begin{align*}
\restrict{\J{\SL{2}}{(4)}(K)}{q=\exp(\frac{2\pi i}{20})} & = 2 \J{\mathcal{D}_6}{P}(K) \\ & = 2 \restrict{\J{\SL{2}}{(1)}(K)}{q=\exp(- \frac{2\pi i}{10})}
\end{align*}
\end{thm}

\begin{thm}[Identities for $n=4$]
\label{thm:identities-4}
\begin{align*}
\restrict{\J{\SL{2}}{(6)}(K)}{q=\exp(\frac{2\pi i}{28})} & = 2 \J{\mathcal{D}_8}{P}(K) \\
                                     & = 2 \HOMFLY(K)(\exp(2\pi i\frac{3}{14}), \exp(\frac{2\pi i}{14}) - \exp(-\frac{2\pi i}{14})) \\
                & = 2 \HOMFLY(K)(\exp(2\pi i\frac{5}{7}), \exp(- \frac{2\pi i}{14}) - \exp(\frac{2\pi i}{14}))
\end{align*}
\end{thm}
\begin{rem}
This isn't just any specialization of the HOMFLYPT polynomial:
\begin{align*}
\HOMFLY(K)&(\exp(2\pi i\frac{3}{14}), \exp(\frac{2\pi i}{14}) - \exp(-\frac{2\pi i}{14})) \\
    & = \restrict{\HOMFLY(K)(q^3,q-q^{-1})}{q=\exp(\frac{2\pi i}{14})} \\
    & =  \restrict{\J{\SL{3}}{(1,0)}(K)}{q=\exp(\frac{2\pi i}{14})} \\
\intertext{and}
\HOMFLY(K)&(\exp(2\pi i\frac{5}{7}), \exp(- \frac{2\pi i}{14}) - \exp(\frac{2\pi i}{14}))    \\
    & = \restrict{\HOMFLY(K)(q^4,q-q^{-1})}{q=\exp(-\frac{2\pi i}{14})} \\
    & =  \restrict{\J{\SL{4}}{(1,0,0)}(K)}{q=\exp(-\frac{2\pi i}{14})} \\
    & =  - \restrict{\J{\SL{4}}{(1,0,0)}(K)}{q=-\exp(-\frac{2\pi i}{14})}
\end{align*}
(The last identity here follows from the fact that every exponent of $q$ in $\J{\SL{4}}{(1,0,0)}(K)$ is odd. We've included this form here to foreshadow \S \ref{sec:applications} where we'll give an independent proof of this theorem, and in which this particular value of $q=-\exp(-\frac{2\pi i}{14})$ will spontaneously appear.)
\end{rem}

\begin{thm}[Identities for $n=5$]
\label{thm:identities-5}
\begin{align*}
\restrict{\J{\SL{2}}{(8)}(K)}{q=\exp(\frac{2\pi i}{36})} & = 2 \J{\mathcal{D}_{10}}{P}(K) \\
        & = 2 \Dubrovnik(K) (\exp(2 \pi i \frac{4}{36}), \exp(2 \pi i \frac{2}{36}) + \exp(2 \pi i \frac{16}{36}))
\end{align*}
\end{thm}
\begin{rem}
Again, this isn't just any specialization of the Dubrovnik polynomial:
\begin{align*}
\Dubrovnik(K) (\exp(2 \pi i \frac{4}{36}),& \exp(2 \pi i \frac{2}{36}) + \exp(2 \pi i \frac{16}{36})) \\
& = \restrict{\Dubrovnik(K) (q^7,q-q^{-1})}{q=-\exp (\frac{-2 \pi i}{18})}\\
& = \restrict{\J{\SO{8}}{(1,0,0,0)}(K)}{q=-\exp(\frac{-2\pi i}{18})}.
\end{align*}
\end{rem}

For the proofs of these statements, we'll need to know how $P \tensor P$ decomposes in each $\mathcal{D}_{2n}$. The following formula was proved in \cite{MR1145672}.
\begin{equation}
\label{eq:P-squared}%
P \tensor P \iso \begin{cases}
Q \directSum \DirectSum_{l=0}^{\frac{n-4}{2}} f^{(4l+2)} & \text{when $n$ is even} \\
P \directSum \DirectSum_{l=0}^{\frac{n-3}{2}} f^{(4l)}   & \text{when $n$ is odd}
\end{cases}
\end{equation}
In particular,
\begin{align*}
P \tensor P & \iso Q                                        & \text{in $\mathcal{D}_{4}$,} \displaybreak[1] \\
P \tensor P & \iso P \directSum f^{(0)}                     & \text{in $\mathcal{D}_{6}$,} \displaybreak[1] \\
P \tensor P & \iso Q \directSum f^{(2)}                     & \text{in $\mathcal{D}_{8}$, and} \displaybreak[1] \\
P \tensor P & \iso P \directSum f^{(0)} \directSum f^{(4)}  & \text{in $\mathcal{D}_{10}$.}
\end{align*}

Further, we'll need a lemma calculating the eigenvalues of the braiding.
\begin{lem} \label{lem:eigenvalues}
Suppose $X$ is an idempotent in the set $\{ f^{(2)}, f^{(6)}, \ldots, f^{(2n-6)}, Q \}$ if $n$ is even, or $X \in \{ f^{(0)}, f^{(4)}, \ldots, f^{(2n-6)}, P \}$ if $n$ is odd. Then the eigenvalues for the braiding in $\mathcal{D}_{2n}$ are 
\begin{equation*}
R_{X \subset P \tensor P} =  (-1)^k q^{k(k+1)-2n(n-1)}
\end{equation*}
where $2k$ is the number of strands in the idempotent $X$.
\end{lem}

\begin{proof}
The endomorphism space for $P \tensor P$ is spanned by the projections onto the direct summands described above in Equation \eqref{eq:P-squared},
and thus by the diagrams
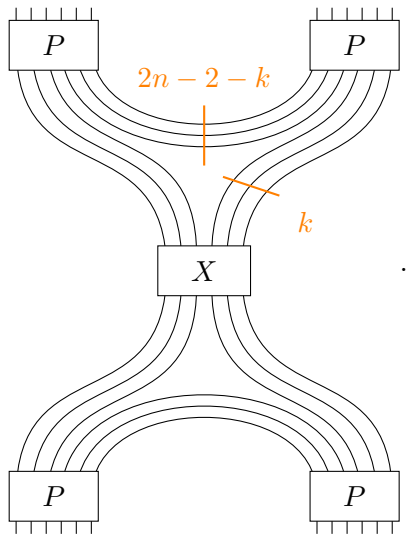
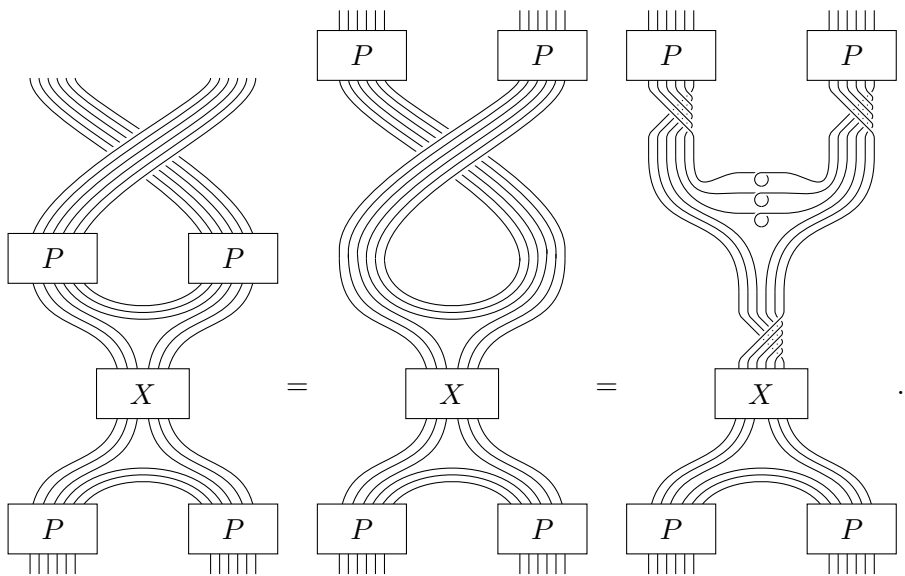
\begin{equation*}
\begin{tikzpicture}[inner sep=2mm, scale=.5, baseline=0]
	\draw (-1,0) .. controls ++(0,4) and ++(0,-4) .. (-5,6);	
	\draw (-.6,.2) .. controls ++(0,4) and ++(0,-4) .. (-4.6,6.2);
	\draw (-.2,.4) .. controls ++(0,4) and ++(0,-4) .. (-4.2,6.4);

	\draw (1,0) .. controls ++(0,4) and ++(0,-4) .. (5,6);	
	\draw (.6,.2) .. controls ++(0,4) and ++(0,-4) .. (4.6,6.2);
	\draw (.2,.4) .. controls ++(0,4) and ++(0,-4) .. (4.2,6.4);
	
	\draw (-3.8,6) .. controls ++(0,-3.6) and ++(0,-3.6) .. (3.8,6);
	\draw (-3.4,6) .. controls ++(0,-3.2) and ++(0,-3.2) .. (3.4,6);
	\draw (-3,6) .. controls ++(0,-2.8) and ++(0,-2.8) .. (3,6);

	\begin{scope}[yscale=-1]
		\draw (-1,0) .. controls ++(0,4) and ++(0,-4) .. (-5,6);	
		\draw (-.6,.2) .. controls ++(0,4) and ++(0,-4) .. (-4.6,6.2);
		\draw (-.2,.4) .. controls ++(0,4) and ++(0,-4) .. (-4.2,6.4);

		\draw (1,0) .. controls ++(0,4) and ++(0,-4) .. (5,6);	
		\draw (.6,.2) .. controls ++(0,4) and ++(0,-4) .. (4.6,6.2);
		\draw (.2,.4) .. controls ++(0,4) and ++(0,-4) .. (4.2,6.4);
	
		\draw (-3.8,6) .. controls ++(0,-3.6) and ++(0,-3.6) .. (3.8,6);
		\draw (-3.4,6) .. controls ++(0,-3.2) and ++(0,-3.2) .. (3.4,6);
		\draw (-3,6) .. controls ++(0,-2.8) and ++(0,-2.8) .. (3,6);
	\end{scope}
	
	\draw (-5,-6)--+(0,-1);
	\draw (-4.6,-6)--+(0,-1);
	\draw (-4.2,-6)--+(0,-1);
	\draw (-3.8,-6)--+(0,-1);
	\draw (-3.4,-6)--+(0,-1);
	\draw (-3,-6)--+(0,-1);

	\draw (5,-6)--+(0,-1);
	\draw (4.6,-6)--+(0,-1);
	\draw (4.2,-6)--+(0,-1);
	\draw (3.8,-6)--+(0,-1);
	\draw (3.4,-6)--+(0,-1);
	\draw (3,-6)--+(0,-1);
	
	\draw (-5,6)--+(0,1);
	\draw (-4.6,6)--+(0,1);
	\draw (-4.2,6)--+(0,1);
	\draw (-3.8,6)--+(0,1);
	\draw (-3.4,6)--+(0,1);
	\draw (-3,6)--+(0,1);

	\draw (5,6)--+(0,1);
	\draw (4.6,6)--+(0,1);
	\draw (4.2,6)--+(0,1);
	\draw (3.8,6)--+(0,1);
	\draw (3.4,6)--+(0,1);
	\draw (3,6)--+(0,1);

	\node[rectangle, fill=white, draw] at (0,0) {$\; \; X \; \;$};
	\node[rectangle, fill=white, draw] at (-4,6) {$\; \; P \; \;$};
	\node[rectangle, fill=white, draw] at (4,6) {$\; \; P \; \;$};
	\node[rectangle, fill=white, draw] at (-4,-6) {$\; \; P \; \;$};
	\node[rectangle, fill=white, draw] at (4,-6) {$\; \; P \; \;$};
	
	\draw[thick, orange] (.5,2.5) -- (2,2) node[below right] {$k$};
	\draw[thick, orange] (0,2.8) -- (0,4.4) node[above] {$2n-2-k$};
\end{tikzpicture}.
\end{equation*}

We calculate

\begin{equation*}
\begin{tikzpicture}[inner sep=2mm, scale=.3,baseline=0]
	\draw (-1,0) .. controls ++(0,4) and ++(0,-4) .. (-5,6);	
	\draw (-.6,.2) .. controls ++(0,4) and ++(0,-4) .. (-4.6,6.2);
	\draw (-.2,.4) .. controls ++(0,4) and ++(0,-4) .. (-4.2,6.4);

	\draw (1,0) .. controls ++(0,4) and ++(0,-4) .. (5,6);	
	\draw (.6,.2) .. controls ++(0,4) and ++(0,-4) .. (4.6,6.2);
	\draw (.2,.4) .. controls ++(0,4) and ++(0,-4) .. (4.2,6.4);
	
	\draw (-3.8,6) .. controls ++(0,-3.6) and ++(0,-3.6) .. (3.8,6);
	\draw (-3.4,6) .. controls ++(0,-3.2) and ++(0,-3.2) .. (3.4,6);
	\draw (-3,6) .. controls ++(0,-2.8) and ++(0,-2.8) .. (3,6);

	\begin{scope}[yscale=-1]
		\draw (-1,0) .. controls ++(0,4) and ++(0,-4) .. (-5,6);	
		\draw (-.6,.2) .. controls ++(0,4) and ++(0,-4) .. (-4.6,6.2);
		\draw (-.2,.4) .. controls ++(0,4) and ++(0,-4) .. (-4.2,6.4);

		\draw (1,0) .. controls ++(0,4) and ++(0,-4) .. (5,6);	
		\draw (.6,.2) .. controls ++(0,4) and ++(0,-4) .. (4.6,6.2);
		\draw (.2,.4) .. controls ++(0,4) and ++(0,-4) .. (4.2,6.4);
	
		\draw (-3.8,6) .. controls ++(0,-3.6) and ++(0,-3.6) .. (3.8,6);
		\draw (-3.4,6) .. controls ++(0,-3.2) and ++(0,-3.2) .. (3.4,6);
		\draw (-3,6) .. controls ++(0,-2.8) and ++(0,-2.8) .. (3,6);
	\end{scope}

	\begin{scope}[xscale=-1]
	\draw (-5,6.4) .. controls ++(0,3) and ++(0,-1) .. (3,14);
	\draw (-4.6,6.3) .. controls ++(0,3) and ++(0,-1.2) .. (3.4,14);
	\draw (-4.2,6.2) .. controls ++(0,3) and ++(0,-1.4) ..  (3.8,14);
	\draw (-3.8,6.1) .. controls ++(0,3) and ++(0,-1.6) ..  (4.2,14);
	\draw (-3.4,6) .. controls ++(0,3) and ++(0,-1.8) ..  (4.6,14);
	\draw (-3,5.9) .. controls ++(0,3) and ++(0,-2) ..  (5,14);
	\end{scope}

	\draw[white,line width=3pt] (-5,6.4) .. controls ++(0,3) and ++(0,-1) .. (3,14);
	\draw[white,line width=3pt] (-4.6,6.3) .. controls ++(0,3) and ++(0,-1.2) .. (3.4,14);
	\draw[white,line width=3pt] (-4.2,6.2) .. controls ++(0,3) and ++(0,-1.4) ..  (3.8,14);
	\draw[white,line width=3pt] (-3.8,6.1) .. controls ++(0,3) and ++(0,-1.6) ..  (4.2,14);
	\draw[white,line width=3pt] (-3.4,6) .. controls ++(0,3) and ++(0,-1.8) ..  (4.6,14);
	\draw[white,line width=3pt] (-3,5.9) .. controls ++(0,3) and ++(0,-2) ..  (5,14);

	\draw (-5,6.4) .. controls ++(0,3) and ++(0,-1) .. (3,14);
	\draw (-4.6,6.3) .. controls ++(0,3) and ++(0,-1.2) .. (3.4,14);
	\draw (-4.2,6.2) .. controls ++(0,3) and ++(0,-1.4) ..  (3.8,14);
	\draw (-3.8,6.1) .. controls ++(0,3) and ++(0,-1.6) ..  (4.2,14);
	\draw (-3.4,6) .. controls ++(0,3) and ++(0,-1.8) ..  (4.6,14);
	\draw (-3,5.9) .. controls ++(0,3) and ++(0,-2) ..  (5,14);

	\draw (-5,-6)--+(0,-2);
	\draw (-4.6,-6)--+(0,-2);
	\draw (-4.2,-6)--+(0,-2);
	\draw (-3.8,-6)--+(0,-2);
	\draw (-3.4,-6)--+(0,-2);
	\draw (-3,-6)--+(0,-2);

	\draw (5,-6)--+(0,-2);
	\draw (4.6,-6)--+(0,-2);
	\draw (4.2,-6)--+(0,-2);
	\draw (3.8,-6)--+(0,-2);
	\draw (3.4,-6)--+(0,-2);
	\draw (3,-6)--+(0,-2);

	\node[rectangle, fill=white, draw] at (0,0) {$\; \; X \; \;$};
	\node[rectangle, fill=white, draw] at (-4,6) {$\; \; P \; \;$};
	\node[rectangle, fill=white, draw] at (4,6) {$\; \; P \; \;$};
	\node[rectangle, fill=white, draw] at (-4,-6) {$\; \; P \; \;$};
	\node[rectangle, fill=white, draw] at (4,-6) {$\; \; P \; \;$};

\end{tikzpicture} =  \begin{tikzpicture}[inner sep=2mm, scale=.3,baseline=0]
	\draw (-1,0) .. controls ++(0,4) and ++(0,-4) .. (-5,6);	
	\draw (-.6,.2) .. controls ++(0,4) and ++(0,-4) .. (-4.6,6.2);
	\draw (-.2,.4) .. controls ++(0,4) and ++(0,-4) .. (-4.2,6.4);

	\draw (1,0) .. controls ++(0,4) and ++(0,-4) .. (5,6);	
	\draw (.6,.2) .. controls ++(0,4) and ++(0,-4) .. (4.6,6.2);
	\draw (.2,.4) .. controls ++(0,4) and ++(0,-4) .. (4.2,6.4);
	
	\draw (-3.8,6) .. controls ++(0,-3.6) and ++(0,-3.6) .. (3.8,6);
	\draw (-3.4,6) .. controls ++(0,-3.2) and ++(0,-3.2) .. (3.4,6);
	\draw (-3,6) .. controls ++(0,-2.8) and ++(0,-2.8) .. (3,6);

	\begin{scope}[yscale=-1]
		\draw (-1,0) .. controls ++(0,4) and ++(0,-4) .. (-5,6);	
		\draw (-.6,.2) .. controls ++(0,4) and ++(0,-4) .. (-4.6,6.2);
		\draw (-.2,.4) .. controls ++(0,4) and ++(0,-4) .. (-4.2,6.4);

		\draw (1,0) .. controls ++(0,4) and ++(0,-4) .. (5,6);	
		\draw (.6,.2) .. controls ++(0,4) and ++(0,-4) .. (4.6,6.2);
		\draw (.2,.4) .. controls ++(0,4) and ++(0,-4) .. (4.2,6.4);
	
		\draw (-3.8,6) .. controls ++(0,-3.6) and ++(0,-3.6) .. (3.8,6);
		\draw (-3.4,6) .. controls ++(0,-3.2) and ++(0,-3.2) .. (3.4,6);
		\draw (-3,6) .. controls ++(0,-2.8) and ++(0,-2.8) .. (3,6);
	\end{scope}

	\begin{scope}[xscale=-1]
	\draw (-5,6)--(-5,6.4) .. controls ++(0,3) and ++(0,-1) .. (3,14);
	\draw (-4.6,6)--(-4.6,6.3) .. controls ++(0,3) and ++(0,-1.2) .. (3.4,14);
	\draw (-4.2,6.2) .. controls ++(0,3) and ++(0,-1.4) ..  (3.8,14);
	\draw (-3.8,6)--(-3.8,6.1) .. controls ++(0,3) and ++(0,-1.6) ..  (4.2,14);
	\draw (-3.4,6) .. controls ++(0,3) and ++(0,-1.8) ..  (4.6,14);
	\draw (-3,5.9) .. controls ++(0,3) and ++(0,-2) ..  (5,14);
	\end{scope}

	\draw[white,line width=3pt] (-5,6.4) .. controls ++(0,3) and ++(0,-1) .. (3,14);
	\draw[white,line width=3pt] (-4.6,6.3) .. controls ++(0,3) and ++(0,-1.2) .. (3.4,14);
	\draw[white,line width=3pt] (-4.2,6.2) .. controls ++(0,3) and ++(0,-1.4) ..  (3.8,14);
	\draw[white,line width=3pt] (-3.8,6.1) .. controls ++(0,3) and ++(0,-1.6) ..  (4.2,14);
	\draw[white,line width=3pt] (-3.4,6) .. controls ++(0,3) and ++(0,-1.8) ..  (4.6,14);
	\draw[white,line width=3pt] (-3,5.9) .. controls ++(0,3) and ++(0,-2) ..  (5,14);

	\draw (-5,6)--(-5,6.4) .. controls ++(0,3) and ++(0,-1) .. (3,14);
	\draw (-4.6,6)--(-4.6,6.3) .. controls ++(0,3) and ++(0,-1.2) .. (3.4,14);
	\draw (-4.2,6.2) .. controls ++(0,3) and ++(0,-1.4) ..  (3.8,14);
	\draw (-3.8,6)--(-3.8,6.1) .. controls ++(0,3) and ++(0,-1.6) ..  (4.2,14);
	\draw (-3.4,6) .. controls ++(0,3) and ++(0,-1.8) ..  (4.6,14);
	\draw (-3,5.9) .. controls ++(0,3) and ++(0,-2) ..  (5,14);

	\draw (-5,-6)--+(0,-2);
	\draw (-4.6,-6)--+(0,-2);
	\draw (-4.2,-6)--+(0,-2);
	\draw (-3.8,-6)--+(0,-2);
	\draw (-3.4,-6)--+(0,-2);
	\draw (-3,-6)--+(0,-2);

	\draw (5,-6)--+(0,-2);
	\draw (4.6,-6)--+(0,-2);
	\draw (4.2,-6)--+(0,-2);
	\draw (3.8,-6)--+(0,-2);
	\draw (3.4,-6)--+(0,-2);
	\draw (3,-6)--+(0,-2);
	
	\draw (-5,15)--+(0,2);
	\draw (-4.6,15)--+(0,2);
	\draw (-4.2,15)--+(0,2);
	\draw (-3.8,15)--+(0,2);
	\draw (-3.4,15)--+(0,2);
	\draw (-3,15)--+(0,2);

	\draw (5,15)--+(0,2);
	\draw (4.6,15)--+(0,2);
	\draw (4.2,15)--+(0,2);
	\draw (3.8,15)--+(0,2);
	\draw (3.4,15)--+(0,2);
	\draw (3,15)--+(0,2);

	\node[rectangle, fill=white, draw] at (0,0) {$\; \; X \; \;$};
	\node[rectangle, fill=white, draw] at (-4,15) {$\; \; P \; \;$};
	\node[rectangle, fill=white, draw] at (4,15) {$\; \; P \; \;$};
	\node[rectangle, fill=white, draw] at (-4,-6) {$\; \; P \; \;$};
	\node[rectangle, fill=white, draw] at (4,-6) {$\; \; P \; \;$};
	
\end{tikzpicture} = \begin{tikzpicture}[inner sep=2mm, scale=.3,baseline=0]
\begin{scope}[yshift=4.5cm]
	\draw (-1,0) .. controls ++(0,4) and ++(0,-4) .. (-5,6);	
	\draw (-.6,.0) .. controls ++(0,4.2) and ++(0,-3.8) .. (-4.6,6);
	\draw (-.2,0) .. controls ++(0,4.4) and ++(0,-3.6) .. (-4.2,6);

	\draw (1,0) .. controls ++(0,4) and ++(0,-4) .. (5,6);	
	\draw (.6,.0) .. controls ++(0,4.2) and ++(0,-3.8) .. (4.6,6);
	\draw (.2,0) .. controls ++(0,4.4) and ++(0,-3.6) .. (4.2,6);
	
	\draw (-3.8,6) .. controls ++(0,-2) and ++(-1,0) .. (-1,3.5);
	\draw (3.8,6) .. controls ++(0,-2) and ++(1,0) .. (1,3.5);
	\begin{scope}[yshift=3.5cm]
		\draw (-1,0) -- (0,0) arc (90:-90:3mm);
		\draw[white, line width=2pt] (0,-.6) arc (-90:-270:3mm) -- (1,0);
		\draw(0,-.6) arc (-90:-270:3mm) -- (1,0);
	\end{scope}

	\draw (-3.4,6) .. controls ++(0,-2) and ++(-1,0) .. (-1,4.4);
	\draw (3.4,6) .. controls ++(0,-2) and ++(1,0) .. (1,4.4);
	\begin{scope}[yshift=4.4cm]
		\draw (-1,0) -- (0,0) arc (90:-90:3mm);
		\draw[white, line width=2pt] (0,-.6) arc (-90:-270:3mm) -- (1,0);
		\draw(0,-.6) arc (-90:-270:3mm) -- (1,0);
	\end{scope}

	\draw (-3,6) .. controls ++(0,-2) and ++(-1,0) .. (-1,5.3);
	\draw (3,6) .. controls ++(0,-2) and ++(1,0) .. (1,5.3);
	\begin{scope}[yshift=5.3cm]
		\draw (-1,0) -- (0,0) arc (90:-90:3mm);
		\draw[white, line width=2pt] (0,-.6) arc (-90:-270:3mm) -- (1,0);
		\draw(0,-.6) arc (-90:-270:3mm) -- (1,0);
	\end{scope}
\end{scope}

\begin{scope}[yscale=-1]
	\draw (-1,0) .. controls ++(0,4) and ++(0,-4) .. (-5,6);	
	\draw (-.6,.0) .. controls ++(0,4.2) and ++(0,-3.8) .. (-4.6,6);
	\draw (-.2,0) .. controls ++(0,4.4) and ++(0,-3.6) .. (-4.2,6);

	\draw (1,0) .. controls ++(0,4) and ++(0,-4) .. (5,6);	
	\draw (.6,.0) .. controls ++(0,4.2) and ++(0,-3.8) .. (4.6,6);
	\draw (.2,0) .. controls ++(0,4.4) and ++(0,-3.6) .. (4.2,6);
		
	\draw (-3.8,6) .. controls ++(0,-3.6) and ++(0,-3.6) .. (3.8,6);
	\draw (-3.4,6) .. controls ++(0,-3.2) and ++(0,-3.2) .. (3.4,6);
	\draw (-3,6) .. controls ++(0,-2.8) and ++(0,-2.8) .. (3,6);
\end{scope}

\begin{scope}[yshift=1.5cm,rounded corners =.5mm]
	\draw (1,-1)-- (1,0)--(-1,2)--(-1,3);
	\draw[white,line width=2pt] (.6,0)--(1,.4)--(-.6,2);
	\draw (.6,-1)--(.6,0)--(1,.4)--(-.6,2)--(-.6,3);
	\draw[white,line width=2pt] (.2,0)--(1,.8)--(-.2,2);
	\draw (.2,-1)--(.2,0)--(1,.8)--(-.2,2)--(-.2,3);
	\draw[white,line width=2pt] (-.2,0)--(1,1.2)--(.2,2);
	\draw (-.2,-1)--(-.2,0)--(1,1.2)--(.2,2)--(.2,3);
	\draw[white,line width=2pt] (-.6,0)--(1,1.6)--(.6,2);
	\draw (-.6,-1)--(-.6,0)--(1,1.6)--(.6,2)--(.6,3);
	\draw[white,line width=2pt] (-1,0)--(1,2);
	\draw (-1,-1)--(-1,0)--(1,2)--(1,3);
\end{scope}

\begin{scope}[yshift=13.5cm, xshift=-4cm, yscale=-1,rounded corners =.5mm]
	\draw (1,-1)-- (1,0)--(-1,2)--(-1,3);
	\draw[white,line width=2pt] (.6,0)--(1,.4)--(-.6,2);
	\draw (.6,-1)--(.6,0)--(1,.4)--(-.6,2)--(-.6,3);
	\draw[white,line width=2pt] (.2,0)--(1,.8)--(-.2,2);
	\draw (.2,-1)--(.2,0)--(1,.8)--(-.2,2)--(-.2,3);
	\draw[white,line width=2pt] (-.2,0)--(1,1.2)--(.2,2);
	\draw (-.2,-1)--(-.2,0)--(1,1.2)--(.2,2)--(.2,3);
	\draw[white,line width=2pt] (-.6,0)--(1,1.6)--(.6,2);
	\draw (-.6,-1)--(-.6,0)--(1,1.6)--(.6,2)--(.6,3);
	\draw[white,line width=2pt] (-1,0)--(1,2);
	\draw (-1,-1)--(-1,0)--(1,2)--(1,3);
\end{scope}

\begin{scope}[yshift=13.5cm, xshift=4cm, yscale=-1,rounded corners =.5mm]
	\draw (1,-1)-- (1,0)--(-1,2)--(-1,3);
	\draw[white,line width=2pt] (.6,0)--(1,.4)--(-.6,2);
	\draw (.6,-1)--(.6,0)--(1,.4)--(-.6,2)--(-.6,3);
	\draw[white,line width=2pt] (.2,0)--(1,.8)--(-.2,2);
	\draw (.2,-1)--(.2,0)--(1,.8)--(-.2,2)--(-.2,3);
	\draw[white,line width=2pt] (-.2,0)--(1,1.2)--(.2,2);
	\draw (-.2,-1)--(-.2,0)--(1,1.2)--(.2,2)--(.2,3);
	\draw[white,line width=2pt] (-.6,0)--(1,1.6)--(.6,2);
	\draw (-.6,-1)--(-.6,0)--(1,1.6)--(.6,2)--(.6,3);
	\draw[white,line width=2pt] (-1,0)--(1,2);
	\draw (-1,-1)--(-1,0)--(1,2)--(1,3);
\end{scope}

	\draw (-5,-6)--+(0,-2);
	\draw (-4.6,-6)--+(0,-2);
	\draw (-4.2,-6)--+(0,-2);
	\draw (-3.8,-6)--+(0,-2);
	\draw (-3.4,-6)--+(0,-2);
	\draw (-3,-6)--+(0,-2);

	\draw (5,-6)--+(0,-2);
	\draw (4.6,-6)--+(0,-2);
	\draw (4.2,-6)--+(0,-2);
	\draw (3.8,-6)--+(0,-2);
	\draw (3.4,-6)--+(0,-2);
	\draw (3,-6)--+(0,-2);
	
	\draw (-5,15)--+(0,2);
	\draw (-4.6,15)--+(0,2);
	\draw (-4.2,15)--+(0,2);
	\draw (-3.8,15)--+(0,2);
	\draw (-3.4,15)--+(0,2);
	\draw (-3,15)--+(0,2);

	\draw (5,15)--+(0,2);
	\draw (4.6,15)--+(0,2);
	\draw (4.2,15)--+(0,2);
	\draw (3.8,15)--+(0,2);
	\draw (3.4,15)--+(0,2);
	\draw (3,15)--+(0,2);

	\node[rectangle, fill=white, draw] at (0,0) {$\; \; X \; \;$};
	\node[rectangle, fill=white, draw] at (-4,15) {$\; \; P \; \;$};
	\node[rectangle, fill=white, draw] at (4,15) {$\; \; P \; \;$};
	\node[rectangle, fill=white, draw] at (-4,-6) {$\; \; P \; \;$};
	\node[rectangle, fill=white, draw] at (4,-6) {$\; \; P \; \;$};
	
\end{tikzpicture}.
\end{equation*}
Here there are negative half-twists on $2n-2$ strands below the top $P$s, and a positive half-twist on $2k$ strands above $X$. The $2n-2-k$ strands
connecting the two $P$s each have a negative kink. 

A positive half-twist on $\ell$ strands adjacent to an ``uncappable" element, such as a minimal projection, gives a factor of $(i s)^{\ell(\ell-1)/2}$, a negative half-twist on  $\ell$ strands adjacent to an uncappable element gives a factor of $(- i s^{-1})^{\ell(\ell-1)/2}$, and a negative kink gives a factor of $-i s^{-3}$. Remembering $q=s^2$, this shows that

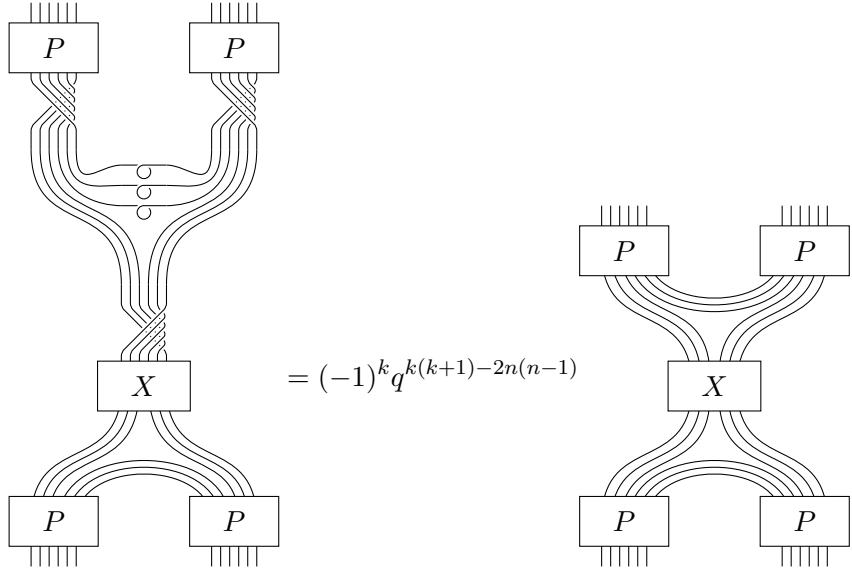
\begin{equation*}
 =  (-1)^k q^{k(k+1)-2n(n-1)}\begin{tikzpicture}[inner sep=2mm, scale=.3 ,baseline=0]
	\draw (-1,0) .. controls ++(0,4) and ++(0,-4) .. (-5,6);	
	\draw (-.6,.2) .. controls ++(0,4) and ++(0,-4) .. (-4.6,6.2);
	\draw (-.2,.4) .. controls ++(0,4) and ++(0,-4) .. (-4.2,6.4);

	\draw (1,0) .. controls ++(0,4) and ++(0,-4) .. (5,6);	
	\draw (.6,.2) .. controls ++(0,4) and ++(0,-4) .. (4.6,6.2);
	\draw (.2,.4) .. controls ++(0,4) and ++(0,-4) .. (4.2,6.4);
	
	\draw (-3.8,6) .. controls ++(0,-3.6) and ++(0,-3.6) .. (3.8,6);
	\draw (-3.4,6) .. controls ++(0,-3.2) and ++(0,-3.2) .. (3.4,6);
	\draw (-3,6) .. controls ++(0,-2.8) and ++(0,-2.8) .. (3,6);

	\begin{scope}[yscale=-1]
		\draw (-1,0) .. controls ++(0,4) and ++(0,-4) .. (-5,6);	
		\draw (-.6,.2) .. controls ++(0,4) and ++(0,-4) .. (-4.6,6.2);
		\draw (-.2,.4) .. controls ++(0,4) and ++(0,-4) .. (-4.2,6.4);

		\draw (1,0) .. controls ++(0,4) and ++(0,-4) .. (5,6);	
		\draw (.6,.2) .. controls ++(0,4) and ++(0,-4) .. (4.6,6.2);
		\draw (.2,.4) .. controls ++(0,4) and ++(0,-4) .. (4.2,6.4);
	
		\draw (-3.8,6) .. controls ++(0,-3.6) and ++(0,-3.6) .. (3.8,6);
		\draw (-3.4,6) .. controls ++(0,-3.2) and ++(0,-3.2) .. (3.4,6);
		\draw (-3,6) .. controls ++(0,-2.8) and ++(0,-2.8) .. (3,6);
	\end{scope}

	\draw (-5,-6)--+(0,-2);
	\draw (-4.6,-6)--+(0,-2);
	\draw (-4.2,-6)--+(0,-2);
	\draw (-3.8,-6)--+(0,-2);
	\draw (-3.4,-6)--+(0,-2);
	\draw (-3,-6)--+(0,-2);

	\draw (5,-6)--+(0,-2);
	\draw (4.6,-6)--+(0,-2);
	\draw (4.2,-6)--+(0,-2);
	\draw (3.8,-6)--+(0,-2);
	\draw (3.4,-6)--+(0,-2);
	\draw (3,-6)--+(0,-2);
	
	\draw (-5,6)--+(0,2);
	\draw (-4.6,6)--+(0,2);
	\draw (-4.2,6)--+(0,2);
	\draw (-3.8,6)--+(0,2);
	\draw (-3.4,6)--+(0,2);
	\draw (-3,6)--+(0,2);

	\draw (5,6)--+(0,2);
	\draw (4.6,6)--+(0,2);
	\draw (4.2,6)--+(0,2);
	\draw (3.8,6)--+(0,2);
	\draw (3.4,6)--+(0,2);
	\draw (3,6)--+(0,2);

	\node[rectangle, fill=white, draw] at (0,0) {$\; \; X \; \;$};
	\node[rectangle, fill=white, draw] at (-4,6) {$\; \; P \; \;$};
	\node[rectangle, fill=white, draw] at (4,6) {$\; \; P \; \;$};
	\node[rectangle, fill=white, draw] at (-4,-6) {$\; \; P \; \;$};
	\node[rectangle, fill=white, draw] at (4,-6) {$\; \; P \; \;$};
	
\end{tikzpicture}
\end{equation*}

Thus
\begin{equation*}
R_{X \subset P \tensor P} = (-1)^k q^{k(k+1)-2n(n-1)}. \qedhere
\end{equation*}

%
\end{proof}

\begin{proof}[Proof of Theorem \ref{thm:identities-2}]
In $\mathcal{D}_4$, $P \otimes P \cong Q$, so part one of Theorem \ref{thm:summands} applies.  Furthermore $\dim P = 1$, so the unframed invariant for the object $P$ in $\mathcal{D}_4$ is trivial. The first equation is just Theorem \ref{thm:half}.
\end{proof}

The same argument yields a previously known identity \cite{MR1134131}. Consider Temperley-Lieb at $q=\exp({\frac{2 \pi i}{6}})$, and notice that $\JW{1} \otimes \JW{1} \cong \JW{0}$ and $\dim \JW{1} = 1$. Thus $\restrict{\J{\SL{2}}{(1)}(K)}{q=\exp(\frac{2\pi i}{6})} = 1$.

\begin{proof}[Proof of Theorem \ref{thm:identities-3}]
In $\mathcal{D}_{6}$, we have that $P \tensor P \iso P \directSum f^{(0)}$, so part two of Theorem \ref{thm:summands} applies, and we know $\J{\mathcal{D}_6}{P}(K)$ is some specialization of the Jones polynomial. Using Lemma \ref{lem:eigenvalues}, we compute the two eigenvalues as
$$\Reigenvalue{f^{(0)}}{P}  = \exp({\frac{2 \pi i}{20}})^{-12} = -(\exp({2 \pi i \frac{-3}{10}}))^{-3}$$
and
$$\Reigenvalue{P}{P}  = \exp({\frac{2 \pi i}{20}})^{-6}  = \exp({2 \pi i\frac{-3}{10}})$$
which is consistent with $\Reigenvalue{P}{P}=\lambda$ and $\Reigenvalue{f^{(0)}}{P}=-\lambda^{-3}$.  So, we conclude that $a=-\lambda^2=-\exp({-6\frac{2 \pi i}{10}}) = \exp({-\frac{2 \pi i}{10}})$.
\end{proof}

\begin{proof}[Proof of Theorem \ref{thm:identities-4}]
In much the same way, for $\mathcal{D}_8$ we have $P \tensor P \iso Q \directSum f^{(2)}$, so $\J{\mathcal{D}_8}{P}(K)$ is some
specialization of the HOMFLYPT polynomial. The eigenvalues are
\begin{align*}
\lambda = \Reigenvalue{f^{(2)}}{P} & =  \exp( 2\pi i \frac{10}{14})  \\
\intertext{and}
\mu = \Reigenvalue{Q}{P} & =  \exp(2\pi i\frac{1}{14}),
\end{align*}
so $\frac{1}{\sqrt{-\lambda \mu}} = \pm \exp\left(2\pi i \frac{-2}{14}\right)$.
The twist factor is $\theta=\exp\left(2 \pi i \frac{-2}{14}\right)$, and so we get $\J{\mathcal{D}_8}{P}(K) = \HOMFLY(K)(a,z)$ with either
$$a= \exp(2 \pi i \frac{5}{7}),\, z=\exp(-2 \pi i \frac{1}{14}) - \exp(2 \pi i \frac{1}{14})$$
(taking the `positive' square root) or
$$a= \exp(2 \pi i \frac{3}{14}),\, z=\exp(2 \pi i \frac{1}{14}) - \exp(- 2 \pi i \frac{1}{14})$$
(taking the other).
\end{proof}

\begin{proof}[Proof of Theorem \ref{thm:identities-5}]
Again, in $\mathcal{D}_{10}$ we have $P \tensor P \iso P \directSum f^{(0)} \directSum f^{(4)}$, so $\J{\mathcal{D}_{10}}{P}(K)$ is a specialization of either the Kauffman polynomial or the Dubrovnik polynomial.
The eigenvalues are
\begin{align*}
a^{-1} &			 = \Reigenvalue{f^{(0)}}{P} 
                          = \exp(2 \pi i \frac{-1}{9}), \\
\lambda & = \Reigenvalue{f^{(4)}}{P} 
                          = \exp(2 \pi i \frac{1}{18}) \\
\intertext{and}
\mu & = \Reigenvalue{P}{P}
                    = \exp(2 \pi i \frac{4}{9})
\end{align*}

Now we apply Theorem \ref{thm:eigenvalues} (4) to these; we see that $\Reigenvalue{f^{(4)}}{P}\Reigenvalue{P}{P} = -1$, so we're in the Dubrovnik case. 
We read off $z =  \exp(2 \pi i \frac{1}{18}) + \exp(2 \pi i \frac{4}{9})$.

We've now shown that
\begin{align*}
\restrict{\J{\SL{2}}{(8)}(K)}{q=\exp(\frac{2\pi i}{36})} & = 2 \J{\mathcal{D}_{10}}{P}(K) \\
        & = 2 \Dubrovnik(K) (\exp(2 \pi i \frac{4}{36}), \exp(2 \pi i \frac{2}{36}) + \exp(2 \pi i \frac{16}{36})).
\end{align*}

To get the last identity, we note that
$$\restrict{(q^7,q-q^{-1})}{q=-\exp (\frac{-2 \pi i}{18})} = (\exp(2 \pi i \frac{4}{36}), \exp(2 \pi i \frac{16}{36}) + \exp(2 \pi i \frac{2}{36}))$$
and use the specialization appearing in Equation \eqref{eq:dubrovnik-D}.
\end{proof}

We remark that when $2n \geq 12$, Equation \eqref{eq:P-squared} shows that $P \tensor P$ has at least three summands which are not isomorphic to $f^{(0)}$,
and thus Theorem \ref{thm:summands} does not apply.

\subsection{Recognizing specializations of the $G_2$ knot invariant}

If $V$ is an object in a ribbon category such that $V \otimes V \cong \id \oplus V \oplus A \oplus B$ then it is reasonable to guess that the knot invariants coming from $V$ are specializations of the $G_2$ knot polynomial.  In particular $\mathcal{D}_{14}$ might be related to $G_2$, since in $\mathcal{D}_{14}$ we have $P \tensor P \iso f^{(0)} \directSum f^{(4)} \directSum f^{(8)} \directSum P$ by Equation \eqref{eq:P-squared}.  In this section we prove such a relationship using results of Kuperberg \cite{MR1265145}.  Applying Kuperberg's theorem requires some direct but tedious calculations.

In work in progress, Snyder has shown that, outside of a few small exceptions, all nontrivial knot invariants coming from tensor categories with $V \otimes V \cong \id \oplus V \oplus A \oplus B$ come from the $G_2$ link invariant, which would obviate the need for these calculations.  (The ``nontrivial" assumption in the last sentence is crucial as the standard representation of the symmetric group $S_n$, or more generally the standard object in Deligne's category $S_t$, also satisfies  $V \otimes V \cong \id \oplus V \oplus A \oplus B$.) 

In the following, by a {\em trivalent vertex} we mean a rotationally invariant map $V \tensor V \to V$ for some symmetrically self-dual object $V$. By a {\em tree} we mean a trivalent graph without cycles (allowing disjoint components).

\begin{thm}[{\cite[Theorem 2.1]{MR1265145}}] 
\label{thm:kuperberg}%
Suppose we have a symmetrically self-dual object $V$ and a trivalent vertex in a ribbon category $\cC$, such that trees with $5$ or fewer boundary points form a basis for the spaces $\Inv{\cC}{V^{\otimes k}}$ for $k \leq 5$. Then the link invariant $\J{\cC}{V}$ is a specialization of the $G_2$ link invariant for some $q$.
\end{thm}

\begin{rem}
The trivalent vertex in $\cC$ is some scalar multiple of the trivalent vertex in the $G_2$ spider.  Note that the $G_2$ link invariant is the same at $q$ and $-q$ since all the relations only depend on $q^2$.  
\end{rem}

\begin{lem} \label{lem:yucky}
Suppose that $\cC$ is a pivotal tensor category with a trivalent vertex such that trees form a basis of $\Inv{}{V^{\otimes k}}$ for $k \leq 3$.  Then
\begin{enumerate}
\item trees are linearly independent in $\Inv{}{V^{\otimes 4}}$ if and only if \begin{equation}-2 b^4 d^5 + b^4 d^6 - 2 b^3 d^4 t + (b^2 d^4 - b^2 d^6) t^2 \neq 0,\end{equation}
\item trees are linearly independent in $\Inv{}{V^{\otimes 5}}$ if and only if
\begin{align*}
& b^{20} \left(d^{15}-10 d^{13}-5 d^{12}+65 d^{11}-62 d^{10}\right) \\
& +5 b^{19} t\left(d^{14}+d^{13}-7 d^{12}-d^{11}+10 d^{10}\right) & \displaybreak[1]  \\
& -5 b^{18} t^2\left(d^{15}-10 d^{13}-3d^{12}+55 d^{11}-61 d^{10}\right) & \\
& -5 b^{17} t^3\left(6 d^{14}+7 d^{13}-40 d^{12}-41 d^{11}+83 d^{10}\right) & \displaybreak[1]  \\
& +5 b^{16} t^4\left(2 d^{15}+3 d^{14}-15 d^{13}-17 d^{12}+72d^{11}-68 d^{10}\right) & \\
& +b^{15} t^5\left(2 d^{15}+60 d^{14}+60 d^{13}-405 d^{12}-485 d^{11}+930 d^{10}\right) & \displaybreak[1]  \\
& -5 b^{14} t^6\left(3 d^{15}+12 d^{14}-8 d^{13}-64 d^{12}+3 d^{11}+71 d^{10}\right) & \\
& -5 b^{13} t^7\left(5 d^{14}+5 d^{13}-44 d^{12}-50 d^{11}+96 d^{10}\right) & \displaybreak[1]  \\
& +5 b^{12} t^8\left(3 d^{15}+12 d^{14}-6 d^{13}-70 d^{12}-17 d^{11}+112d^{10}\right)& \\
& -5 b^{11} t^8\left(2 d^{15}+6 d^{14}-5 d^{13}-29 d^{12}+4 d^{11}+45 d^{10}\right)& \displaybreak[1]  \\
& +b^{10} t^{10}\left(2 d^{15}+5 d^{14}-5 d^{13}-20 d^{12}+10 d^{11}+33 d^{10}\right) &\\ & \neq 0
\end{align*}
\end{enumerate}

Where $d$, $b$ and $t$ are defined by
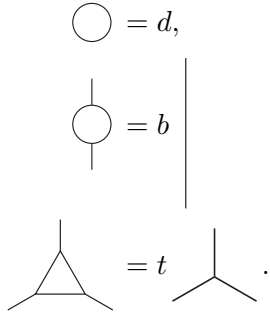
\begin{align*}
\mathfig{0.04}{G2/loop} &= d,\\
\mathfig{0.04}{G2/bigon} &=  b \;\;
\begin{tikzpicture}[baseline=-.1] \draw (0,-1)--(0,1); \end{tikzpicture} \\
\mathfig{0.1}{G2/triangle} &= t \mathfig{0.08}{G2/vertex}.
\end{align*}
\end{lem}
\begin{proof}
Compute the matrix of inner products between trees.  Each of these inner products can be calculated using only the relations for removing circles, bigons, and triangles.  If the determinant of this matrix is nonzero then the trees are linearly independent.
\end{proof}

For $\mathcal{D}_{2n}$ the single strand corresponds to $P$,  and the trivalent vertex is
$$
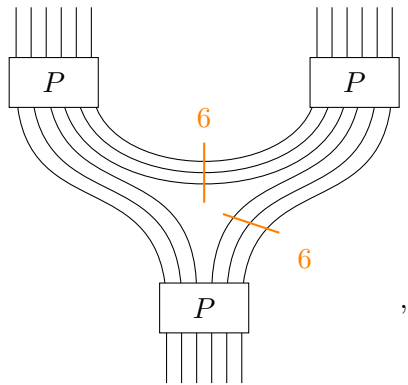
\begin{tikzpicture}[inner sep=2mm, scale=.5, baseline=0]
	\draw (-1,0) .. controls ++(0,4) and ++(0,-4) .. (-5,6);	
	\draw (-.6,.2) .. controls ++(0,4) and ++(0,-4) .. (-4.6,6.2);
	\draw (-.2,.4) .. controls ++(0,4) and ++(0,-4) .. (-4.2,6.4);

	\draw (1,0) .. controls ++(0,4) and ++(0,-4) .. (5,6);	
	\draw (.6,.2) .. controls ++(0,4) and ++(0,-4) .. (4.6,6.2);
	\draw (.2,.4) .. controls ++(0,4) and ++(0,-4) .. (4.2,6.4);
	
	\draw (-3.8,6) .. controls ++(0,-3.6) and ++(0,-3.6) .. (3.8,6);
	\draw (-3.4,6) .. controls ++(0,-3.2) and ++(0,-3.2) .. (3.4,6);
	\draw (-3,6) .. controls ++(0,-2.8) and ++(0,-2.8) .. (3,6);

	\draw (-5,6)--+(0,2);
	\draw (-4.6,6)--+(0,2);
	\draw (-4.2,6)--+(0,2);
	\draw (-3.8,6)--+(0,2);
	\draw (-3.4,6)--+(0,2);
	\draw (-3,6)--+(0,2);

	\draw (5,6)--+(0,2);
	\draw (4.6,6)--+(0,2);
	\draw (4.2,6)--+(0,2);
	\draw (3.8,6)--+(0,2);
	\draw (3.4,6)--+(0,2);
	\draw (3,6)--+(0,2);

	\draw (-1,0)--+(0,-2);
	\draw (-0.6,0)--+(0,-2);
	\draw (-0.2,0)--+(0,-2);
	\draw (0.2,0)--+(0,-2);
	\draw (0.6,0)--+(0,-2);
	\draw (1,0)--+(0,-2);

	\node[rectangle, fill=white, draw] at (0,0) {$\; \; P \; \;$};
	\node[rectangle, fill=white, draw] at (-4,6) {$\; \; P \; \;$};
	\node[rectangle, fill=white, draw] at (4,6) {$\; \; P \; \;$};
	
	\draw[thick, orange] (.5,2.5) -- (2,2) node[below right] {$6$};
	\draw[thick, orange] (0,2.8) -- (0,4.4) node[above] {$6$};
\end{tikzpicture},$$ which is rotationally invariant, because $P$ is invariant under $180$-degree rotation.

In order to apply Lemma \ref{lem:yucky} we must compute the values of $b$ and $t$ in $\frac{1}{2} \mathcal{D}_{14}$.  In order to do so we simplify the expression for the trivalent vertex.

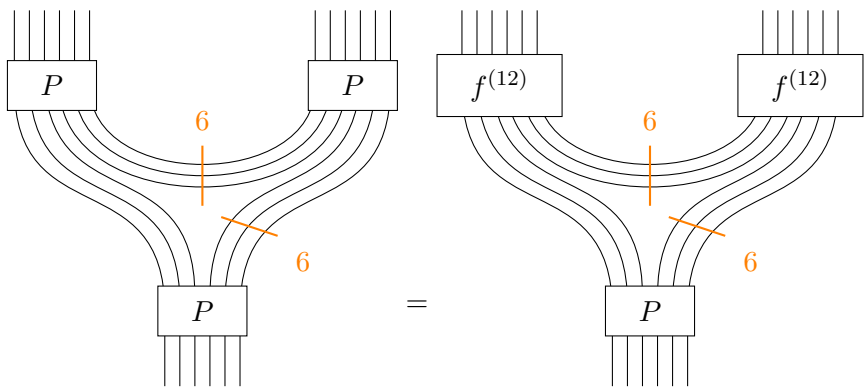
\begin{lem}
$$\begin{tikzpicture}[inner sep=2mm, scale=.5, baseline=0]
	\draw (-1,0) .. controls ++(0,4) and ++(0,-4) .. (-5,6);	
	\draw (-.6,.2) .. controls ++(0,4) and ++(0,-4) .. (-4.6,6.2);
	\draw (-.2,.4) .. controls ++(0,4) and ++(0,-4) .. (-4.2,6.4);

	\draw (1,0) .. controls ++(0,4) and ++(0,-4) .. (5,6);	
	\draw (.6,.2) .. controls ++(0,4) and ++(0,-4) .. (4.6,6.2);
	\draw (.2,.4) .. controls ++(0,4) and ++(0,-4) .. (4.2,6.4);
	
	\draw (-3.8,6) .. controls ++(0,-3.6) and ++(0,-3.6) .. (3.8,6);
	\draw (-3.4,6) .. controls ++(0,-3.2) and ++(0,-3.2) .. (3.4,6);
	\draw (-3,6) .. controls ++(0,-2.8) and ++(0,-2.8) .. (3,6);

	\draw (-5,6)--+(0,2);
	\draw (-4.6,6)--+(0,2);
	\draw (-4.2,6)--+(0,2);
	\draw (-3.8,6)--+(0,2);
	\draw (-3.4,6)--+(0,2);
	\draw (-3,6)--+(0,2);

	\draw (5,6)--+(0,2);
	\draw (4.6,6)--+(0,2);
	\draw (4.2,6)--+(0,2);
	\draw (3.8,6)--+(0,2);
	\draw (3.4,6)--+(0,2);
	\draw (3,6)--+(0,2);

	\draw (-1,0)--+(0,-2);
	\draw (-0.6,0)--+(0,-2);
	\draw (-0.2,0)--+(0,-2);
	\draw (0.2,0)--+(0,-2);
	\draw (0.6,0)--+(0,-2);
	\draw (1,0)--+(0,-2);

	\node[rectangle, fill=white, draw] at (0,0) {$\; \; P \; \;$};
	\node[rectangle, fill=white, draw] at (-4,6) {$\; \; P \; \;$};
	\node[rectangle, fill=white, draw] at (4,6) {$\; \; P \; \;$};
	
	\draw[thick, orange] (.5,2.5) -- (2,2) node[below right] {$6$};
	\draw[thick, orange] (0,2.8) -- (0,4.4) node[above] {$6$};
\end{tikzpicture}=\begin{tikzpicture}[inner sep=2mm, scale=.5, baseline=0]
	\draw (-1,0) .. controls ++(0,4) and ++(0,-4) .. (-5,6);	
	\draw (-.6,.2) .. controls ++(0,4) and ++(0,-4) .. (-4.6,6.2);
	\draw (-.2,.4) .. controls ++(0,4) and ++(0,-4) .. (-4.2,6.4);

	\draw (1,0) .. controls ++(0,4) and ++(0,-4) .. (5,6);	
	\draw (.6,.2) .. controls ++(0,4) and ++(0,-4) .. (4.6,6.2);
	\draw (.2,.4) .. controls ++(0,4) and ++(0,-4) .. (4.2,6.4);
	
	\draw (-3.8,6) .. controls ++(0,-3.6) and ++(0,-3.6) .. (3.8,6);
	\draw (-3.4,6) .. controls ++(0,-3.2) and ++(0,-3.2) .. (3.4,6);
	\draw (-3,6) .. controls ++(0,-2.8) and ++(0,-2.8) .. (3,6);

	\draw (-5,6)--+(0,2);
	\draw (-4.6,6)--+(0,2);
	\draw (-4.2,6)--+(0,2);
	\draw (-3.8,6)--+(0,2);
	\draw (-3.4,6)--+(0,2);
	\draw (-3,6)--+(0,2);

	\draw (5,6)--+(0,2);
	\draw (4.6,6)--+(0,2);
	\draw (4.2,6)--+(0,2);
	\draw (3.8,6)--+(0,2);
	\draw (3.4,6)--+(0,2);
	\draw (3,6)--+(0,2);

	\draw (-1,0)--+(0,-2);
	\draw (-0.6,0)--+(0,-2);
	\draw (-0.2,0)--+(0,-2);
	\draw (0.2,0)--+(0,-2);
	\draw (0.6,0)--+(0,-2);
	\draw (1,0)--+(0,-2);

	\node[rectangle, fill=white, draw] at (0,0) {$\; \; P \; \;$};
	\node[rectangle, fill=white, draw] at (-4,6) {$\; \; \JW{12} \; \;$};
	\node[rectangle, fill=white, draw] at (4,6) {$\; \; \JW{12} \; \;$};
	
	\draw[thick, orange] (.5,2.5) -- (2,2) node[below right] {$6$};
	\draw[thick, orange] (0,2.8) -- (0,4.4) node[above] {$6$};
\end{tikzpicture}$$
\end{lem}
\begin{proof}
Expand $\JW{12} = P + Q$ and use the fact that $P \otimes Q$, $Q \otimes P$, and $Q \otimes Q$ do not have nonzero maps to $P$.
\end{proof}

\begin{lem}
In $\frac{1}{2} \mathcal{D}_{14}$, using the above trivalent vertex, we have that $d$ is the root of $x^6 - 3x^5 -6 x^4 +4 x^3 + 5x^2- x-1$ which is approximately $4.14811$, $b$ is the root of $x^6-12x^5-499x^4-2760x^3-397x^2+276x-1$ which is approximately $0.00364276 $ and $t$ is the root of $x^6 +136 x^5+5072x^4+53866x^3+13132x^2+721x+1$ which is approximately $-0.00142366 $.
\end{lem}
\begin{proof}
The formula for $d$ is just the dimension of $P$.

We use the alternate description of the trivalent vertex to reduce the calculation of $b$ and $t$ to a calculation in Temperley-Lieb  which we do using the formulas of  \cite{MR1280463}.

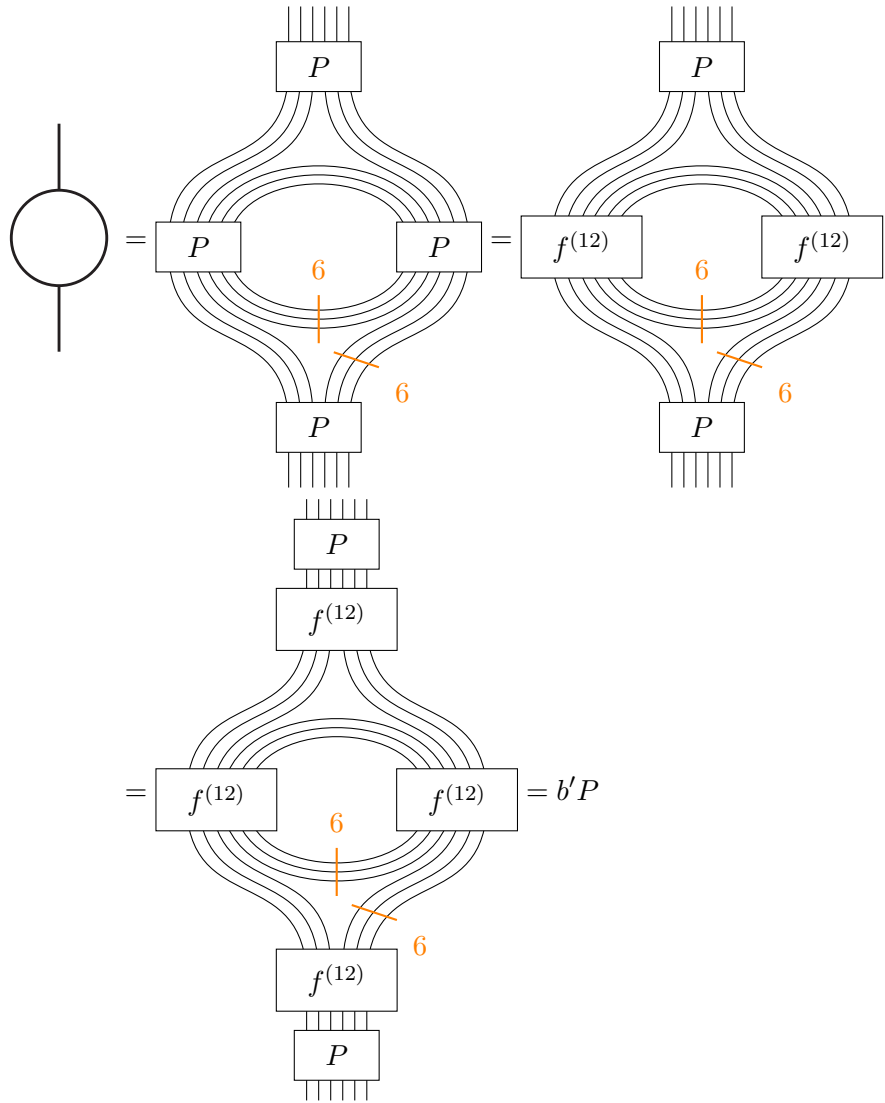
\begin{align*}
\mathfig{0.1}{G2/bigon} &= \begin{tikzpicture}[inner sep=2mm, scale=.4, baseline]
	\draw (-1,-6) .. controls ++(0,4) and ++(0,-4) .. (-5,0);	
	\draw (-.6,-5.8) .. controls ++(0,4) and ++(0,-4) .. (-4.6,.2);
	\draw (-.2,-5.6) .. controls ++(0,4) and ++(0,-4) .. (-4.2,.4);

	\draw (1,-6) .. controls ++(0,4) and ++(0,-4) .. (5,0);	
	\draw (.6,-5.8) .. controls ++(0,4) and ++(0,-4) .. (4.6,0.2);
	\draw (.2,-5.6) .. controls ++(0,4) and ++(0,-4) .. (4.2,0.4);
	
	\draw (-3.8,0) .. controls ++(0,-3.6) and ++(0,-3.6) .. (3.8,0);
	\draw (-3.4,0) .. controls ++(0,-3.2) and ++(0,-3.2) .. (3.4,0);
	\draw (-3,0) .. controls ++(0,-2.8) and ++(0,-2.8) .. (3,0);

	\draw (-1,6) .. controls ++(0,-4) and ++(0,4) .. (-5,-0);	
	\draw (-.6,5.8) .. controls ++(0,-4) and ++(0,4) .. (-4.6,-.2);
	\draw (-.2,5.6) .. controls ++(0,-4) and ++(0,4) .. (-4.2,-.4);

	\draw (1,6) .. controls ++(0,-4) and ++(0,4) .. (5,0);	
	\draw (.6,5.8) .. controls ++(0,-4) and ++(0,4) .. (4.6,-0.2);
	\draw (.2,5.6) .. controls ++(0,-4) and ++(0,4) .. (4.2,-0.4);
	
	\draw (-3.8,0) .. controls ++(0,3.6) and ++(0,3.6) .. (3.8,0);
	\draw (-3.4,0) .. controls ++(0,3.2) and ++(0,3.2) .. (3.4,0);
	\draw (-3,0) .. controls ++(0,2.8) and ++(0,2.8) .. (3,0);
	
	\draw (-1,-6)--+(0,-2);
	\draw (-0.6,-6)--+(0,-2);
	\draw (-0.2,-6)--+(0,-2);
	\draw (0.2,-6)--+(0,-2);
	\draw (0.6,-6)--+(0,-2);
	\draw (1,-6)--+(0,-2);

	\draw (-1,6)--+(0,2);
	\draw (-0.6,6)--+(0,2);
	\draw (-0.2,6)--+(0,2);
	\draw (0.2,6)--+(0,2);
	\draw (0.6,6)--+(0,2);
	\draw (1,6)--+(0,2);
		
	\node[rectangle, fill=white, draw] at (0,-6) {$\; \; P \; \;$};
	\node[rectangle, fill=white, draw] at (-4,0) {$\; \; P \; \;$};
	\node[rectangle, fill=white, draw] at (4,0) {$\; \; P \; \;$};
	\node[rectangle, fill=white, draw] at (0,6) {$\; \; P \; \;$};
		
	\draw[thick, orange] (.5,-3.5) -- (2,-4) node[below right] {$6$};
	\draw[thick, orange] (0,-3.2) -- (0,-1.6) node[above] {$6$};
\end{tikzpicture} = \begin{tikzpicture}[inner sep=2mm, scale=.4, baseline]
	\draw (-1,-6) .. controls ++(0,4) and ++(0,-4) .. (-5,0);	
	\draw (-.6,-5.8) .. controls ++(0,4) and ++(0,-4) .. (-4.6,.2);
	\draw (-.2,-5.6) .. controls ++(0,4) and ++(0,-4) .. (-4.2,.4);

	\draw (1,-6) .. controls ++(0,4) and ++(0,-4) .. (5,0);	
	\draw (.6,-5.8) .. controls ++(0,4) and ++(0,-4) .. (4.6,0.2);
	\draw (.2,-5.6) .. controls ++(0,4) and ++(0,-4) .. (4.2,0.4);
	
	\draw (-3.8,0) .. controls ++(0,-3.6) and ++(0,-3.6) .. (3.8,0);
	\draw (-3.4,0) .. controls ++(0,-3.2) and ++(0,-3.2) .. (3.4,0);
	\draw (-3,0) .. controls ++(0,-2.8) and ++(0,-2.8) .. (3,0);

	\draw (-1,6) .. controls ++(0,-4) and ++(0,4) .. (-5,-0);	
	\draw (-.6,5.8) .. controls ++(0,-4) and ++(0,4) .. (-4.6,-.2);
	\draw (-.2,5.6) .. controls ++(0,-4) and ++(0,4) .. (-4.2,-.4);

	\draw (1,6) .. controls ++(0,-4) and ++(0,4) .. (5,0);	
	\draw (.6,5.8) .. controls ++(0,-4) and ++(0,4) .. (4.6,-0.2);
	\draw (.2,5.6) .. controls ++(0,-4) and ++(0,4) .. (4.2,-0.4);
	
	\draw (-3.8,0) .. controls ++(0,3.6) and ++(0,3.6) .. (3.8,0);
	\draw (-3.4,0) .. controls ++(0,3.2) and ++(0,3.2) .. (3.4,0);
	\draw (-3,0) .. controls ++(0,2.8) and ++(0,2.8) .. (3,0);
	
	\draw (-1,-6)--+(0,-2);
	\draw (-0.6,-6)--+(0,-2);
	\draw (-0.2,-6)--+(0,-2);
	\draw (0.2,-6)--+(0,-2);
	\draw (0.6,-6)--+(0,-2);
	\draw (1,-6)--+(0,-2);

	\draw (-1,6)--+(0,2);
	\draw (-0.6,6)--+(0,2);
	\draw (-0.2,6)--+(0,2);
	\draw (0.2,6)--+(0,2);
	\draw (0.6,6)--+(0,2);
	\draw (1,6)--+(0,2);
		
	\node[rectangle, fill=white, draw] at (0,-6) {$\; \; P \; \;$};
	\node[rectangle, fill=white, draw] at (-4,0) {$\; \; \JW{12} \; \;$};
	\node[rectangle, fill=white, draw] at (4,0) {$\; \; \JW{12} \; \;$};
	\node[rectangle, fill=white, draw] at (0,6) {$\; \; P \; \;$};
		
	\draw[thick, orange] (.5,-3.5) -- (2,-4) node[below right] {$6$};
	\draw[thick, orange] (0,-3.2) -- (0,-1.6) node[above] {$6$};
\end{tikzpicture} \displaybreak[1] \\
&=\begin{tikzpicture}[inner sep=2mm, scale=.4, baseline]
	\draw (-1,-6) .. controls ++(0,4) and ++(0,-4) .. (-5,0);	
	\draw (-.6,-5.8) .. controls ++(0,4) and ++(0,-4) .. (-4.6,.2);
	\draw (-.2,-5.6) .. controls ++(0,4) and ++(0,-4) .. (-4.2,.4);

	\draw (1,-6) .. controls ++(0,4) and ++(0,-4) .. (5,0);	
	\draw (.6,-5.8) .. controls ++(0,4) and ++(0,-4) .. (4.6,0.2);
	\draw (.2,-5.6) .. controls ++(0,4) and ++(0,-4) .. (4.2,0.4);
	
	\draw (-3.8,0) .. controls ++(0,-3.6) and ++(0,-3.6) .. (3.8,0);
	\draw (-3.4,0) .. controls ++(0,-3.2) and ++(0,-3.2) .. (3.4,0);
	\draw (-3,0) .. controls ++(0,-2.8) and ++(0,-2.8) .. (3,0);

	\draw (-1,6) .. controls ++(0,-4) and ++(0,4) .. (-5,-0);	
	\draw (-.6,5.8) .. controls ++(0,-4) and ++(0,4) .. (-4.6,-.2);
	\draw (-.2,5.6) .. controls ++(0,-4) and ++(0,4) .. (-4.2,-.4);

	\draw (1,6) .. controls ++(0,-4) and ++(0,4) .. (5,0);	
	\draw (.6,5.8) .. controls ++(0,-4) and ++(0,4) .. (4.6,-0.2);
	\draw (.2,5.6) .. controls ++(0,-4) and ++(0,4) .. (4.2,-0.4);
	
	\draw (-3.8,0) .. controls ++(0,3.6) and ++(0,3.6) .. (3.8,0);
	\draw (-3.4,0) .. controls ++(0,3.2) and ++(0,3.2) .. (3.4,0);
	\draw (-3,0) .. controls ++(0,2.8) and ++(0,2.8) .. (3,0);
	
	\draw (-1,-6)--+(0,-4);
	\draw (-0.6,-6)--+(0,-4);
	\draw (-0.2,-6)--+(0,-4);
	\draw (0.2,-6)--+(0,-4);
	\draw (0.6,-6)--+(0,-4);
	\draw (1,-6)--+(0,-4);

	\draw (-1,6)--+(0,4);
	\draw (-0.6,6)--+(0,4);
	\draw (-0.2,6)--+(0,4);
	\draw (0.2,6)--+(0,4);
	\draw (0.6,6)--+(0,4);
	\draw (1,6)--+(0,4);
		
	\node[rectangle, fill=white, draw] at (0,-6) {$\; \; \JW{12} \; \;$};
	\node[rectangle, fill=white, draw] at (-4,0) {$\; \; \JW{12} \; \;$};
	\node[rectangle, fill=white, draw] at (4,0) {$\; \; \JW{12} \; \;$};
	\node[rectangle, fill=white, draw] at (0,6) {$\; \; \JW{12} \; \;$};
	\node[rectangle, fill=white, draw] at (0,-8.5) {$\; \; P \; \;$};
	\node[rectangle, fill=white, draw] at (0,8.5) {$\; \; P \; \;$};

	\draw[thick, orange] (.5,-3.5) -- (2,-4) node[below right] {$6$};
	\draw[thick, orange] (0,-3.2) -- (0,-1.6) node[above] {$6$};
\end{tikzpicture}= b' P 
\end{align*}

where $b'$ is the coefficient for removing bigons labelled with $\JW{12}$ in Temperley-Lieb.

The calculation for $t$ is similar: we replace each trivalent vertex with a trivalent vertex with a $P$ on the outside and $\JW{12}$s in the middle.  Then we reduce the inner triangle in Temperley-Lieb.
\end{proof}

\begin{thm}
\label{thm:G2-links}%
For $\ell = -3$ or $10$,
\begin{align*}\restrict{\J{\SL{2}}{(12)}(K)}{q=\exp(\frac{2\pi i}{52})} & = 2 \J{\mathcal{D}_{14}}{P}(K) \\
& = 2\restrict{\J{G_2}{V_{(10)}}(K)}{q=\exp{2\pi i \frac{\ell}{26}}} \\
\end{align*}
\end{thm}
\begin{proof}
Since $\dim \Inv{}{P^{\otimes 0}} = \dim \Inv{}{P^{\otimes 2}} = \dim \Inv{}{P^{\otimes 3}} = 1$ and $\dim \Inv{}{P} = 0$, trees form a basis of $\Inv{}{P^{\otimes k}}$ for $k \leq 3$.  By Lemma \ref{lem:yucky} we see that trees are linearly independent in $\Inv{}{P^{\otimes 4}}$ and $\Inv{}{P^{\otimes 5}}$.  A dimension count shows that trees form a basis for these spaces.   Now we apply Theorem \ref{thm:kuperberg} to see that the theorem holds for some $q$. 
We need to normalize the $\mathcal{D}_{14}$ trivalent vertex for $P$ before it satisfies the $G_2$ relations, specifically multiplying it by the largest real root of $x^{12} - 645 x^{10} -10928 x^8 - 32454 x^6 - 4752 x^4 + 2 x^2 +1$. The quantities $b$ and $t$ are both homogeneous of degree $2$ with respect to scaling the trivalent vertex, so they are both multiplied by the square of this quantity. We now solve the equations
\begin{align*}
d & = q^{10} + q^{8} + q^2 + 1 + q^{-2} + q^{-8} + q^{-10}  \\
b & = -\left(q^6 + q^4 + q^2 + q^{-2} + q^{-4} + q^{-6}\right) \\
t & = q^4 + 1 +q^{-4}
\end{align*}  
and find that they have a four solutions, $q= \exp{2\pi i \frac{\ell}{26}}$ with $\ell=\pm 3, \pm 10$. Not all of these give the correct twist factor, however. The twist factor for $P$ is $\exp(2 \pi i \frac{-10}{26})$, while the twist factor for the representation $V_{(10)}$ of $G_2$ is $q^{12}$; these only agree for $\ell = -3$ or $10$. Since the identity holds for some $q$, and the knot invariant only depends on $q^2$, the identity must hold for each of these values.
\end{proof}

\subsection{Ribbon functors} \label{sec:ribbonfunctors}
The proofs of Theorems \ref{thm:summands} , \ref{thm:eigenvalues}, and \ref{thm:kuperberg} actually construct ribbon functors from a certain diagrammatic category to the ribbon category $\cC$.  Combining this functor with the description of quantum group categories by diagrams in \cite{MR1710999, MR1854694} and \cite{MR1182414} one could prove the coincidences described in the introduction (that is, Theorems \ref{thm:D6-coincidence}, \ref{thm:D8-coincidence} and \ref{thm:D10-coincidence}).  To do this we need the following lemma.

\begin{lem}
Suppose that $\cC$ is a ribbon category such that $\cC^{ss}$ is premodular, that $\cD$ is a pseudo-unitary modular category, and that $\cF$ is a dominant ribbon functor $\cC \rightarrow \cD$. Then $\cD \cong \cC^{ss\ modularize}$.
\end{lem}
\begin{proof}
Since $\cD$ is pseudo-unitary the functor must factor through the semisimplification, and thus  the result follows from the uniqueness of modularization.
\end{proof}

In our cases $\cD_{2n}$ is the target category, and is certainly unitary and modular.  The source category is a category of diagrams (coming from Temperley-Lieb, Kauffman/Dubrovnik, HOMFLYPT, or the $G_2$ spider).   Dominance of the functor is a simple calculation in the fusion ring of $\mathcal{D}_{2n}$.  If $q$ is a large enough root of unity, then the semisimplification of that diagram category has been proven to be pre-modular for each of these cases \cite{MR1854694, MR2132671, MR1710999} except the $G_2$ spider.  Hence the argument of the last subsection does not yet give a proof of the $G_2$ coincidence.  We give a completely different proof in the next subsection.

\subsection{Recognizing $\mathcal{D}_{2n}$ modular categories}\label{sec:recognizeD}

Earlier in this section we found knot polynomial identities and coincidences of modular tensor categories by observing that $P^{\otimes 2}$ broke up in some particular way.  In this section we work in the reverse direction.  The category $\frac{1}{2}\mathcal{D}_{2n}$ has a small object $\JW{2}$ and $\JW{2} \tensor \JW{2} \cong \id \oplus \JW{2} \oplus \JW{4}$.  If we are to have a coincidence of modular tensor categories $\mathcal{D}_{2n} \cong \cC$ then there must be an object in $\cC$ which breaks up the same way.  Using the characterization of the Kauffman and Dubrovnik categories above we can prove that  $\mathcal{D}_{2n} \cong \cC$ by producing this object.  In the following theorem, we use this technique to show $\frac{1}{2} \mathcal{D}_{14} \cong \Rep{U_{\exp({2 \pi i\frac{\ell}{26}})}(\mathfrak{g}_2)}$, for $\ell = -3$ or $10$, sending $P \mapsto V_{(10)}$. It's also possible to prove Theorems \ref{thm:D6-coincidence}, \ref{thm:D8-coincidence} and \ref{thm:D10-coincidence} by this technique, although we don't do this.

\begin{thm}\label{thm:G2}
There is an equivalence of modular tensor categories 
$$\Rep{U_{\exp({2 \pi i\frac{\ell}{26}})}(\mathfrak{g}_2)}\cong \frac{1}{2} \mathcal{D}_{14},$$ 
where $\ell = -3$ or $10$,
sending $\JW{2} \mapsto V_{(02)}$.  Under this equivalence we also have $P \mapsto V_{(10)}$.
\end{thm}
\begin{proof}
Using the Racah rule for tensor products in $\Rep{U_{\exp({2 \pi i\frac{\ell}{26}})}(\mathfrak{g}_2)}$ we see that $V_{(02)}^{\otimes 2} \cong \id \oplus V_{(01)} \oplus V_{(02)} $.  

The eigenvalues for the square of a crossing can be read off from twist factors: $$R_{X \subset Y \tensor Y}^2 = \theta_X \theta_Y^{-2}.$$
The twist factors for the representations $V_{(00)}, V_{(01)}$ and $V_{(02)}$ are $1, q^{24}$ and $q^{60}$ respectively, so the corresponding eigenvalues for the crossing are $q^{-60}, \sigma_1 q^{-48}$ and $\sigma_2 q^{-30}$ for some signs $\sigma_1$ and $\sigma_2$. We thus compute, whether we are in the Kauffman or Dubrovnik settings, that  $a=q^{60}$ and $z=\sigma_1 q^{-48} + \sigma_2 q^{-30}$.

If we are in the Kauffman setting, we must have $\sigma_1 \sigma_2 q^{-78} = 1$, so $\sigma_1 = \sigma_2$. We now see the dimension formula $d=\frac{a+a^{-1}}{z} - 1$ can not be equal to 
$\dim(\JW{2})=[3]_{q=\exp(\frac{2 \pi i}{52})}$ for any choice of $\sigma_1, \sigma_2$.

Hence we must be in the Dubrovnik setting where we have $\sigma_1 = -\sigma_2$
and $d=\frac{a-a^{-1}}{z} + 1$.
Now the dimensions match up exactly when $\sigma_1 =-1$ and $\sigma_2 = 1$.

By \S \ref{sec:ribbonfunctors} and Theorems \ref{thm:summands} and \ref{thm:eigenvalues} we have a functor from the Dubrovnik category with $a= \exp(2 \pi i \frac{1}{13})$ and $z=\exp(2 \pi i \frac{1}{26}) - \exp(2 \pi i \frac{-1}{26})$ to $\Rep{U_{\exp({2 \pi i\frac{\ell}{26}})}(\mathfrak{g}_2)}$.   Since the target category is pseudo-unitary \cite{MR2414692}, this functor factors through the semi\-simpli\-fi\-cat\-ion of the diagram category, which is the premodular category $\Rep{U_{q=\exp(2 \pi i \frac{1}{52})}(\SO{3})}$.  Since the target is modular \cite{MR2286123} and the functor is dominant (a straightforward calculation via the Racah rule in the Grothendieck group of $\Rep{U_{\exp({2 \pi i\frac{\ell}{26}})}(\mathfrak{g}_2)}$) this functor induces an equivalence between the modularization of $\Rep{U_{q=\exp(2 \pi i \frac{1}{52})}(\SO{3})}$, which is nothing but $\frac{1}{2} \mathcal{D}_{14}$, and $\Rep{U_{\exp({2 \pi i\frac{\ell}{26}})}(\mathfrak{g}_2)}$.

The correspondence between simples shown in Figure \ref{fig:g2-weyl-chamber}, can be computed inductively.  Begin with the observation that $\JW{2}$ is sent to $V_{(02)}$ by construction; after that, everything else is determined by working out the tensor product rules in both categories.
\end{proof}

\begin{figure}[!ht]
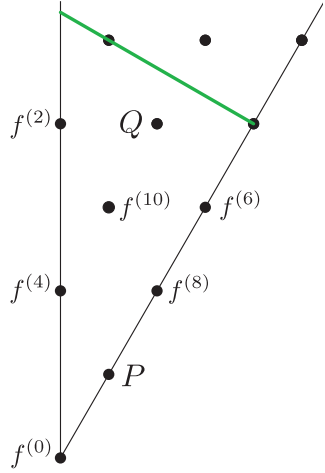

\begin{equation*}
\mathfig{0.3}{coincidences/g2-weyl-chamber}
\end{equation*}
\caption{The positive Weyl chamber for $G_2$, showing the surviving irreducible representations in the semisimple quotient at $q=\pm \exp({2\pi i\frac{-3}{26}})$, and the
correspondence with the even vertices of $\mathcal{D}_{14}$.}
\label{fig:g2-weyl-chamber}
\end{figure}

Note that $q=\pm \exp({2\pi i\frac{-3}{26}})$ corresponds to the fractional level $\frac{1}{3}$ of $G_2$ (see \cite{MR2286123}), which has previously been conjectured to be unitary \cite{MR2414692}. This theorem proves that conjecture.

Finally, we note that the same method gives an equivalence between 
$\Rep{U_{\pm \exp({2 \pi i\frac{-3}{28}})}(\mathfrak{g}_2)}$
and the subcategory of $\Rep{U_{\exp({2 \pi i\frac{1}{28}})}(\SP{6})}$ generated by the representation $V_{(012)}$, sending the  representation $V_{(12)}$ of $\mathfrak{g}_2$ to $V_{(012)}$. On the $\mathfrak{g}_2$ side, we have $V_{(12)}^{\tensor 2} \iso V_{(00)} \directSum V_{(01)} \directSum V_{(02)}$ with corresponding eigenvalues $1, \exp(2\pi i \frac{3}{14})$ and $-\exp(2\pi i \frac{4}{14})$. On the $\SP{6}$ side we have $V_{(012)}^{\tensor 2} \iso V_{(000)} \directSum V_{(010)} \directSum V_{(200)}$ with corresponding eigenvalues $1, \exp(2\pi i \frac{3}{14})$ and $-\exp(2\pi i \frac{4}{14})$. Thus both categories, which are each modular, are the modularization of the semisimplification of the Kauffman category at $a=1, z=\exp(2\pi i \frac{3}{14}) - \exp(2 \pi i \frac{4}{14})$. This proves the conjecture of \cite{MR2414692} that $G_2$ at level $\frac{2}{3}$ is also unitary.

\section{Coincidences of tensor categories}
\label{sec:coincidences}%

In the previous section we found identities between knot polynomials coming from (a priori) different ribbon categories.  In Section \ref{sec:ribbonfunctors} we showed that these identities must come from unexpected functors between these ribbon categories.  In this section we explain how these coincidences of tensor categories follow from general theory.  One should think of the results of this section as quantum analogs of small coincidences in group theory, such as $\operatorname{Alt}_5 \cong \mathbf{PSL}_2(\mathbb{F}_5)$.

There are three important sources of unexpected equivalences (or autoequivalences) between ribbon categories coming from quantum groups: coincidences of small Dynkin diagrams, (deequivariantization related to) generalized Kirby-Melvin symmetry, and level-rank duality.   

There are sometimes coincidences between Dynkin diagrams in different families. For instance, the Dynkin diagrams $A_3$ and $D_3$ are equal, from which it follows that $\SL{4} \iso \SO{6}$ and the associated categories of representations of quantum groups are equivalent too.

Kirby-Melvin symmetry relates link invariants coming from different objects in the same category, when that category has an invertible object.  
Under certain auspicious conditions, one can go further and deequivariantize by the invertible object.

Level-rank duality is a collection of equivalences relating $SU(n)_k$ with $SU(k)_n$, and relating $SO(n)_k$ with $SO(k)_n$,
where $SU(n)_k$ or $SO(n)_k$ refers to the semisimplified representation category of the rank-$n$ quantum group, at a carefully chosen root of unity which depends on the `level' $k$. In some sense, level-rank duality is more natural in the context of $U(n)$ and $O(n)$, and new difficulties arise formulating level-rank duality for the quantum groups $SU(n)$ and $SO(n)$. We give, in Theorem \ref{thm:SO3-duality}, a precise statement for $SO$ level-rank duality with $n=3$ and $k$ even.

We will discuss each of these three sources of unexpected equivalences in the following sections, and then use them to prove the following results:

\begin{thm}
\label{thm:D6-coincidence}
There is an equivalence of modular tensor categories $$\frac{1}{2} \mathcal{D}_{6} \cong \uRep{U_{s = \exp\left(\frac{7}{10} 2 \pi i\right)}(\SL{2} \oplus \SL{2}})^{modularize},$$ sending $P \mapsto V_{(1)} \boxtimes V_{(0)}$.
\end{thm}

\begin{thm}
\label{thm:D8-coincidence}
There is an equivalence of modular tensor categories $$\frac{1}{2} \mathcal{D}_{8} \cong \uRep{U_{s=\exp \left(\frac{5}{14} 2 \pi i\right)}(\SL{4})}^{modularize},$$ sending $P \mapsto V_{(100)}$.
\end{thm}

\begin{thm}
\label{thm:D10-coincidence}
The modular tensor category $\frac{1}{2} \mathcal{D}_{10}$ has an order $3$ automorphism, fixing $\JW{0}, \JW{4}$ and $\JW{6}$, and permuting $$\baselinetrick{\xymatrix@C-8mm@R-3mm{P \ar@{|->}@/^/[rr] & & Q \ar@{|->}@/^/[dl] \\ & f^{(2)} \ar@{|->}@/^/[ul]&}}.$$
\end{thm}

Finally we note that there are other coincidences of small tensor categories that do not follow from these general techniques.  In particular it would be very interesting to better explain the coincidences involving $G_2$.

\subsection{Dynkin diagram coincidences and quantum groups}

The definition of the quantum group and its ribbon category of representations depend only on the Dynkin diagram itself.  For the quantum group and its tensor category this is obvious from the presentation by generators and relations.  For the braiding and the ribbon structure this follows from the independence of choice of decomposition of the longest word in the Weyl group in the multiplicative formula for the R-matrix.

In particular, every coincidence between Dynkin diagrams lifts to a statement about the quantum groups.  We will use that $D_2 = A_1 \times A_1$, that $D_3 = A_3$, and that $D_4$ has triality symmetry. 

The reason these coincidences are useful is that they give two different diagrammatic presentations of the same ribbon category.  For example, the fact that $B_1 = A_1$ tells you that the even part of Temperley-Lieb can be described using the Dubrovnik category, which we used implicity in Section \ref{sec:recognizeD}. The only coincidence we don't use is $B_2 = C_2$. Since $B_2$ is the Dynkin diagram for $\SO{5}$, there is no relationship via level-rank duality with the $\mathcal{D}_{2n}$ planar algebras.

\subsection{Kirby-Melvin symmetry}\label{sec:KM}
Kirby-Melvin symmetry relates link invariants from one
representation of a quantum group to link invariants coming from another
representation which is symmetric to it under a
symmetry of the affine Weyl chamber.  This symmetry principle was proved
in type $A_1$ by Kirby and Melvin \cite{MR1117149}, in type $A_n$ by Kohno and 
Takata \cite{MR1227008}, and for a general quantum group by Le \cite{MR1749439}.  There is another proof in the type $A$ case, using conformal inclusions, due to Xu \cite{MR2124554}.
We give a diagrammatic  proof which generalizes this result to tensor categories which might not come from quantum groups.

Suppose that $\cC$ is a semi-simple ribbon category and that $X$ is an object which is invertible in the sense that $X \otimes X^* \cong \id$.  Kirby-Melvin symmetry relates link invariants coming from a simple object $A$ to invariants coming from the (automatically simple) object $A\otimes X$.  

The key observation is that, for any simple $A$, the objects $A
\otimes X$ and $X \otimes A$ are simple (since $\Hom{}{A \otimes X}{ A \otimes X}
= \Hom{}{A \otimes X \otimes X^*}{ A}$), so the Hom space between them is one dimensional.  Thus the over-crossing and under-crossing must be scalar multiples.  Define $c_A$ by the following formula,
$$\begin{tikzpicture}[baseline]
\node (x) at (0,0){};
    \draw[->] (x.45)-- (.5,.5);
    \draw[->] (x.135) -- (-.5,.5);
    \draw (x.225) -- (-.5,-.5);
    \draw (x.135) -- (.5,-.5);

\node at (-.75,-.75){$A$};
\node at (.75,-.75){$X$};
\end{tikzpicture} = c_A \begin{tikzpicture}[baseline]
\node (x) at (0,0){};
    \draw[->] (x.45)-- (.5,.5);
    \draw[->] (x.135) -- (-.5,.5);
    \draw (x.315) -- (.5,-.5);
    \draw (x.45) -- (-.5,-.5);

\node at (-.75,-.75){$A$};
\node at (.75,-.75){$X$};
\end{tikzpicture}.$$

Note that $c_A^{-1} \dim A \dim X = S_{XA}$ where $S$ is the $S$-matrix.  Using the formula for the square of the crossing in terms of the ribbon element, we see that $c_A = \frac{\theta_A \theta_X}{\theta_{A\otimes X}}$.

\begin{thm}
\label{thm:ourKM}
Let $\cC$ be a semi-simple ribbon category, $A$ be a simple object in $\cC$, $X$ be a simple invertible object, and $L$ a link with $\# L$ components.  Then, $$\J{\cC}{A \otimes X}(L) = \J{\cC}{A}(L) \J{\cC}{X}(L) = (\dim X)^{\# L} \J{\cC}{A}(L).$$
\end{thm}
\begin{proof}
First look at the framed version of the knot invariants.  The framed $A \otimes X$ invariant comes from cabling $L$ and labeling one of the two cables $A$ and the other one $X$.  We unlink the link labeled $A$ from the link labeled $X$ by successively changing crossings where $X$ goes under $A$ to crossings where $X$ goes over $A$.  Each crossing in the original link gives rise to two crossings between the $X$-labelled link and the $A$-labelled link, and exactly one of these crossings needs to be switched.  Furthermore, the sign of the crossing that needs to be switched is the same as the sign of the original crossing.  See the following diagram for what happens at each positive crossing.

$$
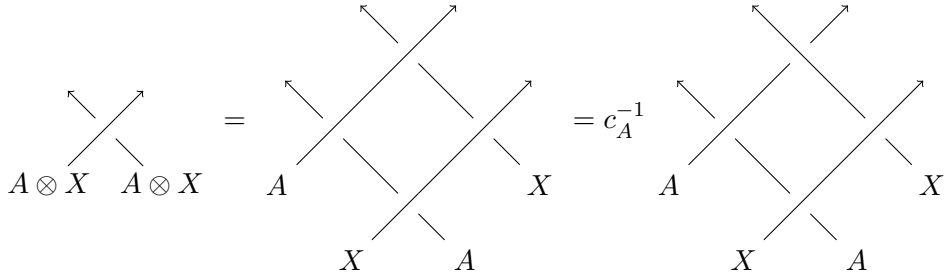
\begin{tikzpicture}[baseline]
\node (x) at (0,0){};
    \draw[->] (x.45)-- (.5,.5);
    \draw[->] (x.135) -- (-.5,.5);
    \draw (x.315) -- (.5,-.5);
    \draw (x.45) -- (-.5,-.5);

\node at (-.75,-.75){$A\otimes X$};
\node at (.75,-.75){$A \otimes X$};
\end{tikzpicture} = \begin{tikzpicture}[baseline]
\node (lower) at (0,-1){};

\node (upper) at (0,1){};
    \draw[->] (upper.45)-- +(.5,.5);
    \draw[->] (upper.135) -- +(-.5,.5);

\node (left) at (-1,0){};
    \draw[->] (left.135) -- +(-.5,.5);

\node (right) at (1,0){};
    \draw[->] (right.45)-- +(.5,.5);

\node (A1) at (-1.75,-.75){$A$};
\node (X1) at (-.75,-1.75){$X$};
\node (A2) at (1.75,-.75){$X$};
\node (X2) at (.75,-1.75){$A$};

\draw (left.45) -- (A1.45);
\draw (lower.45) -- (X1.45);
\draw (lower.315)--(X2.135);
\draw (right.315)--(A2.135);

\draw (left.45)--(upper.45);
\draw (left.315)--(lower.135);
\draw (right.135)--(upper.315);
\draw (right.45)--(lower.45);

\end{tikzpicture} = c_A^{-1}\begin{tikzpicture}[baseline]
\node (lower) at (0,-1){};

\node (upper) at (0,1){};
    \draw[->] (upper.45)-- +(.5,.5);
    \draw[->] (upper.135) -- +(-.5,.5);

\node (left) at (-1,0){};
    \draw[->] (left.135) -- +(-.5,.5);

\node (right) at (1,0){};
    \draw[->] (right.45)-- +(.5,.5);

\node (A1) at (-1.75,-.75){$A$};
\node (X1) at (-.75,-1.75){$X$};
\node (A2) at (1.75,-.75){$X$};
\node (X2) at (.75,-1.75){$A$};

\draw (left.45) -- (A1.45);
\draw (lower.45) -- (X1.45);
\draw (lower.315)--(X2.135);
\draw (right.315)--(A2.135);

\draw (left.45)--(upper.225);
\draw (left.315)--(lower.135);
\draw (right.135)--(upper.135);
\draw (right.45)--(lower.45);

\end{tikzpicture}$$

Hence, unlinking the $X$-labelled link from the $A$-labelled link picks up a factor of $c_A^{-writhe}$.  At this point, the link labelled by $A$ lies completely behind the link labelled by $X$, and we can compute their invariants separately. Thus,
$$\theta_{A \otimes X}^{writhe} \J{\cC}{A \otimes X}(L) = c_A^{-writhe} \theta_A^{writhe} \J{\cC}{A}(L) \theta_X^{writhe} \J{\cC}{X}(L).$$
Rearranging terms and writing $c_A$ in terms of twist factors, we see that $\J{\cC}{A \otimes X}(L) = \J{\cC}{A}(L) \J{\cC}{X}(L)$.  The final equation follows from Theorem \ref{thm:summands}.
\end{proof}

Note that $\dim{X}$ above has to be $1$ or $-1$, since $\dim{X}=\dim{X^*}$ and 
$X \otimes X^*=1$.

Suppose that you have a finite ribbon category whose fusion graph is symmetric.  Take $X$ to be any projection which is symmetric in the fusion graph with $\id$.  Then it is easy to see that its Frobenius-Perron dimension $\dim_{FP}(X) = 1$, and thus that $X$ is invertible.  Hence, any time the fusion graph has a symmetry so do the knot invariants.

If $X$ gives a Kirby-Melvin symmetry, then if you're lucky you can set $X \cong \id$ using the deequivariantization procedure.  Furthermore, even if you can't deequivariantize immediately (for example, if $\dim X \neq 1$) you might still be able to modify the category $\cC$ is some mild way (changing the braiding or changing the pivotal structure, neither of which changes the link invariant significantly) and then be able to deequivariantize.  We give three examples of this:

Consider $\Rep{U_{q=-\exp \left(-2 \pi i \frac{1}{10} \right)}(\SL{2})}$.  The representation $V_3$ is invertible and thus gives a Kirby-Melvin symmetry.  We can make this monoidal category into a ribbon category in many ways: first we can choose $s = q^{\frac{1}{2}}$ in two different ways; second we can choose either the usual pivotal structure or the unimodal one.  For each of these four choices we check each of the conditions needed to define the deequivariantization $\cC//V_3$ (transparency, dimension $1$, and twist factor $1$).

\begin{figure}[!ht]
\begin{center}
\begin{tabular}{c|c|c|c}
&$V_3$ transparent & $\dim V_3$ & $\theta_{V_3}$\\
\hline
$\Rep{U_{s=\exp\left(2 \pi i \frac{1}{5} \right)}(\SL{2})}$&Yes&-1&1\\
\hline
$\Rep{U_{s=\exp\left(2 \pi i \frac{7}{10} \right)}(\SL{2})}$&No&-1&-1\\
\hline
$\uRep{U_{s=\exp\left(2 \pi i \frac{1}{5} \right)}(\SL{2})}$&No&1&-1\\
\hline
$\uRep{U_{s=\exp\left(2 \pi i \frac{7}{10} \right)}(\SL{2})}$& Yes&1&1
\end{tabular}
\end{center}
\end{figure}

Let $\rRep U_q(\mathfrak{g})$ denote the full subcategory of representations whose highest weights are in the root lattice.  (Notice that this ribbon category only depends on $q$, not on a choice of $s=q^{\frac {1}{L}}$.  Furthermore, it does not depend on the choice of ribbon element.)

\begin{lem} \label{lem:KM-SL2}
$\rRep U_{q=-\exp \left(-2 \pi i \frac{1}{10} \right)}(\SL{2}) \cong \uRep{U_{s=\exp\left(2 \pi i \frac{7}{10} \right)}(\SL{2})^{modularize}}$.
\end{lem}

\begin{proof}
We restrict the deequivariantization $$\cF: \uRep{U_{s=\exp\left(2 \pi i \frac{7}{10} \right)}(\SL{2})} \rightarrow \uRep{U_{s=\exp\left(2 \pi i \frac{7}{10} \right)}(\SL{2})} // V_3$$ to $\rRep$.  Since $\otimes V_3$ acts freely on the isomorphism classes of simple objects and since every orbit contains exactly one object in $\rRep$ the restriction of this functor is an equivalence by Lemma \ref{lem:free-quotient}.
\end{proof}

We will need two similar results, for $\Rep U_{q=-\exp \left(-2 \pi i \frac{1}{10}\right)}(\SL{2} \oplus \SL{2})$ and for $\Rep{U_{q=-\exp \left(-\frac{2\pi i}{14}\right)}(\SL{4})}$.

In $\Rep U_{q=-\exp \left(-2 \pi i \frac{1}{10}\right)}(\SL{2} \oplus \SL{2})$ we can consider the root representations, those of the form $V_a \boxtimes V_b$ with both $a$ and $b$ even, as well as the vector representations, those $V_a \boxtimes V_b$ with $a+b$ even. We call these the vector representations because they become the vector representations under the identification $\SL{2} \directSum \SL{2} \iso \SO{4}$.

\begin{lem} \label{lem:KM-SO4}
\begin{align*}
& \rRep U_{q=-\exp \left(-2 \pi i \frac{1}{10}\right)}(\SL{2} \directSum \SL{2}) \\
& \qquad \cong \vRep U_{q=-\exp \left(-2 \pi i \frac{1}{10} \right)}(\SL{2} \directSum \SL{2}) // V_3 \boxtimes V_3  \\
& \qquad \cong \uRep{U_{s=\exp \left(2 \pi i \frac{7}{10} \right)}(\SL{2} \directSum \SL{2})}^{modularize}.
\end{align*}
\end{lem}
\begin{proof}
We make the abbreviations
\begin{align*}
\cR & = \rRep U_{q=-\exp \left(-2 \pi i \frac{1}{10}\right)}(\SL{2} \directSum \SL{2}) \\
\cV & = \vRep U_{q=-\exp \left(-2 \pi i \frac{1}{10} \right)}(\SL{2} \directSum \SL{2}) \\
\cU & = \uRep{U_{s=\exp \left(2 \pi i \frac{7}{10} \right)}(\SL{2} \directSum \SL{2})}.
\end{align*}
It is easy to check that $\cR$ and $\cV$ are not affected by either the choice of $s$ (recall in this situation $s$ is a square root of $q$, required for the definition of the braiding), or changing between the usual and the unimodal pivotal structures.
Thus we have inclusions $$\cR \subset \cV \subset \cU.$$
The invertible objects in $\cU$ are the representations  $V_0 \boxtimes V_0, V_0 \boxtimes V_3, V_3 \boxtimes V_0$ and $V_3 \boxtimes V_3$. For any choice of $s$ and either pivotal structure, $V_3 \boxtimes V_3$ is transparent. The representations $V_0 \boxtimes V_3$ and $V_3 \boxtimes V_0$ are transparent only with $s=\exp \left(2 \pi i \frac{7}{10} \right)$ and the unimodal pivotal structure. Under tensor product, the invertible objects form the group $\Integer/2 \Integer \times \Integer/2 \Integer$.
The invertible objects in $\cV$ are $V_0 \boxtimes V_0$ and $V_3 \boxtimes V_3$, forming the group $\Integer/2 \Integer$.

We have
\begin{align*}
(V_a \boxtimes V_b) \tensor (V_0 \boxtimes V_3) & \iso V_a \boxtimes V_{3-b} \\
(V_a \boxtimes V_b) \tensor (V_3 \boxtimes V_0) & \iso V_{3-a} \boxtimes V_b \\
(V_a \boxtimes V_b) \tensor (V_3 \boxtimes V_3) & \iso V_{3-a} \boxtimes V_{3-b},
\end{align*}
and so see that the action of the group of invertible objects is free.
Each $\Integer/2 \Integer \times \Integer/2 \Integer$ orbit on $\cU$ contains exactly one object from $\cR$, and each $\Integer/2 \Integer$ orbit on $\cV$ contains exactly one object from $\cR$. See Figure \ref{fig:SO4-quotients}.
\begin{figure}[!ht]
\begin{center}
\begin{tabular}{cc}
$\mathfig{0.35}{coincidences/so4-4fold-quotient}$ &
$\mathfig{0.35}{coincidences/so4-vector-quotient}$ \\
(a) & (b)
\end{tabular}
\end{center}
\caption{(a) The $4$-fold quotient of the Weyl alcove and (b) the $2$-fold quotient of the Weyl alcove, with vector representations marked. Lemma \ref{lem:KM-SO4} identifies the two resulting $4$-object categories.}
\label{fig:SO4-quotients}
\end{figure}

Thus both equivalences in this lemma are deequivariantizations, by applying Lemma \ref{lem:free-quotient} to the inclusions $\cR \subset \cU$ and $\cV \subset \cU$.
\end{proof}

In $\Rep{U_{q=-\exp \left(-\frac{2\pi i}{14}\right)}(\SL{4})}$ we can again consider two subcategories, the root representations and the vector representations. The root representations of $\SL{4}$ are those whose highest weight is an $\Natural$-linear combination of $(2,-1,0), (-1,2,-1), (0,-1,2)$ in $\Natural^3$. They form an index $4$ sublattice of the weight lattice. The Weyl alcove for $\SL{4}$ at a $14$-th root of unity consists of those weights $(a,b,c) \in \Natural^3$ with $a+b+c \leq 3$, and so the relevant root representations are $V_{(000)}, V_{(101)}, V_{(210)}, V_{(012)}$ and $V_{(020)}$. The vector representations $\vRep U_{q=-\exp \left(-\frac{2\pi i}{14}\right)}(\SL{4})$ are those that become vector representations under the identification $\SL{4} \iso \SO{6}$ (this is $A_3 = D_3$), namely those $V_{(abc)}$ with $a+c$ even. These form an index $2$ sublattice of the weight lattice, containing the root lattice. Both sublattices are illustrated in Figure \ref{fig:SL4-lattices}; hopefully having these diagrams in mind will ease later arguments.

\begin{figure}[!ht]
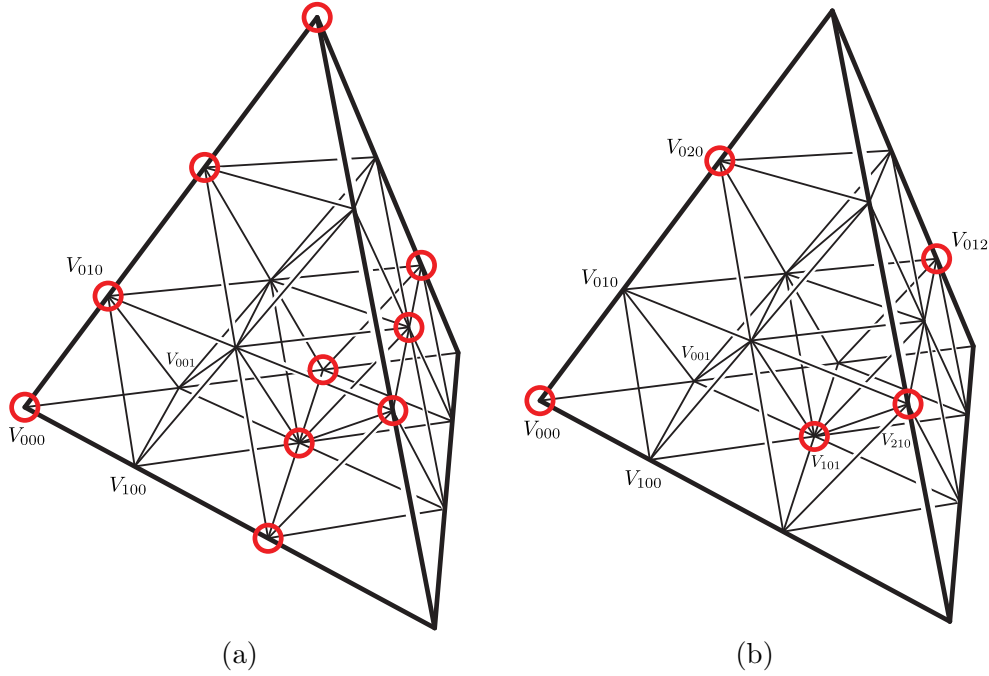

\begin{center}
\begin{tabular}{cc}
$\mathfig{0.45}{coincidences/sl4-vector-labels}$ &
$\mathfig{0.45}{coincidences/sl4-root-labels}$ \\
(a) & (b)
\end{tabular}
\end{center}
\caption{The $\SL{4}$ Weyl alcove at a $14$-th root of unity, showing (a) the vector representation sublattice and (b) the root representation sublattice.}
\label{fig:SL4-lattices}
\end{figure}

\begin{lem} \label{lem:KM-SO6}
\begin{align*}
\rRep U_{q=-\exp \left(-\frac{2 \pi i}{14} \right)}(\SL{4}) & \cong \vRep U_{q=-\exp \left(-\frac{2 \pi i}{14} \right)}(\SL{4}) // V_{(030)}  \\
	& \cong \uRep{U_{s=\exp \left(2 \pi i \frac{5}{14}\right)}(\SL{4})}^{modularize}.
\end{align*}
\end{lem}
\begin{proof}
We make the abbreviations
\begin{align*}
\cR & = \rRep U_{q=-\exp \left(-\frac{2 \pi i}{14} \right)}(\SL{4}) \\
\cV & = \vRep U_{q=-\exp \left(-\frac{2 \pi i}{14} \right)}(\SL{4}) \\
\cU & = \uRep{U_{s=\exp \left(2 \pi i \frac{5}{14}\right)}(\SL{4})}.
\end{align*}
It is easy to check that $\cR$ and $\cV$ are not affected by either the choice of $s$ (recall in this situation $s$ is a $4$-th root of $q$, required for the definition of the braiding), or any variation of pivotal structure. Thus we have inclusions $$\cR \subset \cV \subset \cU.$$
The invertible objects in $\cU$ are the representations  $V_{(000)}, V_{(300)}, V_{(030)}$ and $V_{(003)}$. For any choice of $s$ and pivotal structure, $V_{(030)}$ is transparent. The representations $V_{(300)}$ and $V_{(003)}$ are transparent only with $s=\exp \left(2 \pi i \frac{5}{14}\right)$ and the unimodal pivotal structure. Under tensor product, the invertible objects form the group $\Integer/4 \Integer$.
The invertible objects in $\cV$ are $V_{000}$ and $V_{030}$, forming the group $\Integer/2 \Integer$.

The action of the group of invertible objects is free, and shown in Figure \ref{fig:SL4-action}. Each $\Integer/4 \Integer$ orbit on $\cU$ contains exactly one object from $\cR$, and each $\Integer/2 \Integer$ orbit on $\cV$ contains exactly one object from $\cR$. See Figure \ref{fig:SL4-quotients}.

\begin{figure}[!ht]
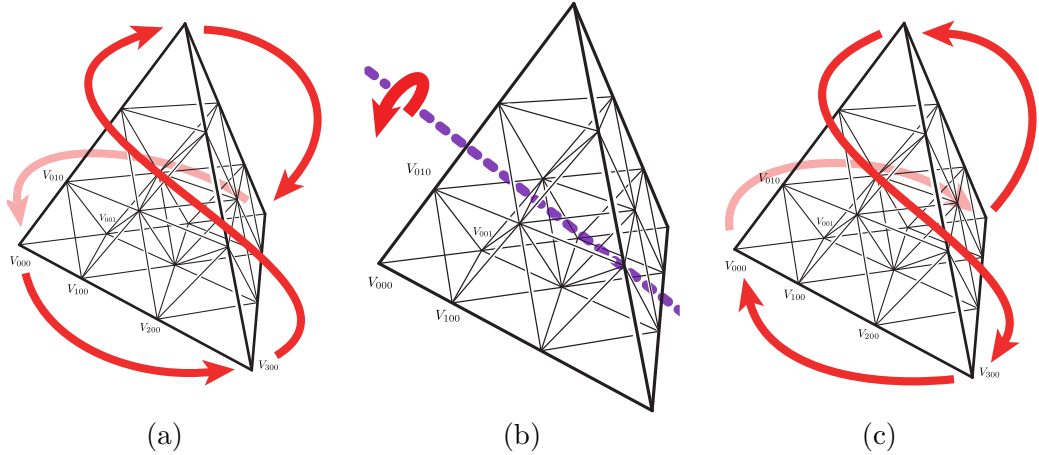

\begin{tabular}{ccc}
$\mathfig{0.3}{coincidences/sl4-300}$ &
$\mathfig{0.3}{coincidences/sl4-030}$ &
$\mathfig{0.3}{coincidences/sl4-003}$
\\
(a) & (b) & (c)
\end{tabular}
\caption{The action of tensor product with an invertible object. (a) $- \tensor V_{(300)}$ and (c) $- \tensor V_{(003)}$ act by orientation reversing isometries, while (b) $- \tensor V_{(030)}$ acts by a $\pi$ rotation.}
\label{fig:SL4-action}
\end{figure}

\begin{figure}[!ht]
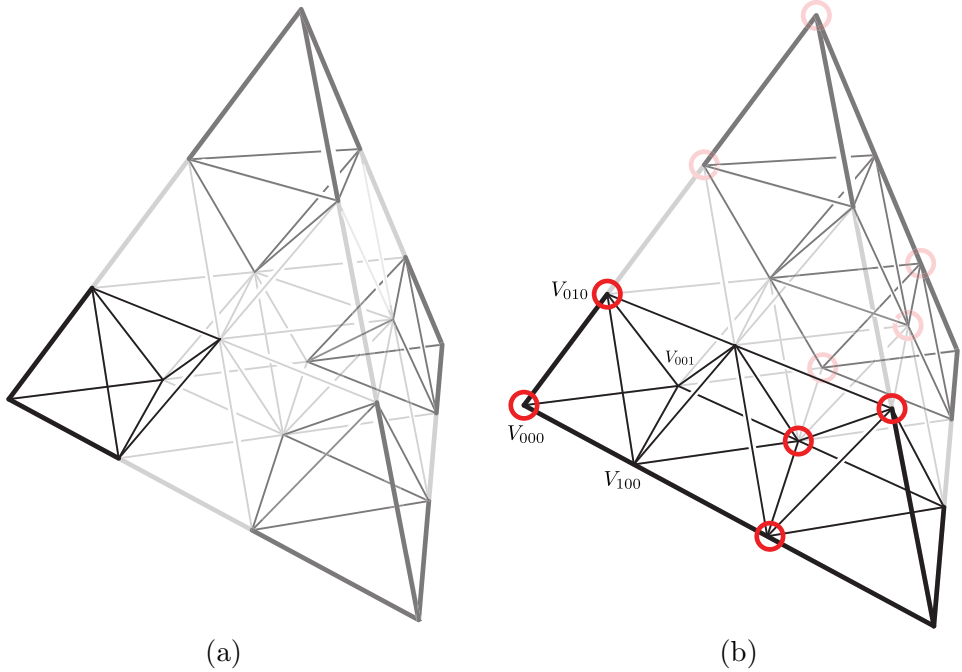

\begin{center}
\begin{tabular}{cc}
$\mathfig{0.45}{coincidences/sl4-4fold-quotient}$ &
$\mathfig{0.45}{coincidences/sl4-vector-quotient}$ \\
(a) & (b)
\end{tabular}
\end{center}
\caption{(a) The $4$-fold quotient of the Weyl alcove and (b) the $2$-fold quotient of the Weyl alcove, with vector representations marked.  Lemma \ref{lem:KM-SO6} identifies the two resulting $5$-object categories.}
\label{fig:SL4-quotients}
\end{figure}

Thus both equivalences in the Lemma are de-equivariantizations, by applying Lemma \ref{lem:free-quotient} to the inclusions $\cR \subset \cU$ and $\cV \subset \cU$.
\end{proof}

Finally, the usual statement in the literature of generalized Kirby-Melvin symmetry involves changing the label of only one component on the link.  This can be proved in a completely analogous way to the result above.  We recall the statement here.

\begin{thm}
\label{thm:KM}
Let $J_{A_1, \ldots, A_k}(L)$ be the value of a framed link $L$ (with components $L_1, \ldots, L_k$), 
labeled by simple objects $A_1, \ldots, A_k$.  Suppose now that $A_1$ is replaced by 
$A_1 \otimes X$ (with $X$ invertible).  Then 
$$J_{A_1 \otimes X, A_2, \ldots, A_k}(L) = 
	\dim{X}  \cdot c_X^{\operatorname{writhe}(L_1)} \cdot \prod_{i=1,\ldots,k} c_{A_i}^{\operatorname{linking}(L_1', L_i)} \cdot J_{A_1, \ldots, A_k}(L)
	$$
where $L_1'$ is a copy of $L_1$ running parallel to $L_1$ in the 
blackboard framing.  
\end{thm}

\subsection{Level-rank duality}
Level-rank duality is a collection of ideas saying that the semisimplified representation theory of a quantum group at a certain root of unity is related to that of a different quantum group, at a (potentially) different root of unity. The rank of a quantum group in this setting is dimension of its natural representation (i.e. the $n$ in $\SO{n}$ or $\SL{n}$).  The level describes the root of unity.  The name ``level" comes from the connection between quantum groups at roots of unity and projective representations of loop groups at a fixed level.  Here the relationship between the root of unity and the level is given  by the formula
$$k  = \frac{l}{2 D} - \check{h}$$
where $l$ is the order of the root of unity, $D$ is the lacing number of the quantum group, and $\check{h}$ is the dual Coxeter number. See Figure \ref{fig:quantum-group-data} for the values for each simple Lie algebra. Notice that not all roots of unity come from loop groups under this correspondence. 

\begin{figure}[!ht]
\begin{center}
\begin{tabular}{cccccc}
type & Lie group & rank  & lacing number $D$ & dual Coxeter number $\check{h}$ & L \\
\hline
$A_n$ & $\SL{n+1}$ & $n$ & $1$ & $n+1$ & $n+1$ \\
$B_{\text{$n$ even}}$ & $\SO{2n+1}$ & $n$ & $2$ & $2n-1$ & 1\\
$B_{\text{$n$ odd}}$ & $\SO{2n+1}$ & $n$ & $2$ & $2n-1$ & 2\\
$C_n$ & $\SP{2n}$ & $n$ & $2$ & $n+1$ & $1$ \\
$D_{\text{$n$ even}}$ & $\SO{2n}$ & $n$ & $1$ & $2n-2$ & 2 \\
$D_{\text{$n$ odd}}$ & $\SO{2n}$ & $n$ & $1$ & $2n-2$ & 4 \\
$E_n$ & $E_{6|7|8}$ & $6, 7, 8$ & $1$ & $12,18,30$ & $3,2,1$ \\
$F_4$ & $F_4$ & $4$ & $2$ & $9$ & $1$ \\
$G_2$ & $G_2$ & $2$ & $3$ & $4$ & $1$
\end{tabular}
\end{center}
\caption{Combinatorial data for the simple Lie algebras.}
\label{fig:quantum-group-data}
\end{figure}

Nonetheless there are versions of level-rank duality for quantum groups at roots of unity not corresponding to loop groups.  In this context what the ``level" measures is which quantum symmetrizers vanish, while the rank measures which quantum antisymmetrizers vanish.  At the level of combinatorics, the rank gives the bound on the number of rows in Young diagrams, while the level gives a bound on the number of columns, and duality is realized by reflecting Young diagrams thus interchanging the roles of rank and level.

We want statements of level-rank duality that give equivalences of braided tensor categories.  In order to get such precise statements several technicalities appear.  First, level-rank duality concerns $SO$, not $Spin$, so we only look at the vector representations.  Second, there is a subtle relationship between the roots of unity you need to pick on each side of the equivalence.  In particular, if the root of unity on the left side is of the form $\exp(\frac{2 \pi i}{m})$ then the root of unity on the right side typically will not be of that form.  Finally, level-rank duality is most natural as a statement about $U$ and $O$, not about $SU$ and $SO$.  Getting statements about $SU$ and $SO$ requires considering modularizations.  (It may seem surprising that this is even possible, since we know that $\Rep(U_{\zeta(\ell)}(\SO{n}))$ is already a modular tensor category \cite[Theorem 6]{MR2286123}. When we restrict to the subcategory of vector representations, however, we lose modularity.)

We found the papers  \cite{MR1710999} (on the $SU$ case) 
and \cite{MR1854694} (on the $SO$ and $Sp$ cases) 
to be exceedingly useful, and we'll give statements and proofs that closely follow their methods. 
Level-rank duality for $SO(3)-SO(4)$ appears in the paper \cite{MR2469528}, where it is used to prove Tutte's golden identity for the chromatic polynomial.  For our particular case of level-rank duality involving $\mathfrak{so}_3$ and the $\mathcal{D}_{2m}$ subfactor planar algebra, see the more physically minded \cite{MR1279547}.
For some more background on level-rank duality, see \cite{MR1103065, MR1063959, MR1710999, MR2124554} for the $SU$ cases, \cite{MR1414470} for level-rank duality at the level of $3$-manifold invariants and \cite{MR1601870} for loop groups.

As explained by Beliakova and Blanchet, level-rank duality is easiest to understand in a diagrammatic setting, where it says that $U(n)_k \iso U(k)_n$ and $O(n)_k \iso O(k)_n$, with $U$ and $O$ being interpreted as categories of tangles modulo either the HOMFLYPT or Dubrovnik relations. The equivalences come from almost trivial symmetries of the relations. 
The reason this modularization is necessary is that to recover $SO$ from $O$, we need to quotient out by the determinant representation. Thus, to translate an equivalence $O(n)_k \iso O(k)_n$ into something like $SO(n)_k \iso SO(k)_n$, we find that in each category there is the `shadow' of the determinant representation in the other category, which we still need to quotient out. See Figure \ref{fig:LR-schematic} for a schematic diagram illustrating this.

\begin{figure}[!ht]
$$\xymatrix{
&O(n)_k=O(k)_n \ar[ddl]_{\mathrm{det}_n \cong \id} \ar[ddr]^{\mathrm{det}_k \cong \id} & \\
&& \\
SO(n)_k \ar[ddr]_{\mathrm{det}_k \cong \id} && SO(k)_n \ar[ddl]^{\mathrm{det}_n \cong \id}\\
&&\\
&SO(n)_k^{\operatorname{mod}} \cong SO(k)_n^{\operatorname{mod}}&
}$$
\caption{A schematic description of $SO$ level-rank duality, suppressing the details of the actual roots of unity appearing.}
\label{fig:LR-schematic}
\end{figure}
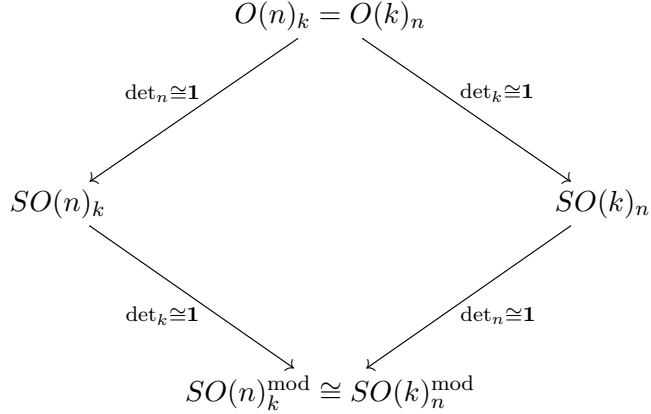

Here is the precise statement of level-rank duality which we will be using.
Define \begin{equation*}
\ell_{n,k} = \begin{cases} 2 (n+k-2) & \text{if $n$ and $k$ are even} \\ 4(n+k-2) & \text{if $n$ is odd and $k$ is even}
\\ n+k-2 & \text{if $n$ is even and $k$ is odd}
\end{cases}
\end{equation*}

\begin{thm}[SO level-rank duality]
\label{thm:DLR}
Suppose $n, k \geq 3$ are not both odd. Suppose $q_1$ is a primitive root of unity with order $\ell_{n,k}$. Choose $q_2$ so that
\begin{equation*}
-1 = \begin{cases}
q_1 q_2 & \text{if $n$ and $k$ are both even} \\
q_1^2 q_2 & \text{if $n$ is odd and $k$ is even} \\
q_1 q_2^2 & \text{if $n$ is even and $k$ is odd} \\
\end{cases}
\end{equation*}
As ribbon categories, there is an equivalence
\begin{equation*}
\vRep(U_{q=q_1}(\SO{n}))//V_{k e_1} \iso \vRep(U_{q=q_2}(\SO{k}))//V_{n e_1}.
\end{equation*}
\end{thm}
\begin{rem}
When both $n$ and $k$ are odd, there is some form of level-rank duality in terms of the Dubrovnik skein relation,  pursued in  \cite{MR1854694} where it is called the $B^{n,-k}$ case.  However it does not seem possible to express this case purely in terms of quantum groups.  
\end{rem}

\begin{rem}
Notice that the order of $q_2$ is always $\ell_{k,n}$. When $n$ and $k$ are both even then the roots of unity on both sides come from loop groups.  However, when $n$ or $k$ is odd the roots of unity are not the ones coming from loop groups.
\end{rem}

\begin{proof}
We begin by defining a diagrammatic category $\cO(t,w)$, and then seeing that a certain $\Integer/2\Integer \times \Integer/2\Integer$ quotient can be realised via two steps of deequivariantization in two different ways. In the first way, after the initial deequivariantization we obtain a category equivalent to $\vRep(U_{q_1}(\SO{n}))$, while in the second way we obtain a category equivalent to $\vRep(U_{q_2}(\SO{k}))$ instead. The second steps of deequivariantizations give the categories in the statement above; since both are the modular quotient of $\cO(t, w)$ for a certain $t$ and $w$, they are equivalent.

\begin{defn}
The category $\widetilde{\cO}(t,w)$ is the idempotent completion of the BMW category (the quotient of the tangle category by the Dubrovnik skein relations) with
\begin{align*}
a & = w^{t-1} \\
z & = w-w^{-1}.
\end{align*}
The category $\cO(t,w)$ is the quotient of $\widetilde{\cO}(t,w)$ by all negligible morphisms.
\end{defn}

Now define $w_{n,k}$ by
\begin{equation*}
w_{n,k}  = \begin{cases} q_1 & \text{if $n$ is even} \\
q_1^2 & \text {if $n$ is odd.}
\end{cases} 
\end{equation*}
Note the $w_{n,k}$ is a root of unity of order $2(n+k-2)$ when $k$ is even and of order $n+k-2$ with $k$ is odd. The hypotheses of the theorem then ensure that
\begin{equation*}
-w_{n,k}^{-1}  = \begin{cases}q_2 & \text{if $k$ is even} \\
q_2^2 & \text{if $k$ is odd.}\end{cases}
\end{equation*}

\begin{lem}
\label{lem:Oiso}
For $n,k \in \Natural$ and not both odd, the categories $\cO(n,w_{n,k})$ and $\cO(k,-w_{n,k}^{-1})$ are equivalent.
\end{lem}
\begin{proof}
In $\cO(k, -w_{n,k}^{-1})$, $z = -w_{n,k}^{-1}+w_{n,k}$, which is the same value of $z$ as appears in $\cO(n,w_{n,k})$. Similarly, in $\cO(k, -w_{n,k}^{-1})$, we have 
\begin{align*}
a & = (-w_{n,k}^{-1})^{k-1} \\
   & = \begin{cases}
   	 	-w_{n,k}^{1-k} = w_{n,k}^{n+k-2+1-k} & \text{if $k$ is even} \\
		w_{n,k}^{1-k} = w_{n,k}^{n+k-2+1-k} & \text{if $k$ is odd}
	 \end{cases} \\
   & = w_{n,k}^{n-1}
\end{align*}
and so the same values of $a$ appear in both categories; thus they actually have exactly the same definition.
\end{proof}

\begin{lem}
When $n \in \Natural$, the category $\cO(n,w)$ has a transparent object with quantum dimension $1$, which we'll call $\det{}_{n}$. Further, if $w=w_{n,k}$, there is another such object $\det{}_{k}$ coming from $\cO(k,-w^{-1})$ via the equivalence of the previous lemma. These transparent objects form the group $\Integer/2\Integer \times \Integer/2\Integer = \{\id, \det{}_{n}, \det{}_{k}, \det{}_{n} \tensor \det{}_{k}\}$ under tensor product.
\end{lem}
\begin{proof}
See \cite[Lemmas 4.1.ii and 4.3]{MR1854694}.
\end{proof}


Write $\ell(q)$ for the order of a root of unity $q$, and define 
\begin{equation*}
\ell'(q) = \begin{cases} \ell(q) & \text{if $2 \nmid \ell(q)$} \\ \ell(q)/2 & \text{if $2 \mid \ell(q)$.}\end{cases}
\end{equation*}

\begin{lem}
We can identify the deequivariantizations as
\begin{align*}
\vRep(U_{q}(\SO{n})) & \iso
	\begin{cases}
		\cO(n,q)//\det{}_{n} & \text{if $n$ is even} \\
		\cO(n,q^2)//\det{}_{n} & \text{if $n$ is odd}
	\end{cases}
\end{align*}
for any $q$, as long as if $q$ is a root of unity, when $n$ is even, $\ell'(q) \geq n-2$, and when $n$ is odd, $\ell'(q) \geq 2(n-2)$ when $2 \mid \ell'(q)$ and $\ell'(q) > n-1$ when $2 \nmid \ell'(q)$.

In particular when $q=q_1$ we obtain
\begin{align*}
\cO(n,w_{n,k})//\det{}_{n} & \iso \vRep(U_{q_1}(\SO{n})) \displaybreak[1] \\
\intertext{and further,}
\cO(n,w_{n,k})//\det{}_{k} & \iso	\vRep(U_{q_2}(\SO{k})) \end{align*}
Moreover, in $\cO(n,w_{n,k})//\det{}_{n}$, we have $\det{}_{k} \iso V_{k e_1}$ and in $\cO(n,w_{n,k})//\det{}_{k}$, we have $\det{}_{n} \iso V_{n e_1}$.
\end{lem}
\begin{proof}
The first equivalence follows from the main results of \cite{MR2132671}. We give a quick sketch of their argument.  The fact that the eigenvalues of the $R$-matrix acting on the standard representation of $\SO{n}$ are $q^{-2n+2}, -q^{-2}$ and $q^2$ when $n$ is odd, or $q^{-n+1}, -q^{-1}$ and $q$ when $n$ is even ensures that this is a functor from $\widetilde{\cO}(n,q^2)$ or $\widetilde{\cO}(n,q)$, by Theorems \ref{thm:summands}, \ref{thm:eigenvalues} and \S \ref{sec:ribbonfunctors}.  One then checks that the functor factors through the deequivariantization.  Finally, by computing dimensions of Hom-spaces one concludes that the functor must kill all negligibles and must be surjective.

One can check that $\ell'(q_1) = n+k-2$ when $n$ is even or $2(n+k-2)$ when $n$ is odd, and so the required inequalities always hold for $\SO{n}$. 

The last equivalence follows from the first and Lemma \ref{lem:Oiso},:
\begin{align*}
\cO(n,w_{n,k})//\det{}_{k} & \iso \cO(k,-w_{n,k}^{-1})//\det{}_{k} \displaybreak[1] \\
					& \iso \vRep(U_{q_2}(\SO{k})) \displaybreak[1]
\end{align*}
Here we check that
$\ell'(q_2)= n+k-2$ when $k$ is even or $2(n+k-2)$ when $k$ is odd, satisfying the inequalities for $\SO{k}$.
\end{proof}

The proof of the theorem is now immediate; we write $\cO(n,w_{n,k})//\{\det{}_n,\det{}_k\}$ in two different ways, obtaining
\begin{align*}
\cO(n,w_{n,k})//\{\det{}_n,\det{}_k\} & = \cO(n,w_{n,k})// \det{}_n //\det{}_k \\
						  & \iso 	\vRep(U_{q_1}(\SO{n})) // V_{k e_1} \displaybreak[1] \\
\intertext{and}
\cO(n,w_{n,k})//\{\det{}_n,\det{}_k\} & = \cO(n,w_{n,k})// \det{}_k //\det{}_n \displaybreak[1] \\
						  & \iso \vRep(U_{q_2}(\SO{k})) // V_{n e_1}. \qedhere
\end{align*}
\end{proof}
\begin{rem}
One can easily verify an essential condition, that the twist factor for $V_{ke_1}$ inside $\vRep(U_{\zeta(\ell_{n,k})}(\SO{n}))$ is $+1$, from the formulas for the twist factor given in \S \ref{sec:quantumgroups}.
\end{rem}

Finally, we specialize to the case $n=3$, where the $\mathcal{D}_{2m}$ planar algebras appear. 

\begin{thm}[$SO(3)$-$SO(k)$ level-rank duality]
\label{thm:SO3-duality}
Suppose $k \geq 4$ is even. There is an equivalence of ribbon categories
\begin{equation*}
\frac{1}{2} \mathcal{D}_{k+2} \iso \vRep(U_{q=-\exp\left(-\frac{2 \pi i}{2k+2}\right)}(\SO{k}))//V_{3 e_1}
\end{equation*}
sending the tensor generator $W_2$ of $\frac{1}{2} \mathcal{D}_{k+2}$ to $V_{2 e_1}$ and $P$ to $V_{2e_{\frac{k}{2}-1}}$.
\end{thm}

This follows immediately, from the description in \S \ref{sec:d2n} of the even part of $\mathcal{D}_{2n}$ as $\frac{1}{2}\mathcal{D}_{2n} \cong \vRep(U_{q=\exp(\frac{2\pi i}{8n-4})}(\SO{3})^{modularize}$, and the general case of level-rank duality.

\subsection{Applications}
\label{sec:applications}%

\subsubsection{Knot invariants}
Combining $SO(3)$-$SO(k)$ level-rank duality for even $k \geq 8$ with Kirby-Melvin symmetry, we obtain the following knot polynomial identities:
\begin{thm}[Identities for $n \geq 3$]
\label{thm:identities-geq}
For all knots $K$,
\begin{align}
\restrict{\J{\SL{2}}{(2n-2)}(K)}{q=\exp(\frac{2\pi i}{8n-4})} & = 2 \J{\mathcal{D}_{2n}}{P}(K) \notag \displaybreak[1]  \\
                                                   & = 2 \restrict{\J{\SO{2n-2}}{2e_{n-2}}(K)}{q=-\exp(-\frac{2\pi i}{4n-2})} \notag \displaybreak[1] \\
                                                   & = (-1)^{1+\ceil{\frac{n}{2}}} 2 \restrict{\J{\SO{2n-2}}{e_{n-2}}(K)}{q=-\exp(-\frac{2\pi i}{4n-2})}
                                                   \label{eq:identities-geq-5} \displaybreak[1] \\
\intertext{and for all links $L$,}                                               
\restrict{\J{\SL{2}}{(2)}(L)}{q=\exp(\frac{2\pi i}{8n-4})}    & = \J{\mathcal{D}_{2n}}{W_2}(L) \notag \displaybreak[1] \\
                                                   & = \restrict{\J{\SO{2n-2}}{2e_1}(L)}{q=-\exp(-\frac{2\pi i}{4n-2})} \notag \displaybreak[1] \\
                                                   & =  \restrict{\J{\SO{2n-2}}{e_1}(L)}{q=-\exp(-\frac{2\pi i}{4n-2})}. 
\end{align}
(The representation of $\SO{2n-2}$ with highest weight $e_{n-1}$ is one of the spinor representations.)
\end{thm}

\begin{proof}
The first two identities are immediate applications of Theorems \ref{thm:half} and \ref{thm:SO3-duality}. For the next identity, we use the statement of Kirby-Melvin symmetry in Theorem \ref{thm:ourKM}, with $A = V_{2e_{n-2}}$ and $X = V_{3e_{n-2}}$. 
We calculate that $\dim{X}=(-1)^{1+\ceil{\frac{n}{2}}}$ by the following trick.  At $q=\exp(\frac{2\pi i}{4n-2})$, this dimension must be $+1$, since it is the dimension of an invertible object in a unitary tensor category.  At $q=\exp(-\frac{2\pi i}{4n-2})$ it is the same, since quantum dimensions are invariant under $q \mapsto q^{-1}$, and finally we can calculate the sign at $q=-\exp(-\frac{2\pi i}{4n-2})$ by checking the parity of the exponents in the Weyl dimension formula. This implies that $X \tensor X^* \iso V_0$.  Using the Racah rule, we find $A \tensor X =  V_{e_{n-2}}$. 

Now we do the same computation again with $A = V_{2e_1}$ and $X=V_{3e_1}$.  This case is simpler since $\dim V_{3e_1} = 1$.

\end{proof}

\begin{rem}
We found keeping all the details of this theorem straight very difficult, and we'd encourage you to wonder if we eventually got it right. We had some help, however, in the form of computer computations. You too can readily check the details of this theorem on small knots and links, assuming you have access to \MMA. Download and install the \package{KnotTheory} package from \url{http://katlas.org}. This includes with it the \package{QuantumGroups} package written by Morrison, which, although rather poorly documented, provides the function \code{QuantumKnotInvariant}. This function can in principle compute any knot invariant coming from an irreducible  representation of a quantum group, but in practice runs into time and memory constraints quickly.  The explicit commands for checking small cases of the above theorem are included as a \MMA \ notebook \code{aux/check.nb} with the \code{arXiv} source of this paper.

\noop{
\begin{mma}
\begin{inm}<<KnotTheory`\end{inm}
\begin{printm}
Loading KnotTheory` version of January 20, 2009. \\
Read more at \url{http://katlas.org/wiki/KnotTheory}.
\end{printm}
\begin{inm}
\\
CheckIdentity1[$n\_$][$K\_$] := \\
\quad Chop[N[ \\
\quad\quad    QuantumKnotInvariant[$A_1$, Irrep[$A_1$][$\{2n-2\}$]] \\
\quad\quad\quad [$K$][Exp[($2\pi I$)/($8n-4$)]] \\
\quad\quad    - 2 $(-1)^{1+\operatorname{Ceiling}[n/2]}$ QuantumKnotInvariant[$D_{n-1}$, Irrep[$D_{n-1}$][UnitVector[$n-1$,$n-1$]]] \\
\quad\quad\quad [$K$][-Exp[-(($2\pi I$)/($4n-2$))]] \\
\quad ]]
\end{inm}
\begin{inm}
CheckIdentity1[$5$] /@ {Knot[$0$,$1$], Knot[$3$,$1$], Knot[$5$,$1$]}
\end{inm}
\begin{outm}
$\{0,0,0\}$
\end{outm}
\begin{inm}
Table[CheckIdentity1[$n$][Knot[$5$,$2$]], $\{n, 5, 6\}$]
\end{inm}
\begin{outm}
$\{0,0\}$
\end{outm}
\begin{inm}
CheckIdentity1[$7$][Knot[$3$,$1$]]
\end{inm}
\begin{outm}
$0$
\end{outm}
\begin{inm}
\\
CheckIdentity2[$n\_$][$K\_$] := \\
\quad Chop[N[ \\
\quad\quad    QuantumKnotInvariant[$A_1$, Irrep[$A_1$][$\{2\}$]] \\
\quad\quad\quad [$K$][Exp[($2\pi I$)/($8n-4$)]] \\
\quad\quad    - QuantumKnotInvariant[$D_{n-1}$, Irrep[$D_{n-1}$][UnitVector[$n-1$,$1$]]] \\
\quad\quad\quad [$K$][-Exp[-(($2\pi I$)/($4n-2$))]] \\
\quad ]]
\end{inm}
\begin{inm}
CheckIdentity2[$5$] /@ {Knot[$0$,$1$], Knot[$3$,$1$], Knot[$5$,$1$]}
\end{inm}
\begin{outm}
$\{0,0,0\}$
\end{outm}
\begin{inm}
Table[CheckIdentity1[$n$][Knot[$5$,$2$]], $\{n, 5, 6\}$]
\end{inm}
\begin{outm}
$\{0,0\}$
\end{outm}
\end{mma}
}
\end{rem}

Note that the $n=5$ case of Equation \eqref{eq:identities-geq-5} in Theorem \ref{thm:identities-geq} reproduces the statement of Theorem \ref{thm:identities-5}.

The $n=3$ and $n=4$ cases of Theorem \ref{thm:identities-geq} also reproduce previous results. 
The Lie algebra $\SO{2n-2}$ has Dynkin diagram $D_{n-1}$, with the spinor representations
corresponding to the two extreme vertices. At $n=4$, $D_{n-1}$ becomes
the Dynkin diagram $A_3$, and the spinor representations become the
standard and dual representations; this explains Theorem
\ref{thm:identities-4}. See Theorem \ref{thm:D8-coincidence} and Figure \ref{fig:SO6} for a full explanation. 

At $n=3$, $D_{n-1}$ becomes $A_1 \times A_1$, and
the spinor representations become $(\text{standard}) \boxtimes (\text{trivial})$
and $(\text{trivial}) \boxtimes (\text{standard})$, giving the case
described in Theorem \ref{thm:identities-3}. See Theorem \ref{thm:D6-coincidence} and Figure \ref{fig:SO4} for a full explanation.

\subsubsection{Coincidences}

\begin{proof}[Proof of Theorem \ref{thm:D6-coincidence}]
We want to construct an equivalence $$\frac{1}{2} \mathcal{D}_{6} \cong \uRep{U_{s = \exp\left(\frac{7}{10} 2 \pi i\right)}(\SL{2} \oplus \SL{2})}^{modularize},$$ sending $P \mapsto V_{(1)} \boxtimes V_{(0)}$.

First, we recall that the $k=4$ case of $SO(3)$ level-rank duality (Theorem \ref{thm:SO3-duality}) gave us the equivalence  
\begin{align*}
\frac{1}{2} \mathcal{D}_{6} & \iso \vRep U_{q=\exp( \frac{2 \pi i}{20})} (\SO{3})^{modularize} \\
& \iso \vRep U_{q=-\exp \left( -\frac{2 \pi i}{10}\right)}(\SO{4}) // V_{3e_1} 
\end{align*}
Since the Dynkin diagrams $D_2$  and $A_1 \times A_1$ 
coincide we can replace $\SO{4}$ by $\SL{2} \oplus \SL{2}$. The representation $V_{3e_1}$ of $\SO{4}$ is sent to $V_3 \boxtimes V_3$, so we have 
\begin{align*}
\frac{1}{2} \mathcal{D}_{6} & \iso \vRep U_{q=-\exp \left( -\frac{2\pi i}{10} \right)}(\SL{2} \oplus \SL{2}) // V_3 \boxtimes V_3
\end{align*}
Next, by Lemma \ref{lem:KM-SO4} we can replace this $2$-fold quotient of the vector representations with a $4$-fold quotient of the entire representations category of $\SL{2} \oplus \SL{2}$, as long as we carefully choose $s$ and the unimodal pivotal structure. Figure \ref{fig:SO4} shows the identification between the objects of $\frac{1}{2} \mathcal{D}_6$ and the corresponding objects in the $4$-fold quotient.
\end{proof}
\begin{figure}[!ht]
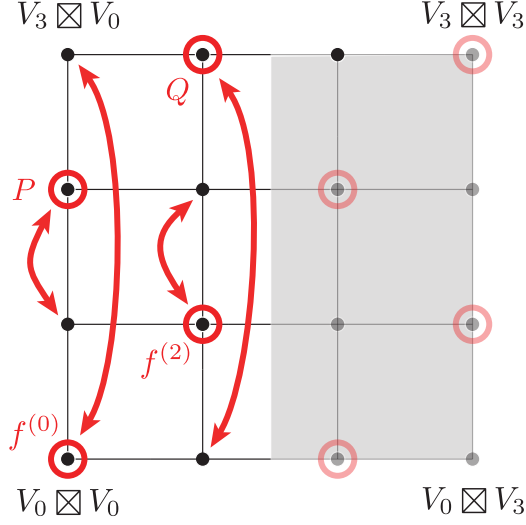

$$
\mathfig{0.5}{coincidences/SO4-chamber}
$$
\caption{
The simple objects of $\frac{1}{2} \mathcal{D}_6$ may be identified with the representatives of the vector representations in the quotient $\Rep U_{q=-\exp \left( -\frac{2 \pi i}{10}\right)}(\SO{4}) // V_{3e_1}$, via level-rank duality. We can replace $\SO{4}$ here with $\SL{2} \directSum \SL{2}$. The object $V_{3e_1}$ becomes $V_3 \boxtimes V_3$. The circles above indicate the vector representations, labelled by the corresponding objects of $\frac{1}{2}\mathcal{D}_6$. Next we can apply Lemma \ref{lem:KM-SO4} to realize $\frac{1}{2} \mathcal{D}_6$ as the modularization of $\uRep{U_{s = \exp\left(\frac{7}{10} 2 \pi i\right)}(\SL{2} \oplus \SL{2})}$. In this modularization, we quotient out the four corner vertices. Note that $P$ is sent to $V_1 \boxtimes V_0$, and in particular the knot invariant coming from $P$ recovers a specialization of the Jones polynomial.
}
\label{fig:SO4}
\end{figure}
\begin{rem}
Theorem \ref{thm:identities-3} is now an immediate corollary.
\end{rem}

\begin{rem}
The coincidence of Dynkin diagrams $D_2 = A_1 \times A_1$ also implies that the $D_2$ specialization of the Dubrovnik polynomial is equal to the square of the Jones polynomial:
$$\Dubrovnik(K)(q^3,q-q^{-1}) = J(K)(q)^2.$$
This was proved by Lickorish \cite[Theorem 3]{MR966143}, without using quantum groups.
\end{rem}

\begin{cor}
Looking at the object $\JW{2} \in \frac{1}{2} \mathcal{D}_6$, we have
\begin{equation*}
\restrict{\J{\SL{2}}{(2)}(K)}{q=\exp(\frac{2\pi i}{20})}  = \restrict{\J{\SL{2}}{(1)}(K)^2}{q=\exp(-\frac{2\pi i}{10})}.
\end{equation*}
This identity is closely related to Tutte's golden identity, c.f. \cite{MR2469528}.
\end{cor}

\begin{proof}[Proof of Theorem \ref{thm:D8-coincidence}]
We want to construct an equivalence
\begin{align*}
\frac{1}{2} \mathcal{D}_{8} & \cong \vRep U_{q=-\exp \left(-\frac{2 \pi i}{14} \right)}(\SO{6}) // (V_{3e_1}) \\
			   & \cong \uRep{U_{s=\exp \left(\frac{5}{14} 2 \pi i\right)}(\SL{4})}^{modularize},
\end{align*}
sending $P$ to a spinor representation of $\SO{6}$ and to $V_{(100)}$, the standard representation of $\SL{4}$.

The first step is the $k=6$ special case of Theorem \ref{thm:SO3-duality} on level-rank duality. The second step uses the coincidence of Dynkin diagrams $D_3 = A_3$ to obtain
$$\vRep U_{q=-\exp \left(-\frac{2 \pi i}{14} \right)}(\SO{6} // (V_{3e_1}) \cong
	\vRep U_{q=-\exp \left(-\frac{2 \pi i}{14} \right)}(\SL{4}) // V_{(030)}$$
after which Lemma \ref{lem:KM-SO6} gives the desired result.
For more details see Figure \ref{fig:SO6}.
\end{proof}
\begin{rem}
Theorem \ref{thm:identities-4} is now an immediate corollary.
\end{rem}

\begin{figure}[!pht]
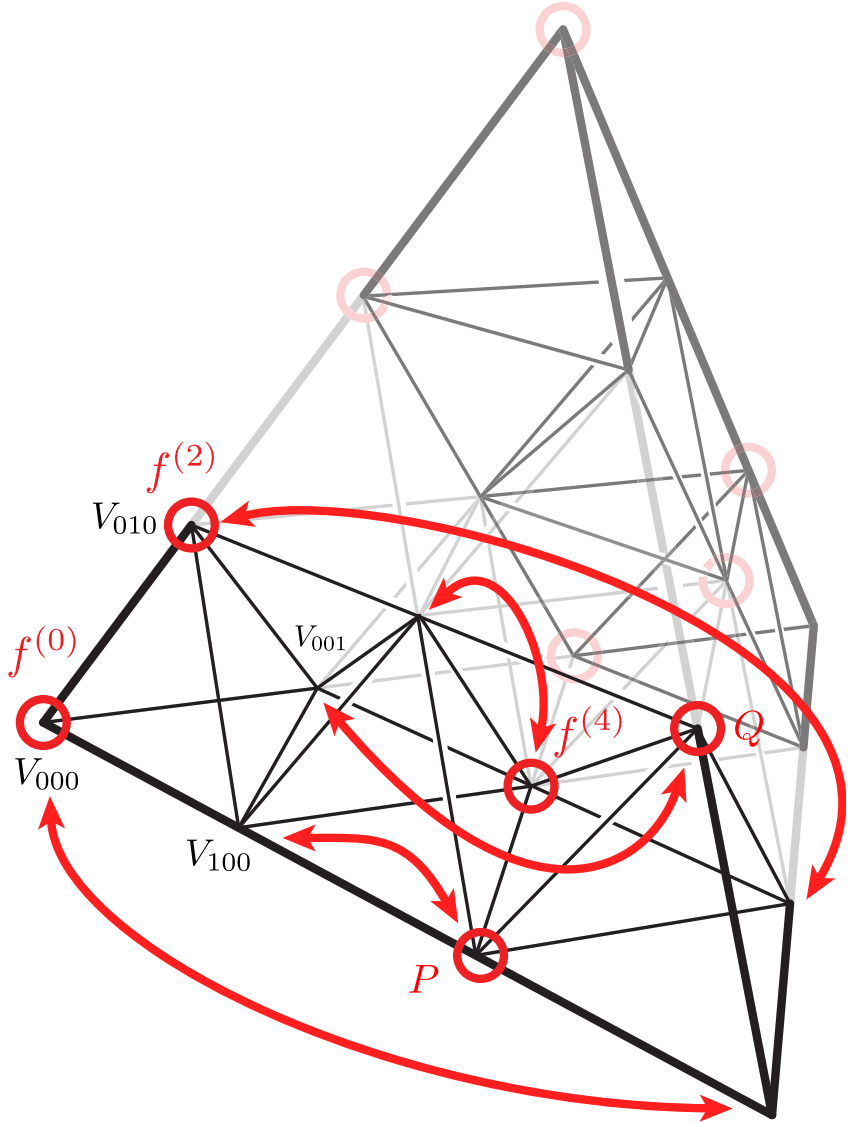
 
\begin{equation*}
\mathfig{0.8}{coincidences/sl4-modularisation-domain}
\end{equation*}

\caption{
We can realise $\frac{1}{2} \mathcal{D}_{8}$ as the vector representations in the $2$-fold quotient $\Rep U_{q=-\exp \left(-\frac{2 \pi i}{14} \right)}(\SL{4}) // (V_{(030)})$, via level-rank duality and the $A_3 = D_3$ coincidence of Dynkin diagrams. The figure shows a fundamental domain for the $2$-fold quotient. The objects of $\frac{1}{2} \mathcal{D}_{8}$ are shown circled (with fainter circles in the other domain showing their other representatives). Now we can apply Lemma \ref{lem:KM-SO6}, and instead identify these vector representations with representations in the $4$-fold quotient $\uRep{U_{s=\exp \left(2 \pi i \frac{5}{14}\right)}(\SL{4})}^{modularize}$ of the unimodal representation theory of $\SL{4}$, at a particular choice of $s$. These identifications are shown as arrows. Note that $P$ is sent to $V_{(100)}$, the standard representation of $\SL{4}$. In particular, the knot invariant coming from $P$ matches up with a specialization of the HOMFLYPT polynomial.
}
\label{fig:SO6}
\end{figure}

\begin{proof}[Proof of Theorem \ref{thm:D10-coincidence}]
We want to show that $\frac{1}{2} \mathcal{D}_{10}$ has an order $3$ automorphism: $$\baselinetrick
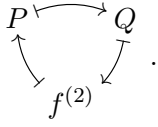
{\xymatrix@C-8mm@R-3mm{P \ar@{|->}@/^/[rr] & & Q \ar@{|->}@/^/[dl] \\ & f^{(2)} \ar@{|->}@/^/[ul]&}}.$$

\newcommand{\isoto}{\xrightarrow{\iso}}
Again, we first apply the $k=8$ special case of level-rank duality (Theorem \ref{thm:SO3-duality})  to see there is a functor $$\mathcal{L} :\frac{1}{2}\mathcal{D}_{10} \isoto \vRep U_{q=-\exp\left(-\frac{2\pi i}{18}\right)}(\SO{8}) // V_{(3000)}$$ with $\JW{2}$ corresponding to $V_{(1000)}$ and $P$ to $V_{(0002)}$.
In an exactly analogous manner as in Lemmas \ref{lem:KM-SO4} and \ref{lem:KM-SO6}, we can identify this two-fold quotient of the vector representations of $\SO{8}$ with a four-fold quotient of all the representations in the Weyl alcove. That is, there is a functor
\begin{multline*} \mathcal{K} : \vRep U_{q=-\exp\left(-\frac{2\pi i}{18}\right)}(\SO{8}) // V_{(3000)} \isoto \\
    \uRep U_{q=-\exp\left(-\frac{2\pi i}{18}\right)}(\SO{8}) // (V_{(3000)}, V_{(0030)}, V_{(0003)}).
\end{multline*}

The triality automorphism of the Dynkin diagram $D_4$ gives an automorphism $T$ of this category.  A direct computation shows that $T$ induces a cyclic permutation of  $P$, $Q$, and $\JW{2}$ in $\frac{1}{2} \mathcal{D}_{10}$.
For example,
\begin{align*}
\mathcal{L}^{-1}(\mathcal{K}^{-1}(T(\mathcal{K}(\mathcal{L}(\JW{2}))))) & = \mathcal{L}^{-1}(\mathcal{K}^{-1}(T(V_{(1000)}))) \\
	        & = \mathcal{L}^{-1}(\mathcal{K}^{-1}(V_{(0001)})) \\
	        & = \mathcal{L}^{-1}(\mathcal{K}^{-1}(V_{(0001)} \tensor V_{(0003)})) \\
	        & = \mathcal{L}^{-1}(\mathcal{K}^{-1}(V_{(0002)})) \\
	        & = \mathcal{L}^{-1}(V_{(0002)}) \\
	        & = P
\end{align*}

See Figures \ref{fig:SO8} and \ref{fig:SO8-4fold} for more details. It may be that this automorphism of $\frac{1}{2}\mathcal{D}_{10}$ is related to the exceptional modular invariant associated to $\SL{2}$ at level $16$ described in \cite{MR1105431}.
\end{proof}

\begin{figure}[!pht]
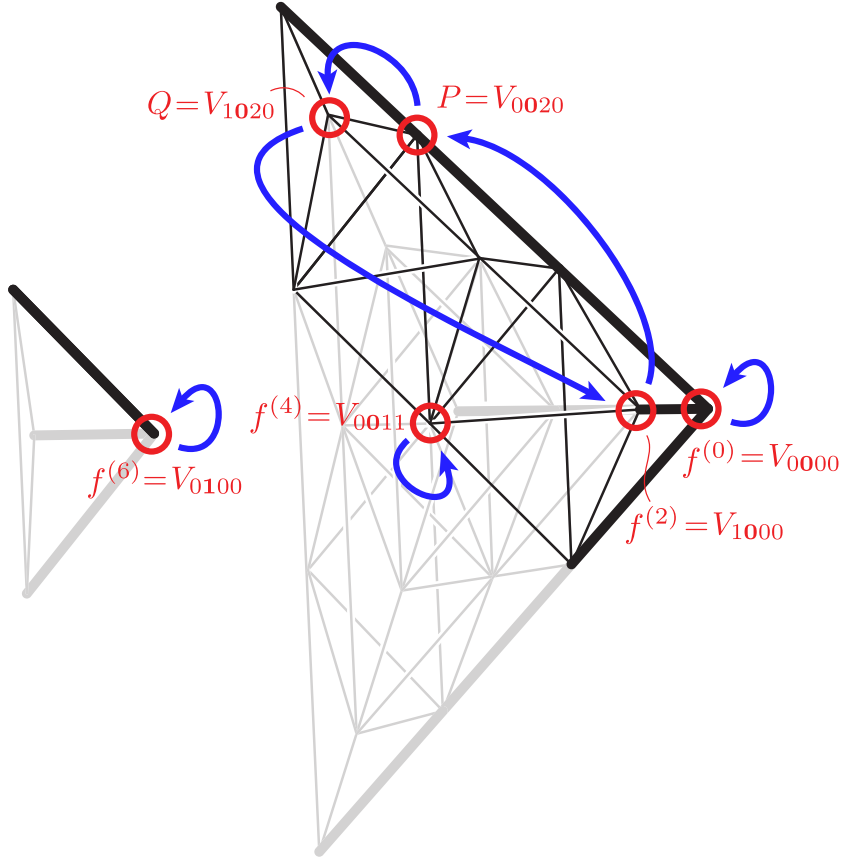

\begin{equation*}
\mathfig{0.8}{coincidences/SO8-chamber}
\end{equation*}
\caption{ We can realise $\frac{1}{2} \mathcal{D}_{10}$ as the vector representations in the $2$-fold quotient $\vRep U_{q=-\exp \left(-\frac{2 \pi i}{18} \right)}(\SO{8}) // (V_{(3000)})$, via level-rank duality.
The Weyl alcove for $\SO{8}$ at $q=-\exp\left(-\frac{2\pi i}{18}\right)$ consists of those $V_{(abcd)}$ such that $a+2b+c+d \leq 3$.  In particular, $b=0$ or $b=1$.  So we draw this alcove as two tetrahedra, the $V_{\star 0 \star \star}$ tetrahedron, and the $V_{\star1 \star \star}$ tetrahedron.  The vector representations are those $V_{(abcd)}$ with $c+d$ even.
We show a fundamental domain for the modularization involution $\tensor V_{(3000)}$, which acts on the $V_{\star0 \star \star}$ tetrahedron by $\pi$ rotation about the line joining $\frac{3}{2}000$ and $00\frac{3}{2}\frac{3}{2}$ and on the $V_{\star1 \star \star}$ tetrahedron by $\pi$ rotation about the line joining $\frac{1}{2}100$ and $01\frac{1}{2}\frac{1}{2}$.
The tensor category of $\frac{1}{2}\mathcal{D}_{10}$ is equivalent to the tensor subcategory of this modularization consisting of images of vector representations, with the equivalence sending $\JW{0} \mapsto V_{(0000)}$, $\JW{2} \mapsto V_{(1000)}$, $\JW{4} \mapsto V_{(0011)}$, $\JW{6} \mapsto V_{(0100)}$,  $P \mapsto V_{(0020)}$ and $Q \mapsto V_{(1020)}$.
The blue arrows shows the action of the triality automorphism $V_{(abcd)} \mapsto V_{(cbda)}$ for $\SO{8}$ on the image of $\frac{1}{2}\mathcal{D}_{10}$. This action is computed via the equivalence with the $4$-fold quotient of all representations, shown in Figure \ref{fig:SO8-4fold}. Notice that under this automorphism $P$ is sent to the standard representation.  In particular, the knot invariant coming from $P$ matches up with a specialization of the Kauffman polynomial.
} \label{fig:SO8}
\end{figure}

\begin{figure}[!htb]
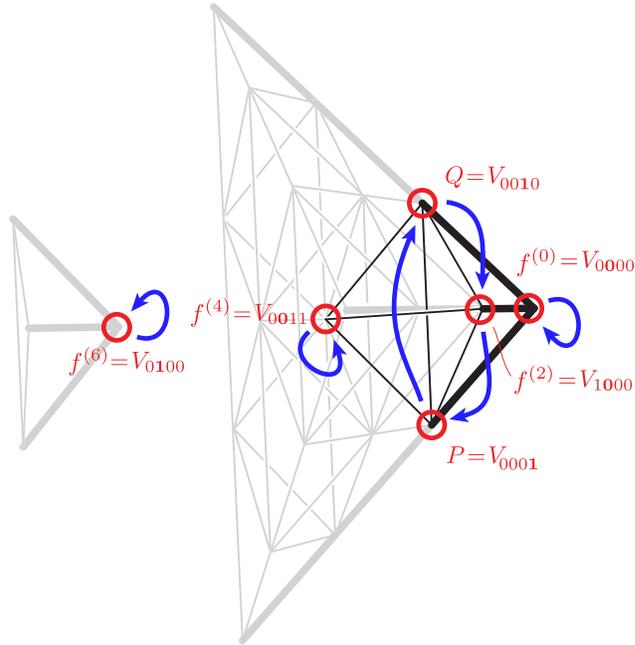

\begin{equation*}
\mathfig{0.6}{coincidences/SO8-4fold-quotient}
\end{equation*}
\caption{The action of the $D_4$ triality automorphism on the four-fold quotient
$\Rep U_{q=-\exp \left(-\frac{2 \pi i}{18} \right)}(\SO{8}) // (V_{(3000)}, V_{(0030)}, V_{(0003)}).$}
\label{fig:SO8-4fold}
\end{figure}




\newcommand{\urlprefix}{}
\forGTART{\bibliographystyle{gtart}}
\forQT{\bibliographystyle{qt-journal}}
\bibliography{bibliography/bibliography}

This paper is available online at 
\href{http://arxiv.org/abs/1003.0022}{\tt arXiv:\nolinkurl{1003.0022}},
and at
\url{http://tqft.net/identities}.

\end{document}